\documentclass[10pt]{article}
\usepackage[margin = 1in]{geometry}

\usepackage{times}

\usepackage[utf8]{inputenc}
\usepackage{amsmath, amssymb}
\usepackage{amsthm}
\usepackage{bbm}
\usepackage{mathbbol}
\usepackage{float}
\usepackage{dirtytalk}
\usepackage{graphicx}
\DeclareMathOperator*{\argminA}{arg\,min}
\usepackage[english]{babel}

\usepackage{booktabs}
\usepackage{empheq}
\usepackage{mdframed}
\usepackage{pifont}
\usepackage{enumitem}
\usepackage{wrapfig}
\usepackage{empheq}
\usepackage{bm}

\usepackage{cancel}
\usepackage{array}
\usepackage{comment}
\newcommand*{\vertbar}{\rule[-1ex]{0.5pt}{2.5ex}}
\newcommand*{\horzbar}{\rule[.5ex]{2.5ex}{0.5pt}}

\usepackage{tikz}
\usetikzlibrary{shapes}
\usetikzlibrary{decorations.markings}
\usetikzlibrary{plotmarks}
\usetikzlibrary{patterns}

\usepackage{color, colortbl}
\definecolor{mydarkblue}{rgb}{0,0.08,0.45}
\usepackage[colorlinks=true,
    linkcolor=mydarkblue,
    citecolor=mydarkblue,
    filecolor=mydarkblue,
    urlcolor=mydarkblue,
    pdfview=FitH]{hyperref}
    
\definecolor{myteal}{RGB}{27,158,119}
\definecolor{myorange}{RGB}{217,95,2}
\definecolor{myred}{RGB}{231,41,138}
\definecolor{mypurple}{RGB}{152,78,163}
\definecolor{myblue}{RGB}{55,126,184}
\definecolor{mygreen}{RGB}{0,100,0}

\newtheorem{proposition}{Proposition}[section]
\newtheorem{lemma}{Lemma}[section]
\newtheorem{theorem}{Theorem}[section]

\newtheorem{corollary}{Corollary}[section]

\newtheorem{assumption}{Assumption}[section]

\DeclareMathOperator{\E}{\mathbb{E}}

\usepackage{hyperref}
\usepackage[sort]{natbib}

\title{SGD in the Large:\\ Average-case Analysis, Asymptotics, and Stepsize Criticality} 

 \author{Courtney Paquette\thanks{Google Research, Brain Team} \footnotemark[2]
 \and Kiwon Lee\footnotemark[2]
 \and Fabian Pedregosa\footnotemark[1]
  \and Elliot Paquette\thanks{Department of Mathematics and Statistics, McGill University, Montreal, QC, Canada, H3A 0B9; CP is a CIFAR AI chair; \url{https://cypaquette.github.io/}. Research by EP was supported by a Discovery Grant from the
Natural Science and Engineering Research Council (NSERC) of Canada.; \url{https://elliotpaquette.github.io/}.}
}
\date{\today}

\def\N{\mathbb{N}}

\def\aa{{\boldsymbol a}}
\def\bb{{\boldsymbol b}}

\def\xx{{\boldsymbol x}}

\def\XX{{\boldsymbol X}}
\def\YY{{\boldsymbol Y}}

\def\aa{{\boldsymbol a}}
\def\bb{{\boldsymbol b}}

\def\WW{{\boldsymbol W}}

\def\II{{\boldsymbol I}}
\def\yy{{\boldsymbol y}}
\def\vv{{\boldsymbol v}}
\def\uu{{\boldsymbol u}}

\def\AA{{\boldsymbol A}}

\def\MM{{\boldsymbol M}}
\def\DD{{\boldsymbol D}}
\def\OO{{\boldsymbol O}}
\def\PP{{\boldsymbol P}}

\def\UU{{\boldsymbol U}}
\def\VV{{\boldsymbol V}}
\def\HH{{\boldsymbol H}}

\def\SSigma{{\boldsymbol \Sigma}}

\def\nnu{{\boldsymbol \nu}}

\def\eeta{{\boldsymbol \eta}}

\def\g{{g}}

\def\dif{\mathop{}\!\mathrm{d}}

\def\MP{\mu_{\mathrm{MP}}}
\def\RR{{\mathbb R}}
\def\EE{{\mathbb E}\,}

\def\defas{\stackrel{\text{def}}{=}}

\DeclareMathOperator*{\Var}{Var}

\DeclareMathOperator*{\diag}{\mathbf{diag}}

\DeclareDocumentCommand{\Prto} {o} {
  \IfNoValueTF {#1}
  {\overset{\Pr}{\longrightarrow}}
  { \xrightarrow[ #1 \to \infty]{\Pr }}
}

\DeclareDocumentCommand{\Asto} {o} {
  \IfNoValueTF {#1}
  {\overset{\text{\rm a.s.}}{\longrightarrow}}
  { \xrightarrow[ #1 \to \infty]{\text{\rm a.s.} }}
}

\DeclareDocumentCommand{\law} {o} {
  \IfNoValueTF {#1}
  {\overset{\text{law}}{=}}
  { \xrightarrow[ #1 \to \infty]{\Pr }}
}

\begin{document}

\maketitle
\begin{abstract}
   We propose a new framework, inspired by random matrix theory, for analyzing the dynamics of stochastic gradient descent (SGD) when both number of samples and dimensions are large. This framework applies to any fixed stepsize and the finite sum setting. Using this new framework, we show that the dynamics of SGD on a least squares problem with random data become deterministic in the large sample and dimensional limit. 
   Furthermore, the limiting dynamics are governed by a Volterra integral equation. This model predicts that SGD undergoes a phase transition at an explicitly given critical stepsize that ultimately affects its convergence rate, which we also verify experimentally. Finally, when input data is isotropic, we provide explicit expressions for the dynamics and average-case convergence rates (\textit{i.e.,} the complexity of an algorithm averaged over all possible inputs). These rates show significant improvement over the worst-case complexities.
\end{abstract}

\section{Introduction}

Stochastic gradient descent (SGD) \cite{robbins1951} is one of the most popular and important stochastic optimization methods for use in large-scale problems. There are well-established worst-case convergence rates, but SGD lacks a detailed theory that encompasses both its successes and its empirically observed peculiarities. For example, the solutions to which SGD converges have qualitative differences that seemingly depend on how SGD is tuned \citep{jastrzkebski2017three,keskar2016on}.  Furthermore, the dependence of the runtime of SGD on its stepsize is complicated, and stepsize selection is an active area of research \citep{schaul2013no,vaswani2019painless,mahsereci2017probablistic,bollapragada2018adaptive,friedlander2012hybrid}. Beyond the confines of SGD, the behavior of other stochastic optimization algorithms is even more poorly understood \citep{sutskever2013on}. Because of these challenges, \emph{making good quantitative predictions for the dynamics of stochastic algorithms remains a difficult, broad and deep problem}. 

A prolific technique for analyzing optimizations methods, both stochastic and deterministic, is the stochastic differential equations (SDE) paradigm \citep{li2017stochastic, mandt2016variational,jastrzkebski2017three,su2016differential,Kushner,ljung1977,hu2017on, chaudhari2018stochastic}. These SDEs relate to the dynamics of the optimization method by taking the limit when the stepsize goes to zero, so that the trajectory of the objective function over the lifetime of the algorithm converges to the solution of an SDE. Naturally, in practice, the stepsize is taken as large as possible, which limits the predictive power of the SDE method.

A related popular paradigm for analyzing the behavior of SGD is the noisy gradient model.  Often used in conjunction with the SDE approach, in this model, one supposes that the stochastic gradient estimators in SGD are the true gradient plus some independent noise (typically assumed to be Gaussian with some covariance structure) \citep{li2017stochastic, mandt2016variational,jastrzkebski2017three, mert2019tail} or more generally the gradient estimators
are independent with a common distribution, see for \textit{e.g.}
\cite{ward2020}. The latter is equivalent to the ``streaming setting''
\cite{jain2017} or the ``one-pass'' assumption on the data
\citep{mert2020heavy-tail}.

\begin{wrapfigure}[18]{r}{0.45\textwidth}
     \includegraphics[width = 1.0\linewidth]{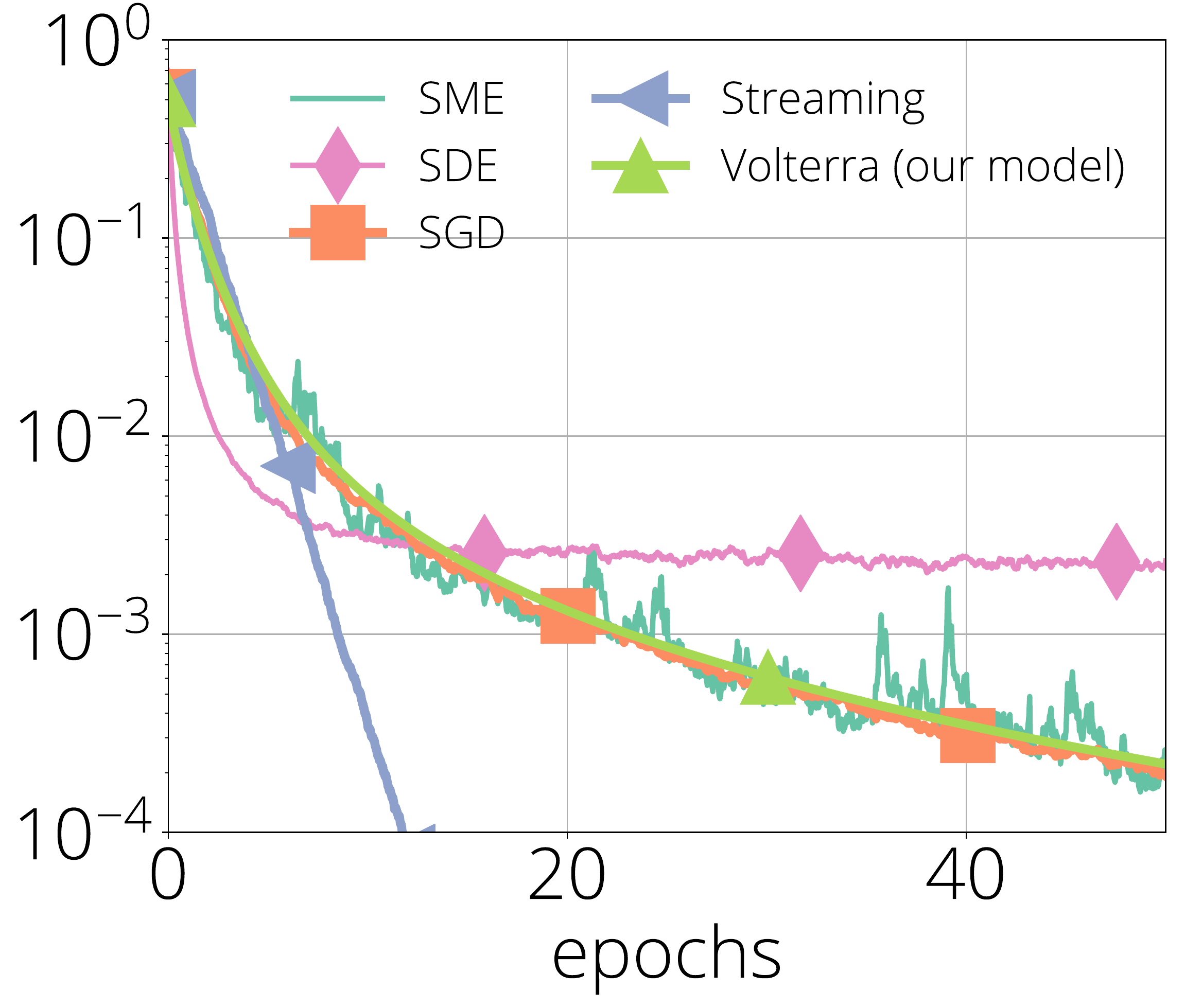}
     \vspace{-0.5cm}
     \caption{The proposed Volterra equation model \textbf{accurately} tracks SGD on a random least-squares problem for any choice of stepsize. Other models introduce biases that substantially impact model fidelity.}
    \label{fig:volterra_sde_comparison}
  \end{wrapfigure}
Here one generates a new sample at each iteration and does not reuse any past data. In practice, SGD is typically implemented on a finite dataset with multiple passes over the data. Such modeling assumptions on the stochastic gradient estimators can not capture the full dynamics of SGD (see Figure \ref{fig:volterra_sde_comparison}). 

We offer a new alternative, inspired by the phenomenology of random matrix theory.
We prove that SGD with a \emph{fixed} stepsize $\gamma$ has deterministic dynamics, when run on the least squares problem with \emph{high--dimensional} random data, and, we analyze these dynamics to provide stepsize selection and convergence properties (see Figure~\ref{fig:volterra_sde_comparison} for a comparison). We neither impose assumptions on the gradient estimators nor take the stepsize to 0 and we work in the \textit{non-streaming} or \textit{finite sum} setting (a.k.a. incremental gradient). These deterministic dynamics are governed by a Volterra integral equation, that is, the function values converge to the solution $\psi_0$ of
\begin{equation} \begin{gathered} 
\psi_0(t) = z(t)  + r \gamma^2 \int_0^t \!  h_2(t-s) \psi_0(s)\,\dif s,\\
\quad
\text{where}
\quad
h_2(t) = \int_0^\infty  \! \! x^2 e^{-2\gamma tx} \,\dif \mu(x).
\label{eq:main_volterra_initial}
\end{gathered} \end{equation}

Here, $r$ is the ratio of the number of parameters to sample size, and $\mu$ is the distribution of the eigenvalues of the Hessian's objective. The function $z$ is an explicit forcing function, which has dependence on all parts of the problem, including the initialization $\xx_0$ and the target $\bb.$ See Theorem \ref{thm: concentration_main} for the precise statement. The value of the stepsize $\gamma$ can be as large as the convergence threshold which we explicitly provide. This Volterra equation \eqref{eq:main_volterra_initial} has rich behavior; the asymptotic suboptimality of $\psi_0$ has a discontinuity in $\gamma$ at a critical stepsize (see Theorem~\ref{thm:main_critical_stepsize}).


\paragraph{Notation.} We denote vectors in lowercase boldface ($\xx$) and matrices in uppercase boldface ($\HH$). 
 A sequence of random variables $\{y_d\}_{d =0}^\infty$ converges in probability to $y$, indicated by $y_d \Prto[d] y$, if for any $\varepsilon > 0$, $\displaystyle \lim_{d \to \infty} \Pr(|y_d-y| > \varepsilon) = 0$. 
 Probability measures are denoted by $\mu$ and their densities by $\dif \mu$. 
We say a sequence of random measures $\mu_d$ converges to $\mu$ weakly in probability if for any bounded continuous function $f$, we have $\int f \dif \mu_d \to \int f \dif \mu$ in probability. 
For two random variables $x$ and $y$ we write $x \law y$ to mean they have the same distribution.

\subsection{Problem setting.} \label{sec:main_problem_setting} We consider the least--squares problem when the number of samples ($n$) and features ($d$) are large:
 \begin{equation}\label{eq:lsq}
    \argminA_{\xx\in\mathbb{R}^d} \Big\{ f(\xx) \overset{\mathrm{def}}{=} \frac{1}{2n}\sum_{i=1}^n (\aa_i \xx - b_i)^2\Big\}, \quad \text{with $\bb \defas \AA \widetilde{\xx} + \sqrt{n} \, \eeta$,}
\end{equation}
where $\AA \in \mathbb{R}^{n \times d}$ is a random data matrix whose $i$-th row is denoted by $\aa_i \in \mathbb{R}^{d \times 1}$, $\widetilde{\xx} \in \mathbb{R}^d$ is the signal vector, and $\eeta \in \mathbb{R}^n$ is a source of noise.  The target $\bb = \AA \widetilde{\xx} + \sqrt{n} \, \eeta$ comes from a generative model corrupted by noise. 

We apply SGD (incremental gradient) to the finite sum, quadratic problem above.  On the $k$-th iteration it selects a uniformly random subset
$B_k \subset \{1,2,\cdots,n\},$ of batch-size $\beta$ and makes the updates
\begin{equation}\label{eq:oursgd}
    \xx_{k+1} 
    = \xx_k - \frac{\gamma}{n}\sum_{i \in B_k}\nabla f_{i}(\xx_k) 
    = \xx_k - \frac{\gamma}{n}\AA^T \PP_k (\AA\xx_k - \bb), \quad \text{where} \quad \PP_k \defas \sum_{i \in B_k} \mathbb{e}_{i}\mathbb{e}_{i}^T.
\end{equation}
 Here $\PP_k$ is a random orthogonal projection matrix with $\mathbb{e}_i$ the $i$-th standard basis vector, $\beta \in \N$ is a batch-size parameter, which we will allow to depend on $n$, $\gamma > 0$ is a stepsize parameter, and the function $f_i$ is the $i$-th element of the sum in \eqref{eq:lsq}. Typical stepsizes for SGD (see \textit{e.g} \cite[Thm 4.6]{bottou2018optimization}) include the second moment of the stochastic gradients, which under our problem setting grows like $n$. This explains the dependency on $n$ in \eqref{eq:oursgd}. We remark that $\beta$ can equal $1$ in which case \eqref{eq:oursgd} reduces to the simple SGD setting.

To perform our analysis we make the following explicit assumptions on the signal $\widetilde{\xx}$, the noise $\eeta,$ and the data matrix $\AA.$
\begin{assumption}[Initialization, signal, and noise] \label{assumption: Vector} The initial vector $\xx_0 \in \RR^d$, the signal $\widetilde{\xx} \in \RR^d$, and noise vector $\eeta \in \RR^n$ satisfy the following conditions:
\begin{enumerate}[leftmargin=*]
    \item The difference $\xx_0-\widetilde{\xx}$ is any deterministic vector such that $\|\xx_0 - \widetilde{\xx}\|^2_2 = R.$
    \item The entries of the noise vector $\eeta$ are i.i.d.\ random variables that verify for some constant $\widetilde{R} > 0$
    \begin{equation}\label{eq:eeeeta}
    \EE[\eeta] = \bm{0}, \quad \EE[\|\eeta\|_2^2] = \widetilde{R}, \quad \text{and} \quad \EE[\|\eeta\|_\infty^p] = \mathcal{O}( n^{\epsilon-p/2}) \quad \text{for any } \epsilon, p >0.
    \end{equation}
\end{enumerate}
\end{assumption}
\noindent Any subexponential law for the entries of $\eeta$ (say, uniform or Gaussian with variance $\widetilde{R}/n$) will satisfy \eqref{eq:eeeeta}.
The scalings of the vectors $\xx_0-\widetilde{\xx}$ and $\eeta$ arise as a result of preserving a constant signal-to-noise ratio in the generative model. Such generative models with this scaling have been used in numerous works \citep{mei2019generalization,hastie2019surprises,gerbelot2020asymptotic}. 

Next we state an assumption on the eigenvalue and eigenvector distribution of the data matrix $\AA$. We then review practical scenarios in which this is verified. 

\begin{assumption}[Data matrix] \label{assumption: spectral_density}
Let $\AA$ be a random $n \times d$ matrix such that the number of features, $d$, tends to infinity proportionally to the size of the data set, $n$, so that $\tfrac{d}{n} \to r \in (0, \infty)$. Let $\HH \defas \tfrac{1}{n} \AA^T \AA$ with eigenvalues $\lambda_1 \leq \ldots \leq \lambda_d$ and let $\delta_{\lambda_i}$ denote the Dirac delta with mass at $\lambda_i$. We make the following assumptions on the eigenvalue distribution of this matrix:
\begin{enumerate}[leftmargin=*]
    \item The eigenvalue distribution of $\HH$ converges to a deterministic limit $\mu$ with compact support.  Formally, the empirical spectral measure (ESM) satisfies
    \begin{equation} \label{eq:ESM_convergence}
    \mu_{\HH} = \frac{1}{d}\sum_{i=1}^d \delta_{\lambda_i} \to \mu \quad \text{weakly in probability\,.}
    \end{equation}
    \item The largest eigenvalue $\lambda_{\HH}^+$ of $\HH$ converges in probability to the largest element $\lambda^+$ in the support of $\mu$, i.e.
    \begin{equation} \label{eq:max_eigenvalue} \lambda_{\HH}^+ \Prto[d] \lambda^+. \,\end{equation}
    \item (Orthogonal invariance) Let $\UU \in \mathbb{R}^{d \times d}$ and $\OO \in \mathbb{R}^{n \times n}$ be orthogonal matrices. The matrix $\AA$ is orthogonally invariant in the sense that 
    \begin{equation} \label{eq:ortho_invariance}
    \AA \UU \law \AA \quad \text{and} \quad \OO \AA \law \AA 
    \end{equation}
\end{enumerate}
\end{assumption}

\begin{wrapfigure}[15]{r}{0.45\textwidth}
\vspace{-0.45cm}
 \centering
     \includegraphics[width = 0.9\linewidth]{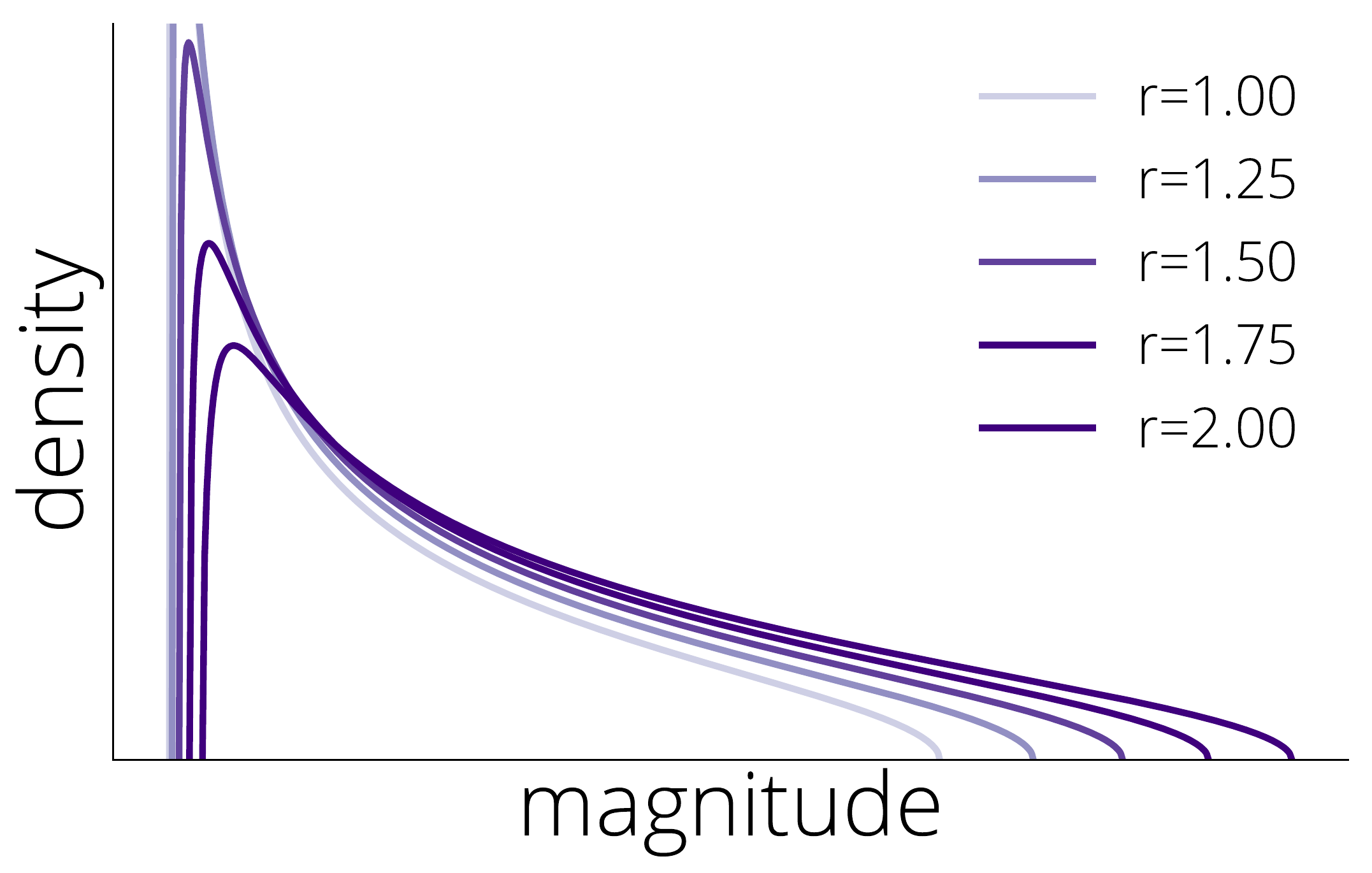} 
     
     \vspace{-0.3cm}
     
     \caption{The ESM of matrices $\tfrac{1}{n} \AA^T \AA$ with i.i.d. entries converges as $n, d \to \infty$ to the \emph{Marchenko-Pastur} distribution, shown here for different values of $r = d/n$.}
    \label{fig:MP}
\end{wrapfigure}

Assumption~\ref{assumption: spectral_density} characterizes the distribution of eigenvalues for the random matrix $\HH$ which approximately equals the distribution $\mu$. The ESM and its convergence to the limiting spectral distribution $\mu$ are well studied in random matrix theory, and for many random matrix ensembles the limiting spectral distribution is known. In machine learning literature, it has been shown that the spectrum of the Hessians of neural networks share many characteristics with the limiting spectral distributions found in classical random matrix theory \citep{dauphin2014identifying, papyan2018the, sagun2016eigenvalues, behrooz2019investigation,martin2018implicit}.

The last assumption, orthogonal invariance, is a rather strong condition as it implies that the singular vectors of $\AA$ are uniformly distributed on the sphere.  The classic example of a matrix which satisfies this property are matrices whose entries are generated from standard Gaussians. There is however, a large body of literature \citep{KnowlesYin,CES2020} showing that other classes of large dimensional random matrices behave like orthogonally invariant ensembles; weakening the orthogonal invariance assumption is an interesting future direction of research which is beyond the scope of this paper. Moreover our numerical simulations suggest that \eqref{eq:ortho_invariance} is unnecessary as our Volterra equation holds for ensembles without this orthogonal invariance property (see one-hidden layer networks in Section~\ref{sec: numerical_simulation}). For a more thorough review of random matrix theory see \citep{bai2010spectral, tao2012topics}.

\paragraph{Examples of data distributions.} \label{sec: data_generate}
In this section we review examples of data-generating distributions that verify Assumption~\ref{assumption: spectral_density}.

\paragraph{Example 1: Isotropic features with Gaussian entries.}  The first model we consider has entries of $\AA$ which are i.i.d. standard Gaussian random variables, that is, $A_{ij} \sim N(0, 1)$ for all $i,j$. This ensemble has a rich history in random matrix theory. When the number of features $d$ tends to infinity proportionally to the size of the data set $n$, $\tfrac{d}{n} \to r \in (0, \infty)$, the seminal work of \citet{marvcenko1967distribution} showed that the spectrum of $\HH = \tfrac{1}{n} \AA^T \AA$ asymptotically approaches a deterministic measure $\MP$, verifying Assumption~\ref{assumption: spectral_density}. This measure, $\MP$, is given by the Marchenko-Pastur law:
\begin{equation} \begin{gathered} \label{eq:MP} \dif \MP(\lambda) \defas \delta_0(\lambda) \max\{1-\tfrac{1}{r}, 0\} + \frac{\sqrt{(\lambda-\lambda^-)(\lambda^+-\lambda)}}{2 \pi \lambda r} 1_{[\lambda^-, \lambda^+]}\,,\\
\text{where} \qquad \lambda^- \defas (1 - \sqrt{r})^2 \quad \text{and} \quad \lambda^+ \defas (1+ \sqrt{r})^2\,.
\end{gathered} \end{equation}
\paragraph{Example 2: Planted spectrum}
One may wonder if there are limits to what singular value distributions can appear for orthogonally invariant random matrices, but as it turns out, any singular value distribution is attainable.  Suppose that 
\begin{equation} \label{eq:general_isotropic}
    \AA = \UU \SSigma \VV^T,  
\end{equation}
where $\VV \in \RR^{d \times d}$ and $\UU \in \RR^{n \times n}$ are random matrices, uniformly chosen from the orthogonal group and $\SSigma \in \mathbb{R}^{n \times d}$ is any deterministic matrix such that the squared singular values of $\SSigma$ have an empirical distribution that converges to a desired limit $\mu$. Then $\AA$ is orthogonally invariant.  As in the previous case, we assume that the dimensions of the matrix $\AA$ grow at a comparable rate given by $\tfrac{d}{n} \to r \in (0, \infty)$. 
Constructions like this appear in neural networks initialized with random orthogonal weight matrices, and they produce exotic singular value distributions \cite[Figure 7]{saxe2013exact}.

\paragraph{Example 3: Linear neural networks.} 
This model encompasses linear neural networks with a squared loss, where the $m$ layers have random weights ($\WW_i$ with $i = 1, \hdots, m$) and the final layer's weights are given by the regression coefficient $\xx$. The entries of these random weight matrices are sampled i.i.d.\ from a standard Gaussian. The optimization problem in \eqref{eq:lsq} becomes

\begin{equation}
    \min_{\xx} \bigg \{ f(\xx) = \frac{1}{2n} \| \AA \xx - \bb \|_2^2 \bigg \}, \quad \text{where $\AA = \WW_1 \WW_2 \WW_3 \cdots \WW_m$.}
\end{equation}
It is known that products of Gaussian matrices satisfy \eqref{eq:ortho_invariance} with a limiting spectral measure in the large $n$ limit and fixed depth given by the Fuss-Catalan law~\cite{Alexeev, LiuSongWang}.

\subsection{Main contributions}

\paragraph{A new paradigm for analyzing the dynamics of SGD.}  
We propose a framework for the analysis of SGD that exploits the fact that when increasing the problem size (\textit{i.e.} $n$ and $d$ large), statistics that are driven by the full population converge to deterministic processes; the spirit of which is behind law of large numbers and concentration of measure. A practical outcome of this framework is a new expression for the function values of SGD as a Volterra equation:


\begin{theorem}[Concentration of SGD] \label{thm: concentration_main} Suppose the stepsize satisfies $\gamma < \frac{2}{r} \big ( \int_0^\infty x \dif \mu(x) \big )^{-1}$ and the batchsize satisfies $\beta(n) \leq n^{1/5-\delta}$ for some $\delta >0$. Under Assumptions~\ref{assumption: Vector} and \ref{assumption: spectral_density}, 
\begin{equation}
  \sup_{0 \le t \le T}  \big | f \big ( \xx_{ \lfloor \tfrac{n}{\beta} t \rfloor } \big ) -  \psi_0(t) \big | \Prto[n] 0\,,
\end{equation}
where the function $\psi_0$ is the solution to the Volterra equation
\begin{equation} \begin{gathered} \label{eq:Volterra_eq_main}
    \psi_0(t) = \tfrac{R}{2} h_1(t) + \tfrac{\widetilde{R}}{2} \big ( rh_0(t)+ (1-r) \big ) + \int_0^t \gamma^2 r h_2(t-s)\psi_0(s)\, \dif s,\\
    \quad \text{and} \quad h_k(t) = \int_0^\infty x^k e^{-2\gamma tx} \, \dif \mu(x)\,
    \quad \text{for all}\quad k \geq 0. 
\end{gathered}
\end{equation}
\end{theorem}

\begin{wrapfigure}[38]{r}{0.45\textwidth}
\vspace{-0.7cm}
 \includegraphics[width = 1\linewidth]{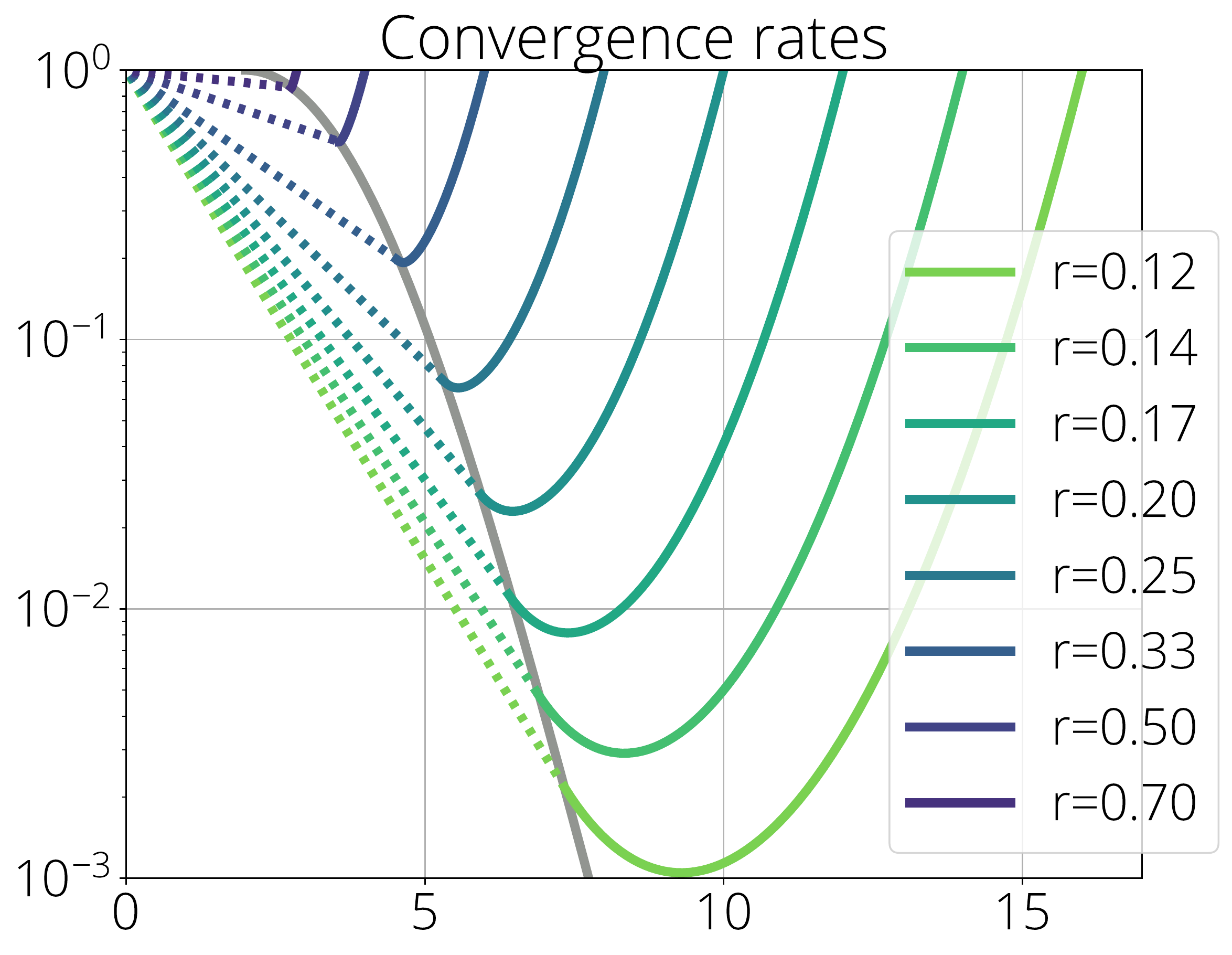} 
  \includegraphics[width = 1\linewidth]{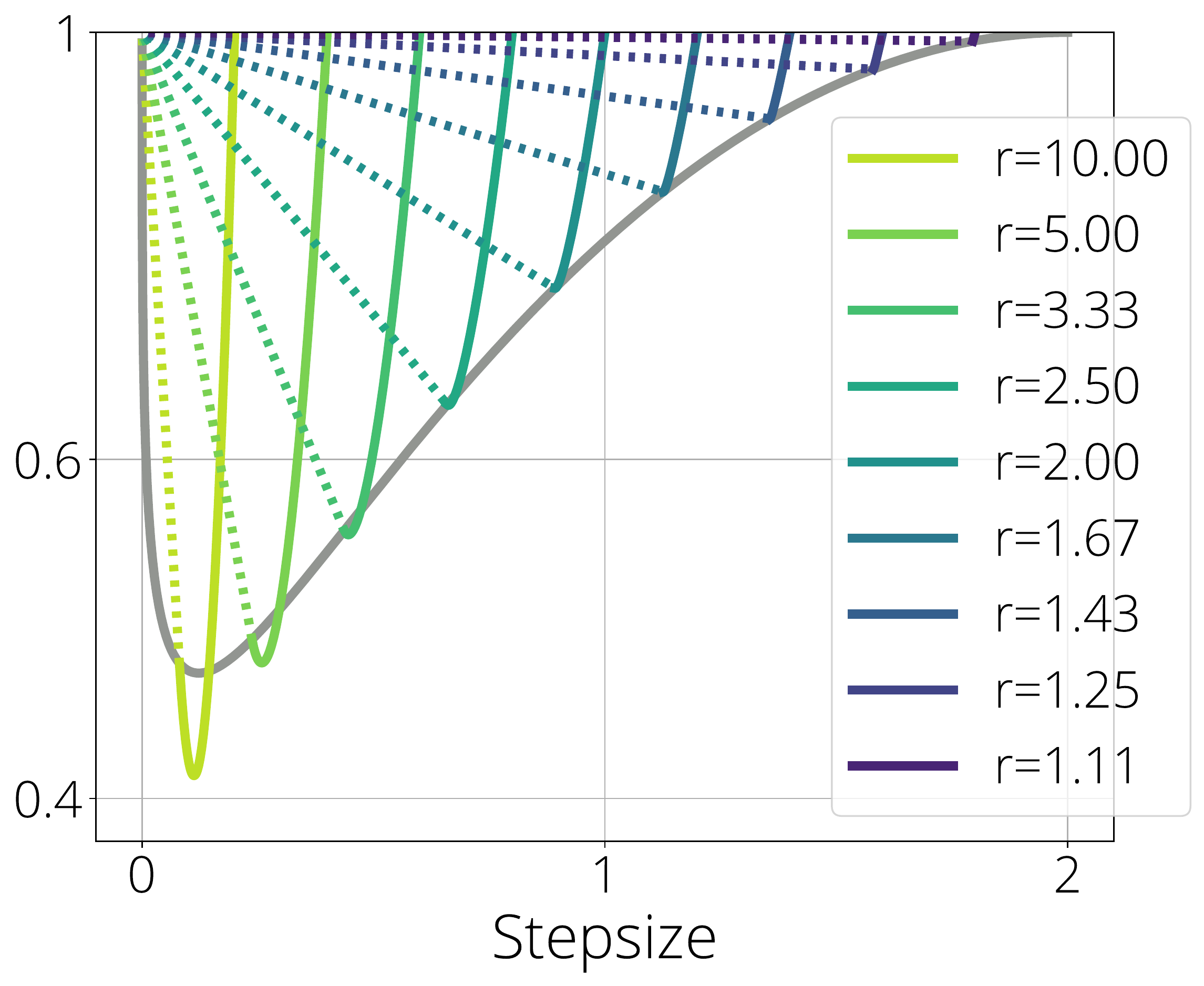} \vspace{-0.8cm}
     \caption{\textbf{Phase transition of the convergence rate} (y-axis) as a function of the stepsize (x-axis, $\gamma$) for the isotropic features model. Smaller stepsizes (dotted) yield convergence rates which depend linearly on $\gamma$ with a slope that is always frozen on $\lambda^-$. The convergence rate abruptly changes behavior once it hits the \textit{critical stepsize} (solid gray, $\gamma_*$), becoming a non-linear function of $\gamma$. The critical stepsize appears to be a good predictor for the optimal stepsize. In addition, the more over-parameterized the data matrix ($r \to 0$) is, the smaller the window of convergent stepsizes and as $\HH$ becomes ill-conditioned ($r \to 1)$, the linear rate degenerates and the high temperature phase disappears. }
    \label{fig:stepsizes}
\end{wrapfigure} 

 The expression highlights how the algorithm, stepsize, signal and noise levels interact with each other to produce different dynamics. For instance, our framework allows one to see the effect of the \textit{entire} spectrum of the data matrix on the dynamics. Also we note that the batch-size $\beta$ does not appear in the limiting Volterra equation. Numerical simulations in Section \ref{sec: numerical_simulation} confirm that $\psi$ accurately predicts the behavior of SGD. 


\paragraph{Phase transition of SGD dynamics and critical stepsize.} We prove a surprising dichotomy in the dynamics of SGD for a general measure: SGD undergoes a phase transition at a critical stepsize which we denote by $\gamma_*$ 
\begin{equation} \label{eq:main_critical_stepsize}
\gamma_* \defas \frac{1}{\tfrac{r}{2} \int_0^\infty \frac{x^2}{x-\lambda^-} \, \dif \mu(x)}\,.
\end{equation}
Starting at small stepsizes, we see that the linear rate of convergence for SGD \textit{freezes} on the smallest eigenvalue of $\HH,$
that is $f($\raisebox{0.1em}{$\xx_{ \lfloor \tfrac{n}{\beta} t \rfloor }$}$)$\vspace{-0.3em}
decreases like $e^{-2\gamma \lambda^- t}$. However when $\gamma$ passes the transition point $\gamma_*$, the dynamics of SGD have a more complicated dependency on the stepsize (in particular it is no longer log-linear in $\gamma$). This is strongly reminiscent of a \textit{freezing transition}, often seen in the free energies of random energy models (see \cite{Derrida}), with $\gamma$ playing the role of temperature. This is summarized in our second main contribution -- the asymptotic rates for SGD under a general spectral measure $\mu$ (see Appendix \ref{app: general_analysis_volterra}).
\begin{theorem}[Critical stepsize, asymptotic
  rates] \label{thm:main_critical_stepsize} Suppose $r \neq 1$ (i.e. strongly convex regime). 
For $\gamma_* < \gamma < \frac{2}{r} \big ( \int_0^\infty x \dif \mu(x) \big )^{-1}$, the value of $\lambda^*(\gamma)$ is given as the unique solution to
\begin{equation} \label{eq:lambda_*}
r \gamma^2 \int_0^\infty e^{2\gamma \lambda^* t}h_2(t) \, \dif t = 1
\,.\end{equation}
The function $\psi_0(t)$ satisfies that for some explicit constant $c(R, \widetilde{R}, \mu) > 0,$ 
\begin{equation}
    \psi_0(t) - \frac{\widetilde{R}}{2} \cdot \frac{r\mu( \{0\} ) + (1-r)}{1-\tfrac{\gamma r}{2} \big ( \int_0^\infty x \, \dif \mu(x) \big ) } \sim \frac{c}{\gamma} e^{-2\gamma t\lambda^*(\gamma) }.
\end{equation}
If in addition $\gamma_* > 0,$ and $\mu([\lambda^{-},\lambda^{-}+t])\sim c_\mu t^{\alpha}$
as $t \to 0$ 
then there is a constant $c(R,\widetilde{R},\gamma,\mu)>0$ so that
for\\
$0 < \gamma < \gamma_*$, 
\begin{equation}
    \psi_0(t) - \frac{\widetilde{R}}{2} \cdot \frac{r\mu( \{0\} ) + (1-r)}{1-\tfrac{\gamma r}{2} \big ( \int_0^\infty x \, \dif \mu(x) \big ) } 
    \sim \frac{c}{t^\alpha}e^{-2 \gamma t \lambda^-}.
\end{equation}
\end{theorem}
\noindent We also give rates for the case of $r=1$ in Thm \ref{thm:GVolterra_convex}.  See App.~\ref{app: general_analysis_volterra} for further discussion and the derivation.

\begin{table}[t!]
\renewcommand{\arraystretch}{2}
    \centering
    \begin{tabular}[]{cccc}
    \hline
         &\begin{minipage}{0.27\textwidth} \begin{center} \textbf{Strongly convex, $\bm{\gamma < \gamma_*}$} \end{center} \end{minipage} & \begin{minipage}{0.27\textwidth} \begin{center} \textbf{Strongly convex, $\bm{\gamma = \gamma_*}$} \end{center} \end{minipage} \\
   \hline
         Worst & \textcolor{purple}{$\text{exp}(-\gamma t \lambda^- + \tfrac{\gamma^2}{2} (\lambda^+)^2 t)$ }& \textcolor{purple}{$\text{exp}(-\gamma t \lambda^- + \tfrac{\gamma^2}{2} (\lambda^+)^2 t)$}\\
         \hline
         \textbf{Average} & \textcolor{purple}{$\text{\rm exp}(-2\gamma\lambda^- t)$} $\cdot$ \textcolor{teal}{$\frac{1}{t^{3/2}}$} &  \textcolor{purple}{$\text{\rm exp}(-2\gamma\lambda^- t)$} $\cdot$ \textcolor{teal}{$\frac{1}{t^{1/2}}$} \\ 
         \hline
         &  \begin{minipage}{0.27\textwidth} \begin{center} \textbf{Strongly convex, $\bm{\gamma >  \gamma_*}$} \end{center} \end{minipage}  & \begin{minipage}{0.27\textwidth} \begin{center} \textbf{Non-strongly convex\\ w/ noise, $\bm{r = 1, \widetilde{R} > 0}$} \end{center} \end{minipage} \\
         \hline
         Worst & 
          \textcolor{purple}{$\text{exp}(-\gamma t \lambda^- + \tfrac{\gamma^2}{2} (\lambda^+)^2 t)$} & $(R + \widetilde{R} \cdot \textcolor{purple}{d}) \cdot \textcolor{teal}{\frac{1}{t}}$\\
          \hline
          \textbf{Average} &  \textcolor{purple}{$ \text{exp}[ -\gamma t \big (1-\frac{r\gamma}{2} \big ) \big ( 1+r + \sqrt{(1+r)^2 - \frac{8}{\gamma}} \big ) \big ]$} & $R \cdot$ \textcolor{teal}{$\frac{1}{t^{1/2}}$} + $\widetilde{R} \cdot$ \textcolor{teal}{$\frac{1}{t^{1/2}} $} \\
          \hline
    \end{tabular}
    
    
       \caption{\textbf{Asymptotic convergence guarantees} for $f($\raisebox{0.1em}{$\xx_{ \lfloor \tfrac{n}{\beta} t \rfloor }$}$)\vspace{-0.3em} - \tfrac{\widetilde{R}}{2} \big (1-\tfrac{r\gamma}{2} \big )^{-1} \max\{0,1-r\}$ on the isotropic features model.
        Stepsizes smaller than $\gamma_*$  have \textcolor{purple}{linear rates} based only on $\lambda^-$ multiplied by a \textcolor{teal}{polynomial} term (\textcolor{teal}{$t^{\alpha}$} in Theorem~\ref{thm:main_critical_stepsize}). Larger stepsizes have linear rates with factor $\lambda^*$ made explicit here.
        Average-case complexity are strictly better than the worst-case complexity, in some cases by a factor $\gamma$ vs $\gamma^2$. 
        Note also how the rates highlight the freezing transition in the strongly convex regime.
        For worst-case rates, see \cite[Theorem 4.6]{ bottou2018optimization} \cite[Theorem 2.1]{ghadimi2013stochastic}; $\lambda^{+}$ can be replaced by the max-$\ell^2$-row-norm in the worst-case bounds below. 
    } 
    \label{tab:complexity_main}
\end{table}

\paragraph{Average-case complexity for SGD.} Our last contribution is one of the first average-case complexity results for any stochastic optimization algorithm. The value $\psi_0(t)$ is the average function value at iteration $t$ after first taking the model size to infinity. Consequently, this yields a notion of average complexity for SGD to a neighborhood. When the data matrix $\AA$ satisfies the isotropic features model $\AA$, we give an explicit formula for the expected function values $\psi_0(t)$, the critical stepsize $\gamma_*$, and the corresponding $\lambda^*$ (Appendix~\ref{app: explicit_sol_for_volterra}, Theorem~\ref{thm:isotropic_noisy} and Section~\ref{sec:explicit}, Theorem~\ref{thm:main_noiseless_SGD_dynamics}). Table~\ref{tab:complexity_main} summarizes our average rates.

The average-case complexity in the strongly convex case has significantly better linear rates than the worst-case guarantees and, in particular, there is no dependence on $\lambda^+$. We additionally capture a second-order behavior, the \textit{polynomial correction term} (\textcolor{teal}{green} in Table~\ref{tab:complexity_main}). This polynomial term has little effect on the complexity compared to the linear rate. However as the matrix $\HH$ becomes ill-conditioned $(r\to1)$, the polynomial correction starts to dominate the average-case complexity. The sublinear rates in Table~\ref{tab:complexity_main} for $r=1$ show this effect and it accounts for the improved average-case rates in the convex setting. This improvement in the average rate indeed highlights that \textit{the support of the spectrum does not fully determine the rate.} Many eigenvalues contribute meaningfully to the average rate. Hence, our results are not and cannot be purely explained by the support of the spectrum. As noted in \cite{paquette2020}, the worst-case rates when $r =1$ have dimension-dependent constants due to the distance to the optimum $\|\xx_0-\xx^\star\|^2 \approx d,$ which appears in the bounds. 



\paragraph{Related work.} 

\textit{Average-case versus worst-case complexity.} Traditional worst-case analysis of optimization algorithms provide complexity bounds no matter how unlikely \citep{nemirovski1995information, nesterov2004introductory}. There are a plethora of results on the worst-case analysis of SGD \citep{robbins1951, bertsekas2000gradient,ghadimi2013stochastic,bottou2018optimization,gower2019sgd} and in particular, specific results for SGD applied to the least squares problem (see \textit{e.g.} \cite{jain2017,bertsekas1997new}). Worst-case analysis gives convergence guarantees, but the bounds are not always representative of typical runtime. 

Average-case analysis, in contrast, gives sharper runtime estimates when some or all of its inputs are random. This type of analysis has a long history in computer-science and numerical analysis and it is often used to justify the superior performances of algorithms such as QuickSort \citep{Hoare1962Quicksort} and the simplex method, see for \textit{e.g.}, \citep{Spielman2004Smooth, smale1983on, borgwardt1986probabilistic}. Despite its rich history, average-case is rarely used in optimization due to the ill-defined notion of a typical objective function. Recently \citet{pedregosa2020average, lacotte2020optimal} derived a framework for average-case analysis of gradient-based methods on the least-squares problem with vanishing noise and it was later extended by \cite{paquette2020}. Similar results for the conjugate gradient method were derived in \citet{paquette2020universality,deift2019conjugate}. Our work is in the same line of research--providing the first average-case complexity for SGD.

For stochastic algorithms, \citet{Sagun2017Universal} showed empirical evidence that SGD on neural networks exhibits concentration of the function values. Other works \cite{mei2019mean,ward2020,sirignano2020mean,mert2020heavy-tail,mei2018mean} have used random matrix theory to analyze stochastic algorithms, but only in online or one-pass settings ($n \gg d$). We emphasize that our work applies to the finite sum setting; as we allow for multiple passes over the data. 

\textit{Continuous time processes.} A popular approach \citep{li2017stochastic,mandt2016variational, jastrzkebski2017three, nguyen2019first,zhu2019anisotropic,an2018stochastic} is to model the dynamics of SGD by imposing some structure on the
noise and, by sending stepsize to $0$, relate the iterates of SGD to the stochastic differential equation (SDE):
\begin{equation} \label{eq:SDE}
    \dif \XX_t = - \nabla f(\XX_t)\dif t + (\gamma \SSigma(\XX_t))^{1/2} \dif \boldsymbol{B}_t.
\end{equation}
Here one typically assumes the stochastic gradient noise $\nabla f_i(\xx) - \nabla f(\xx)$ is normally distributed (but not necessarily \citep{mert2019tail}) with some specific covariance structure $\SSigma(\XX)$. A common choice, called the stochastic modified equation (SME) \citep{li2017stochastic,mandt2016variational}, matches the covariance matrix $\SSigma(\XX)$ of the Gaussian noise with the actual covariance of the stochastic gradients at $\xx$ (\textit{i.e.} $\SSigma(\XX) = \frac{1}{n} \sum_{i=1}^n (\nabla f(\xx)-\nabla f_i(\xx))(\nabla f(\xx)-\nabla f_i(\xx))^T$ ). This covariance makes SME have correct mean behavior so the expected function values of the SME model are good approximations for the expected function values of SGD. \cite{li2017stochastic} show that by taking the stepsize $\gamma$ small, the behavior of SGD and SME align. They and \cite{mandt2016variational} also give a modified SME which gives even higher order accuracy of SGD as stepsize goes to $0$. 

 
These SDEs have been used to study numerous properties of SGD including the dynamics of regularized loss functions \citep{kunin2020neural} and generalization \citep{pflug1986stochastic,jastrzkebski2017three,zhu2019anisotropic,mert2019tail}. Despite their wide use, it has been observed that there is no small stepsize limit SGD that converges to an SDE \citep{yaida2018fluctuationdissipation}. Our approach, instead, looks at the large-$n$ limit and shows, in fact, that SGD concentrates while maintaining fixed stepsize. Moreover, our Volterra equation is relatively easy to analyze. We note that the SME has the same mean behavior as SGD so when $n \to \infty$, the mean behavior of SME and our Volterra equation match. \textit{However the SME does not capture this concentration effect and greatly overestimates the fluctuations of the sub-optimality.} 

\section{Dynamics of SGD: reduction to the Volterra equation} \label{sec:dynamics_of_SGD}

In this section, we develop the framework for the dynamics of SGD and sketch the argument of our main result (Theorem~\ref{thm: concentration_main}). Full proofs can be found in Appendix~\ref{apx:proof_main_result}. 

\paragraph{Step 1: Change of basis.} 
A key feature of the SGD least squares iteration \eqref{eq:oursgd} is that the projection of $\xx_k$ onto a singular vector $\vv_j$ of $\AA$ with singular value $\sigma_j$ decreases in expectation exponentially in the number of iterations at a rate proportionally to the squared singular value $\sigma_j^2$ \citep{strohmer2009randomized,steinerberger2020randomized}. This observation suggests the following change of basis.  Consider the singular value decomposition of $\frac{1}{\sqrt{n}}\AA = \UU \SSigma \VV^T$, where $\UU$ and $\VV$ are orthogonal matrices, i.e.\ $\VV \VV^T = \VV^T \VV = \II$ and $\SSigma$ is the $n \times d$ singular value matrix with diagonal entries $\diag(\sigma_j), j=1,\ldots,d$. We define the \emph{spectral weight} vector $\widehat{\nnu}_k \defas \VV^T (\xx_k-\widetilde{\xx}),$ which therefore evolves like
\begin{equation}\label{eq:main_nuk}
    \widehat{\nnu}_{k+1} = \widehat{\nnu}_k 
    - \gamma \SSigma^T
    \UU^T \PP_k
    (\UU\SSigma \widehat{\nnu}_k - \eeta).
\end{equation}
For this point on, we consider the evolution of $\widehat{\nnu}$. We note our above observation on the singular vectors only holds on average for individual coordinates of $\xx_k,$ and it does not alone explain the emergence of the Volterra equation dynamics.
It also guarantees nothing about the concentration of the suboptimality. 

\paragraph{Step 2: Embedding into continuous time.} We next consider an embedding of the $\widehat{\nnu}$ into continuous time.  This is done to simplify the analysis, and it does not change the underlying behavior of SGD.  We let $N_t$ be a standard univariate Poisson process with rate $\tfrac{n}{\beta},$ so that for any $t > 0,$ $\mathbb{E}(N_t) = \tfrac{nt}{\beta}$. We embed the spectral weights $\widehat{\nnu}$ into continuous time, by taking $\nnu_t = \widehat{\nnu}_{\tau_{N_t}}$.  We note that we have scaled time (by choosing the rate of the Poisson process) so that in a single unit of time $t,$ the algorithm has done one complete pass (in expectation) over the data set.

We then show that $f(\xx_{N_t})$ is well approximated by $\psi_0(t).$  As the mean of $N_t$ is large for any fixed $t > 0,$ the Poisson process concentrates around $\tfrac{nt}{\beta},$ and it follows as an immediate corollary that $f(\xx_{\lfloor \tfrac{nt}{\beta}\rfloor})$ is also well approximated by $\psi_0(t)$. 



\paragraph{Step 3: Doob--Meyer decomposition \& the approximate Volterra equation.} Under this continuous-time scaling, we can write the function values at $\xx_t$ in terms of $\nnu_t$ as
\begin{equation}\label{eq:main_psidef}
    \psi_{\varepsilon}(t) \defas f(\xx_{N_t})
    =
    \frac{1}{2}\| \SSigma \nnu_t - \UU^T\eeta\|^2 =
\frac{1}{2}
\sum_{j=1}^d \sigma_j^{2}\nu_{t,j}^2
-\sum_{j=1}^{n \wedge d} \sigma_j\nu_{t,j}\, (\UU^T\eeta)_j 
+\frac{1}{2}\|\eeta\|^2,
\end{equation}
where $\nu_{t,j}$ is the $j$-th coordinate of the vector $\nnu_t$. Hence the dynamics of $\psi_{\varepsilon}(t)$ are governed by the behaviors of $\nu_{t,j}$ and $\nu_{t,j}^2$ processes. Using \eqref{eq:main_nuk} and Doob decomposition for quasi-martingales \citep[Thm 18, Chpt. 3]{protter2005stochastic}, we have an expression for the $\nu_{t,j}$ and $\nu_{t,j}^2$, that is, if we let $\mathcal{F}_t$ be the $\sigma$-algebra of the information available to the process at time $t \ge 0$, we get
\begin{gather}
\nu_{t,j} = \nu_{0,j} + \int_0^t \mathcal{B}_{s,j}\,\dif s + \widetilde{M}_{t,j} \quad \text{and} \quad 
\nu_{t,j}^2 = \nu_{0,j}^2 + \int_0^t \mathcal{A}_{s,j}\, \dif s + M_{t,j}, \nonumber \\
\text{where} \quad \mathcal{B}_{t,j} \defas
\partial_t \E[\nu_{t,j}~|~\mathcal{F}_t]  = -\gamma \sigma_j^2 \nu_{t,j} + \gamma \sigma_j(\UU^T\eeta)_j,  \label{eq:main_doob} \\
\mathcal{A}_{t,j} \! \defas \! 
\partial_t \E[\nu_{t,j}^2~|~\mathcal{F}_t] = 2\nu_{t,j} \mathcal{B}_{t,j} \!
    +\tfrac{\beta-1}{n-1}
    \bigl(
    \mathcal{B}_{t,j}
    \bigr)^2\! \!+ \!\gamma^2 \sigma_j^2\big( 1-\tfrac{\beta-1}{n-1}\big)
    \sum_{i=1}^n
    \bigl(\mathbb{e}_j^T
    \UU^T 
    \mathbb{e}_{i}
    \bigr)^2
    \bigl(
    \mathbb{e}_{i}^T
    (\UU\SSigma \nnu_t - \eeta)\bigr)^2, \nonumber
\end{gather}
and $(M_{t,j}, \widetilde{M}_{t,j}: t \geq 0)$ are $\mathcal{F}_t$--adapted martingales. The last identities for $\mathcal{B}_{t,j}$ and $\mathcal{A}_{t,j}$ are derived in Lemma~\ref{lem:subsetsum} in Appendix \ref{apx:proof_main_result}). 

We will now see how the terms $\mathcal{A}_{t,j}$ and $\mathcal{B}_{t,j}$ can be simplified in the large-$n$ limit. In this regime, sums of spectral quantities converge to integrals against the limiting spectral measure $\mu$ as a direct consequence of Assumptions \ref{assumption: Vector} and \ref{assumption: spectral_density}. Since we are working in the regime where $\beta = o(n)$, the terms with $\tfrac{\beta-1}{n-1}$ vanish in the large-$n$ limit, disappearing entirely when $\beta = 1$, and explaining why $\beta=o(n)$ does not affect the limiting dynamics of SGD. Our \textit{key lemma}, which explains the Volterra dynamics of the mean of $f(\xx_{N_T})$ (Lemma~\ref{prop:errorKH}, App.~\ref{app:error_vanish_key_heuristic}), is that 
$\bigl(\mathbb{e}_j^T\UU^T  \mathbb{e}_{i} \bigr)^2$ self averages to $\tfrac{1}{n}$, and this is the point where we leverage orthogonal invariance of $\AA$ most heavily.  

These simplifications can be summarized as
\begin{equation} \begin{gathered} \label{eq: main_1}
    \mathcal{A}_{t,j} \approx \widehat{\mathcal{A}}_{t,j} \defas 
    \textcolor{teal}{-\gamma 2\sigma_j^2\nu_{t,j}^2}
    +\textcolor{purple}{\gamma 2\sigma_j\nu_{t,j}(\UU^T\eeta)_j}
    +\textcolor{mypurple}{\gamma^2\tfrac{2\sigma_j^2\psi_{\varepsilon}(t)}{n}}.
\end{gathered}
\end{equation}
The expression $\widehat{\mathcal{A}}_{t,j}$ explains the limiting Volterra dynamics for $\psi_{\epsilon}(t)$, and why the mean \textcolor{teal}{``gradient flow''} term does not correctly describe the dynamics of SGD.  Due to the \textcolor{teal}{gradient flow} term, the squared spectral weights $\nu_{t,j}^2$ tend to decay linearly with rate $2\gamma \sigma_j^2$. On the other hand, coordinates can not decay too quickly, as there is a \textcolor{mypurple}{mass redistribution} term, which explains the rate at which mass from other spectral weights is added to $\nu_{t,j}^2$ and which is due to SGD updates being noisy analogues for gradient flow.  Finally, there is a \textcolor{purple}{noise} term which in principle depends on $\nu_{t,j}$ which would greatly complicate the limiting dynamics.  However, when averaged in $j$ the independence of the noise $\eeta$ leads to a concentration effect, due to which only the mean behavior of $\nu_{t,j}$ survives.  As this mean is just gradient flow, this leads to a simple deterministic forcing term in the Volterra equation. 

Plugging \eqref{eq: main_1} into \eqref{eq:main_psidef} and \eqref{eq:main_doob}, we can produce a perturbed Volterra equation for $\psi_\varepsilon(t)$.
For any $t > 0$ we have
     \begin{gather}
\psi_\varepsilon(t) = 
{\frac{Rh_1(t)}{2}} + \textcolor{purple}{\frac{\widetilde{R}(rh_0(t) +(1-r))}{2}}
+ \varepsilon_1^{(n)}(t) + \!\int_0^t \!\bigl( \textcolor{mypurple}{\gamma^2 r h_2(s)} + \varepsilon_2^{(n)}(s)\bigr)\psi_\varepsilon(t-s)\,\dif s, 
\label{eq:main_approxvolterra}
    \end{gather}
for error terms $\varepsilon_i^{(n)}$  (see Appendix~\ref{app: bounding_the_error_terms} for a precise definition of the errors). The $h_k(t)$ are defined in Theorem~\ref{thm: concentration_main} as the Laplace transforms of the measure $\mu$, and arise naturally due to the presence of the \textcolor{teal}{gradient flow generator}.

\paragraph{Step 4: Control of the errors and stability of the Volterra equation.} 
The expression \eqref{eq:main_approxvolterra} is a Volterra equation of convolution type --- a well-studied equation, see \textit{e.g.}, \cite{gripenberg1990volterra}, with established stability and existence/uniqueness theorems. In particular, we can summarily conclude that (see Proposition \ref{prop:volterrastability} in Appendix~\ref{App: stability_volterra})
\[
\max_{i=1,2}
\sup_{0 \leq t \leq T} |\varepsilon_i^{(n)}(t)| \Prto[n] 0
\quad
\implies
\quad
\sup_{0 \le t \le T} |\psi_\varepsilon(t)-\psi_0(t) | \Prto[n] 0.
\]
Thus, Theorem~\ref{thm: concentration_main}, the dynamics for SGD immediately follows provided control of the errors in \eqref{eq:main_approxvolterra}. 

Beyond controlling the error of $\mathcal{A}_{t,j}-\widehat{\mathcal{A}}_{t,j},$ we must separately control the fluctuations of the martingale terms in \eqref{eq:main_doob}, which represent the randomness of SGD.  A central challenge here is to show in a suitable sense that the entries of $\sqrt{n}\nnu_t$ remain bounded on compact sets of time (see the discussion in Appendix \ref{app: bounding_the_error_terms} for a detailed overview), which in turn can be seen as a consequence of the updates of SGD being very nearly orthogonal to any fixed row of $\UU$.  Here again we use the orthogonal invariance of $\AA$, but in a weaker way, in that we only need that the maximum of the entries of $\UU$ are in control.  Such results are well-developed for other random matrix ensembles.

\section{Explicit formulas for isotropic features} \label{sec:explicit}




We solve the Volterra equation and derive exact expressions for the average-case analysis, the critical stepsize $\gamma_*$, and the rate $\lambda^*$ (Thm~\ref{thm:main_critical_stepsize}) under the isotropic features model. In this case, the empirical spectral measure converges to the Marchenko-Pastur measure $\MP$ \eqref{eq:MP}. Volterra equations of convolution type can be solved using Laplace transforms, which conveniently, for Marchenko-Pastur, are explicit due to a connection with the Stieltjes transform. This leads us to our next main result.

\begin{theorem}[Dynamics of SGD in noiseless setting] \label{thm:main_noiseless_SGD_dynamics} Suppose $\widetilde{R} = 0$ and the stepsize $\gamma < \tfrac{2}{r}$. Define the constants $\varrho$ and $\omega$ and critical stepsize
\begin{equation}
    \varrho = \frac{1+r}{2} \left (1-\frac{r\gamma}{2} \right ),  \, \,  \omega = \frac{1}{4} \left (1-\frac{r\gamma}{2} \right )^2 \left (\frac{8}{\gamma} - (1+r)^2 \right ), \, \, \text{and} \, \, \gamma_* = \frac{2}{\sqrt{r} (r-\sqrt{r} + 1)}.
\end{equation}
The iterates of SGD satisfy if $\gamma \le \gamma_*$,
\[f( \xx_{ \lfloor \tfrac{n}{\beta} t \rfloor } \big ) \Prto[n] R \cdot \frac{1}{\gamma} \left (1-\frac{r\gamma}{2} \right ) \int_0^\infty \frac{xe^{-2\gamma xt}}{(x-\varrho)^2 + \omega} \, \dif \MP(x)\]
and if $\gamma > \gamma_*$, for some explicit constant $c(\gamma, r)$, the iterates of SGD follow 
\begin{align*}
    f( &\xx_{ \lfloor \tfrac{n}{\beta} t \rfloor } \big ) \Prto[n]  R \cdot \frac{1}{\gamma} \bigg (1 - \frac{r \gamma}{2} \bigg )\int_0^\infty \frac{x e^{-2\gamma xt}}{(x-\varrho)^2 + \omega} \, \dif \MP(x) + R \cdot c(\gamma, r) \cdot  e^{-2\gamma (\varrho + \sqrt{|\omega|}) t}.
 \end{align*}
\end{theorem} 
We only record the dynamics for SGD in the noiseless regime and refer the reader to the Appendix~\ref{app: explicit_sol_for_volterra}, Theorem~\ref{thm:isotropic_noisy} for the noisy setting. We first observe the freezing transition as predicted by renewal theory -- a jamming term appears for $\gamma > \gamma_*$ that slows convergence. We note that when the ratio of features to samples $r$ does not equal $1$, the least squares problem in \eqref{eq:lsq} is (almost surely) strongly convex as $\dif \MP$ has a gap between the first non-zero eigenvalue and zero (see Figure~\ref{fig:MP}). As $r$ approaches $1$, the smallest non-zero eigenvalue become arbitrarily close to $0$. This phenomenon suggests different convergence rates in the regimes $r =1$ and $r \neq 1$. Moreover, we see the explicit value of $\lambda^*$, $\varrho + \sqrt{|\omega|}$ which vanishes when $r = 1$. We present our average-case rates in Table~\ref{tab:complexity_main}.  

\section{Numerical simulations} \label{sec: numerical_simulation}

We compare models of SGD's dynamics on two data distributions for moderately-sized problems ($n=1000$): the isotropic features model (see Section~\ref{sec:main_problem_setting}) and one-hidden layer network with random weights. In the latter model, the entries of $\AA$ are the result of a matrix multiplication composed with an activation function $g \, :  \mathbb{R} \to \mathbb{R}$:
\begin{align}
    A_{ij} \defas g \big (\tfrac{[\WW \YY]_{ij}}{\sqrt{m}} \big ), \quad \text{where $\WW \in \RR^{n \times m}$, $\YY \in \RR^{m \times d}$ are random matrices.} 
\end{align}
For the simulations, we took this activation function $g$ to be a shifted ReLU function; the shift makes $\EE[\AA]=0$ (see Appendix \ref{apx: tools_RMT} for details). This model encompasses two-layer neural networks with a squared loss, where the first layer has random weights and the second layer's weights are given by the regression coefficients $\xx$. Note that while the isotropic features model satisfies our assumptions, the one-hidden layer model does not. 
For all these approaches, we compute the objective suboptimality as a function of the number of passes over the dataset (epochs) for the models: (1). SDE (\textit{i.e.}, $\SSigma(\XX) = 0.01 \II$ in \eqref{eq:SDE}), (2). SME (\textit{i.e.}, $\SSigma(\XX)$ matches covariance of the stochastic gradients), (3). streaming (regenerate $\aa_i$ at each step), and (4). our Volterra equation. See Appendix \ref{apx:exp_details} for full details on the setup as well as experiments with other values of $r$. The outcome is displayed in Figure \ref{fig:comparision_SDE_volterra} and discussed in the caption. The fit of the Volterra equation to SGD is extremely accurate across different stepsizes and data distributions (some not covered by our assumptions) and even for medium-sized problems ($n=1000$).
We also note that while SME is often a good approximation, obtaining convergence rates from it is an open problem. On the other hand, the proposed Volterra equation can be analyzed through its link with renewal theory.


\begin{figure}[t]
    \centering
    \includegraphics[width=\linewidth]{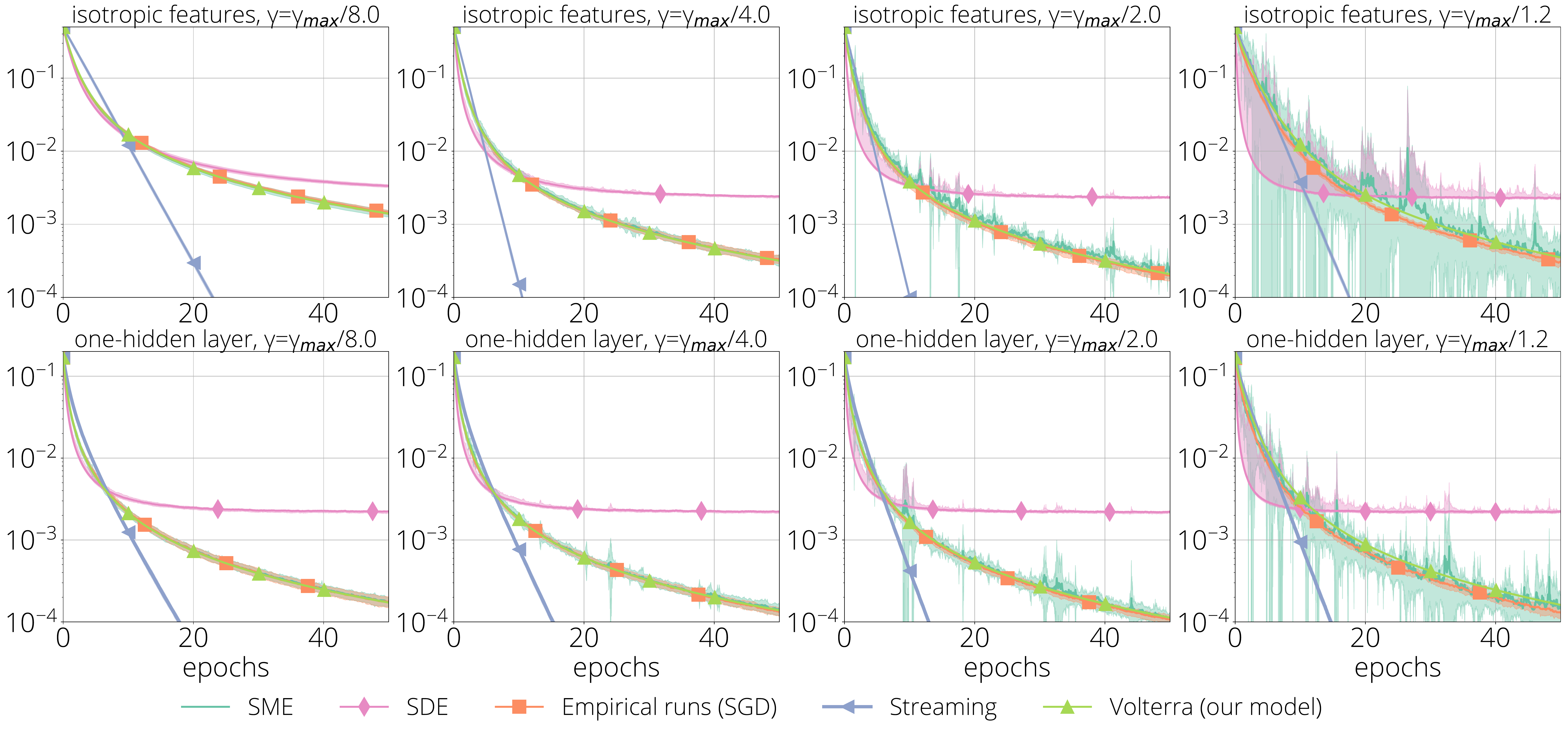} 
    
    \vspace{-0.25cm}
    
    \caption{{\bfseries Comparison of different SGD models}: isotropic features (top) and one-hidden layer network (bottom); $r = 1.2$. Across all stepsizes the Volterra overlaps the objective suboptimality of the empirical SGD runs (orange). The SME (teal) fits SGD for small stepsizes, whereas streaming (blue) and SDE (pink) have noticeable divergences from SGD for all stepsizes. Stochastic methods were averaged across 10 runs, with filled area representing the standard deviation. The parameter $\gamma_{\max}$ is the largest stepsize which still yields convergence of SGD, $\gamma_{\max} = \frac{2}{r} (\frac{1}{d}\text{tr}(\HH))^{-1}$ from Theorem~\ref{thm:main_critical_stepsize}. 
}
    \label{fig:comparision_SDE_volterra}
\end{figure}

\paragraph{Conclusion and future work} 
We have shown that the SGD method on least squares objectives admits a tight analysis in the large $n$ and $d$ limit. We described the dynamics of this algorithm through a Volterra integral equation and characterize its average-case convergence rate as well as its stepsize regimes. Although our results only hold in the large $n$-limit, the Volterra equation is remarkably accurate for relatively small dimensions (see \textit{e.g.} Figure~\ref{fig:comparision_SDE_volterra}).

While our theoretical results focus on problems with isotropic data matrix $\AA$, 
Figure \ref{fig:comparision_SDE_volterra} shows that the Volterra equation also predicts remarkably well the dynamics on data generated from a one-hidden layer network model. This suggests that the Volterra prediction might hold in even greater generality, a conjecture that is left for future work. Another direction of future work consists in extending to include other algorithms and problems. We believe the framework presented here should apply to methods like SGD momentum, RMSprop or ADAM and problems such as PCA.

\section*{Acknowledgements} The authors would like to thank our colleagues Nicolas Le Roux, Bart van Merri\"enboer, Zaid Harchaoui, Manuela Girotti, Gauthier Gidel, and Dmitriy Drusvyatskiy for their feedback on this manuscript.

\newpage
\bibliographystyle{plainnat}
\bibliography{reference}

\newpage
\appendix

\begin{center}
\LARGE{{\bfseries SGD in the Large:\\
\Large{Average-case Analysis, Asymptotics, and Stepsize Criticality}}\\
\vspace{0.5em}\Large{Supplementary material}
\vspace{0.5em}}
\end{center}
The appendix is organized into six sections as follows:
\begin{enumerate}
\item Appendix~\ref{apx: tools_RMT} expands upon the data examples in Section~\ref{sec:main_problem_setting}.
    \item Appendix~\ref{apx:proof_main_result} derives the Volterra equation and proves the main concentration for the dynamics of SGD (Theorem~\ref{thm: concentration_main}).
    \item We show in Appendix~\ref{sec:ubounds} that the error terms associated with concentration of measure on the high-dimensional orthogonal group disappear in the large-$n$ limit. This includes the \textit{key lemma}, Proposition~\ref{prop:errorKH}.
    \item Appendix~\ref{sec:mgles} shows the error terms which vanish due to martingale concentration results. 
    \item Appendix~\ref{app: Avg_case} derives the average-case complexity results from Section~\ref{sec:explicit} and provides a proof of the Malthusian exponent (Theorem~\ref{thm:main_critical_stepsize}).
    \item Appendix~\ref{apx:exp_details} contains details on the simulations. 
\end{enumerate}
Unless otherwise stated, all the results hold under Assumptions~\ref{assumption: Vector} and \ref{assumption: spectral_density}. We include all statements from the previous sections for clarity.

\paragraph{Notation.} 
All stochastic quantities defined hereafter live on a probability space denoted by $(\Pr, \Omega, \mathcal{F})$ with probability measure $\Pr$ and the $\sigma$-algebra $\mathcal{F}$ containing subsets of $\Omega$. A random variable (vector) is a measurable map from $\Omega$ to $\RR$ $(\RR^d)$ respectively. Let $X : (\Omega, \mathcal{F}) \to (\RR, \mathcal{B})$ be a random variable mapping into the borel $\sigma$-algebra $\mathcal{B}$ and the set $B \in \mathcal{B}$. We use the standard shorthand for the event $\{X \in B\} = \{\omega : X(\omega) \in B\}$. We denote the minimum of $a$ and $b$ by $\min\{a,b\} = a \wedge b$. An event $E$ that occurs \emph{with high probability} (or shortened to w.h.p.) is one whose probability depends on $n$, related to matrix dimension in our paper, and the probability of its complementary event goes to 0 as $n\to\infty$. Whereas event $E$ is said to occur \emph{with overwhelming probability} (or w.o.p.) if the probability of its complementary event goes to 0 faster than any polynomial order of $n$ as $n\to \infty$, i.e.\ for any $k>0$,
\[
\lim_{n\to\infty}\Pr(E^c)n^k = 0.
\]
 Throughout the paper, $\beta = \beta(n)$, which denotes the size of $B_k$, a uniformly random subset of $\{1,\cdots,n\}$ at the $k$-th iteration of SGD, is assumed to satisfy $\beta \le n^{1/5-\delta}$ for some $\delta>0$.

\section{Data distributions} \label{apx: tools_RMT}

\subsection{Elaboration on isotropy}\label{app:isotropy}

We recall that in Assumption \ref{assumption: spectral_density}, we have assumed that the data matrix $\AA$ is \emph{orthogonally invariant}.  This is a strong from of isotropy, under which the matrix looks the same in any orthogonal basis.  On a technical level, we work in this setting as it leads to a singular value decomposition with especially simple statistics. 

To state this property, we recall that the set of $n \times n$ orthogonal matrices form a group $O(n)$ under multiplication, and that this group naturally admits a Lie group structure.  In particular, there is a probability measure on this group, the \emph{Haar measure}, which is invariant by left and right multiplication by fixed orthogonal matrices.  To refer to a random matrix whose law is Haar measure, we will simply say a Haar-distributed orthogonal matrix.  While it may appear unwieldy, there are many exceptionally nice tools that exist for working with this measure.  We will elaborate on many of them in Section \ref{sec:ubounds}.  We also refer to \cite{Meckes} for a rich exposition on the intrinsic properties of this group.

The main feature that we will need is the following:
\begin{lemma}\label{lem:orthogonalinvariance}
Suppose that $\AA$ is an $n \times d$ orthogonally invariant random matrix in that $\AA \OO_d \law \AA$ and $\OO_n \AA \law \AA$ for any orthogonal matrices $\OO_n \in O(n)$ and $\OO_d \in O(d)$.  Then there is a singular value decomposition 
\[
    \AA = \UU \SSigma \VV
\]
with $\SSigma$ an $n \times d$ random matrix having 
\[
\SSigma_{11} \geq \SSigma_{22} \geq \SSigma_{33} \geq \cdots \geq \SSigma_{mm}
\quad\text{where}\quad m = \min\{n,d\},
\]
so that $( \UU, \SSigma, \VV)$ are independent, and $\UU$ and $\VV$ are Haar orthogonally distributed. 
\end{lemma}
\begin{proof}
The key observation is that if we introduce a new, independent Haar distributed random matrix $\UU \in O(n)$, then $\UU^T \AA$ has the same law as $\AA$ and moreover $(\UU, \UU^T\AA)$ are independent.  To see that $\UU^T \AA$ has the same law as $\AA$ we just observe that conditionally on $\UU^T,$ $\UU^T \AA \law \AA$ by assumption. As the conditional law does not depend on $\UU,$ it follows that $\UU^T\AA \law \AA$ \emph{and} $\UU$ is independent of $\UU^T\AA$.  Extending this, if we introduce two new independent Haar distributed random matrices $\UU \in O(n)$ and $\VV$ in $O(d),$ it follows that $(\UU, \UU^T\AA\VV^T, \VV)$ is a triple on independent random matrices.  
Let 
\[
\UU^T\AA\VV^T = \widetilde{\UU} \SSigma \widetilde{\VV}
\]
be the singular value decomposition with $\SSigma$ having the properties stated in the lemma.  Then
\[
\AA = 
\UU^T(\UU\AA\VV)\VV^T=
(\UU \widetilde{\UU}) \SSigma (\widetilde{\VV} \VV),
\]
with $\UU,\VV$ and $(\widetilde{\UU} \SSigma \widetilde{\VV})$ independent.  By invariance of Haar measure, the triple $(\UU \widetilde{\UU},\VV \widetilde{\VV},\SSigma)$ remain independent and uniformly distributed on $O(n)$ and $O(d)$ respectively.
\end{proof}
\noindent \emph{As a consequence, we will frequently condition on the singular values $\SSigma$ of $\AA$, and most estimates we need are estimates that hold conditionally on $\SSigma$.}

\subsection{Isotropic features and Random features} \label{app:isotropic_random}

In this section, we expand upon Assumptions~\ref{assumption: Vector} and \ref{assumption: spectral_density} in the main text of the paper. We discuss in detail two examples: isotropic features and one-hidden layer networks. 

\subsubsection{Isotropic features}
In their seminal work, \citet{marvcenko1967distribution} show that the spectrum of $\HH = \tfrac{1}{n} \AA^T \AA$ under the isotropic features model converged to a deterministic measure. Subsequent work then characterized the convergence of the largest eigenvalue of $\HH$. We summarize these results below. 

\begin{lemma}[Isotropic features]({\rm \textbf{\citet[Theorem 5.8]{bai2010spectral}}}) \label{lem:bai_Spectral} 
Suppose the matrix ${\AA \in \RR^{n \times d}}$ is generated using the isotropic features model. 
Then the empirical spectral measure (EMS) $\mu_{\HH}$ converges weakly almost surely to the Marchenko-Pastur measure $\MP$ and the largest eigenvalue of $\HH$, $\lambda_{\HH}^+$, converges in probability to $\lambda^+$ where $\lambda^+ = (1+ \sqrt{r})^2$ is the top edge of the support of the Marchenko-Pastur measure. 
\end{lemma}

The results stated so far did not require that the entries of $\AA$ are normally distributed, and hold equally well for any i.i.d.\ matrices with mean $0$, entry variance $1$ and bounded fourth-moment.  Under the additional assumption that the entries of $\AA$ are normally distributed, it is easily checked by a covariance computation that for fixed orthogonal matrices $\UU$ and $\VV$ the entries $\UU \AA$ and $\AA \VV$ remain independent, mean $0$ and variance $1$.  We summarize this claim below.
\begin{lemma}\label{lem:GinibreIndependence}
    For an $n \times d$ matrix $\AA$ of i.i.d.\ standard normal random variables, $\UU \AA$ and $\AA \VV$ are again $n \times d$ matrices of independent standard normals for fixed orthogonal matrices $\UU$ and $\VV$.
\end{lemma}
\noindent We emphasize that while this is not true for matrices $\AA$ with entries that are independent of mean $0$ and variance $1$, there are many senses in which this is approximately true (see \cite{KnowlesYin}).

\subsubsection{One-hidden layer networks}
\paragraph{One-hidden layer network with random weights.} In this model, the entries of $\AA$ are the result of a matrix multiplication composed with a (potentially non-linear) activation function $g \, : \, \mathbb{R} \to \mathbb{R}$:
\begin{align}\label{eq:random_features_model}
    A_{ij} \defas g \big (\tfrac{[\WW \YY]_{ij}}{\sqrt{m}} \big ), \quad \text{where $\WW \in \RR^{n \times m}$, $\YY \in \RR^{m \times d}$ are random matrices.} 
\end{align}
The entries of $\WW$ and $\YY$ are i.i.d. with zero mean, isotropic variances
    $\EE[W_{ij}^2] = \sigma_w^2$ and $\EE[Y_{ij}^2] = \sigma_y^2$, and light tails (see App.~\ref{app:isotropic_random} for details). As in the previous case to study the large dimensional limit, we assume that the different dimensions grow at comparable rates given by $\frac{m}{n} \to r_1 \in (0, \infty)$ and $\frac{m}{d} \to r_2 \in (0, \infty)$.
This model encompasses two-layer neural networks with a squared loss, where the first layer has random weights and the second layer's weights are given by the regression coefficients $\xx$.
Particularly, the optimization problem in \eqref{eq:lsq} becomes
\begin{equation} \label{eq: general_LS}
    \min_\xx \, \left\{ f(\xx) = \frac{1}{2n} \|\g \big ( \tfrac{1}{\sqrt{m}} \WW \YY \big )\xx - \bb\|^2_2  \right\}.
\end{equation} 

The model was introduced by \citep{Rahimi2008Random} as a randomized approach for scaling kernel methods to large datasets, and has seen a surge in interest in recent years as a way to study the generalization properties of neural networks
\citep{hastie2019surprises,mei2019generalization,pennington2017nonlinear}. 

The difference between this and the isotropic features model is the activation function, $g$. We assume $g$ to be entire with a growth condition and have zero Gaussian-mean (App.~\ref{app:isotropic_random}). These assumptions hold for common activation functions such as sigmoid $\g(z) = (1+ e^{-z})^{-1}$ and softplus $\g(z) = \log(1+e^z)$, a smoothed variant of ReLU.

\cite{benigni2019eigenvalue} recently showed that the empirical spectral measure and largest eigenvalue of $\HH$ converge to a deterministic measure and largest element in the support, respectively. 
This implies that this model verifies Assumption~\ref{assumption: spectral_density}.
However, contrary to the isotropic features model, the limiting measure does not have an explicit expression, except for some specific instances of $g$ in which it is known to coincide with the Marchenko-Pastur distribution. 

For the one-hidden layer network, following \citep{benigni2019eigenvalue}, we assume that the activation function $g$ is an entire function with a growth condition satisfying the following zero Gaussian mean property:
\begin{align} \label{eq: Gaussian_mean}
    \hspace{-3em}\text{(Gaussian mean)} \qquad \int \g(\sigma_w \sigma_y z) \tfrac{e^{-z^2/2}}{\sqrt{2 \pi} } \, \dif z = 0\,.
\end{align}
The additional growth condition on the function $g$ is precisely given as there exists positive constants $C_g, c_g, A_0 > 0$ such that for any $A \ge A_0$ and any $n \in \mathbb{N}$ 
\begin{align}
    \sup_{z \in [-A,A]} |g^{(n)}(z)| \le C_g A^{c_g n}\, .
\end{align}
This growth condition is verified for polynomials which can approximate to arbitrary precision common activation functions such as the sigmoid $\g(z) = (1+ e^{-z})^{-1}$ and the softplus $\g(z) = \log(1+e^z)$, a smoothed approximation to the ReLU. The Gaussian mean assumption \eqref{eq: Gaussian_mean} can always be satisfied by incorporating a translation into the activation function.

In addition to the i.i.d., mean zero, and isotropic entries, we also require an assumption on the tails of $\WW$ and $\YY$, that is, there exists constants $\theta_w, \theta_y > 0$ and $\alpha > 0$ such that for any $t > 0$
\begin{equation} \label{eq: light_tail}
    \Pr(|W_{11}| > t) \le \exp(-\theta_w t^{\alpha}) \quad \text{and} \quad \Pr(|Y_{11}| > t) \le \exp(-\theta_y t^{\alpha}).
\end{equation}
Although stronger than bounded fourth moments, this assumption holds for any sub-Gaussian random variables (\textit{e.g.}, Gaussian, Bernoulli, etc). Under these hypotheses, Assumption~\ref{assumption: spectral_density} is verified.

\begin{lemma}[One-hidden layer network]({\rm \textbf{\citet[Theorems~2.2 and~5.1]{benigni2019eigenvalue}}}) \label{lem:rand_feat_measure} Suppose the matrix $\AA \in \RR^{n \times d}$ is generated using the random features model. Then there exists a deterministic compactly supported measure $\mu$ such that $\mu_{\HH} \underset{d \to \infty}{\longrightarrow} \mu$ weakly almost surely. Moreover $\lambda_{\HH}^+ \Prto[d] \lambda^+$ where $\lambda^+$ is the top edge of the support of $\mu$.
\end{lemma}

\section{Derivation of the dynamics of SGD} \label{apx:proof_main_result}

In this section, we derive the Volterra equation from \eqref{eq:Volterra_eq_main}, that is,
\begin{equation} \begin{gathered}
    \psi_0(t) = \tfrac{R}{2} h_1(t) + \tfrac{\widetilde{R}}{2} \big ( rh_0(t)+ (1-r) \big ) + \int_0^t \gamma^2 r h_2(t-s)\psi_0(s)\,ds,\\
    \quad \text{and} \quad h_k(t) = \int_0^\infty x^k e^{-2\gamma tx} \, d\mu(x). 
\end{gathered}
\end{equation}
and prove Theorem~\ref{thm: concentration_main}:
\begin{equation}
  \sup_{0 \le t \le T}  \big | f \big ( \xx_{ \lfloor \tfrac{n}{\beta} t \rfloor } \big ) -  \psi_0(t) \big | \Prto[n] 0,
\end{equation}
provided that error terms go to zero. We begin by setting up the tools to derive an approximate Volterra equation. 

\subsection{Change of basis}
Recall the iterates of SGD satisfy 
\begin{equation}\label{eq:appsgd}
    \xx_{k+1} 
    = \xx_k - \frac{\gamma}{n}\sum_{i \in B_k}\nabla f_{i}(\xx_k) 
    = \xx_k - \frac{\gamma}{n}\AA^T \PP_k (\AA\xx_k - \bb), \quad \text{where} \quad \PP_k \defas \sum_{i \in B_k} \mathbb{e}_{i}\mathbb{e}_{i}^T.
\end{equation}
 Here $\PP_k$ is a random orthogonal projection matrix with $\mathbb{e}_i$ the $i$-th standard basis vector, $\beta \in \N$ is a batch-size parameter, which we will allow to depend on $n$, $\gamma > 0$ is a stepsize parameter, and the function $f_i$ is the $i$-th element of the sum in \eqref{eq:lsq}. 

We recall that $\bb$ has the representation $\bb = \AA \widetilde{\xx} + \sqrt{n} \, \eeta$, and both $\widetilde{\xx}$ and $\eeta$ have norms bounded independently of $n$.  Hence we can represent the updates of SGD \eqref{eq:oursgd}  equation in matrix form as
\begin{equation*}
    \xx_{k+1} 
    = \xx_k - \frac{\gamma}{n}\AA^T \PP_k (\AA(\xx_k - \widetilde{\xx}) + \sqrt{n}\eeta).
\end{equation*}

We will consider a singular value decomposition guaranteed by Lemma \ref{lem:orthogonalinvariance} of $\frac{1}{\sqrt{n}}\AA = \UU \SSigma \VV^T$, where $\UU$ and $\VV$ are Haar distributed orthogonal matrices, i.e.\ $\VV \VV^T = \VV^T \VV = \II$ and $\SSigma$ is the $n \times d$ singular value matrix with diagonal entries $\diag(\sigma_i), i=1,\ldots,n$.  Our analysis will use a different choice of variables.  So we define the vector $\widehat{\nnu}_k \defas \VV^T (\xx_k-\widetilde{\xx}),$ which therefore evolves like
\begin{equation}\label{eq:nuk}
    \widehat{\nnu}_{k+1} = \widehat{\nnu}_k 
    - \gamma \SSigma^T
    \UU^T \PP_k
    (\UU\SSigma \widehat{\nu}_k - \eeta).
\end{equation}

\subsection{Embedding into continuous time}
We next consider an embedding of the process $\widehat{\nnu}_k$ into continuous time.  This is done to simplify the analysis, and it does not change the underlying behavior of SGD. Let $\N_0 = \N \cup \{0\}$. We define an infinite random sequence $\{ \tau_k : k \in \N_0\} \subset [0,\infty)$ with $0 = \tau_0 < \tau_1 < \tau_2 < \cdots,$ which will record the time at which the $k$-th update of SGD occurs.  The distribution of these $\{ \tau_k : k \in \N_0\}$ will follow a standard rate-$(\tfrac{n}{\beta})$ Poisson process.  This means that the family of interarrival times $\{ \tau_k-\tau_{k-1} : k \in \N\}$ are i.i.d. $\text{Exp}(\tfrac{n}{\beta})$ random variables, \textit{i.e.},\ those with mean $\tfrac{\beta}{n}$, and we note that this randomization is independent of both the SGD, $\AA$ and $\bb.$  The function $N_t$ will count the number of arrivals of this Poisson process before time $t$, that is
\[
N_t = \sup \{ k \in \mathbb{N}_0 ~:~ \tau_k \leq t \}.
\]
Then for any $t > 0,$ $N_t$ has the distribution of $\text{Poisson}(\tfrac{n}{\beta}t)$.

We embed the process $\widehat{\nnu}$ into continuous time, by taking $\nnu_t = \widehat{\nnu}_{\tau_{N_t}}$.  We note that we have scaled time (by choosing the rate of the Poisson process) so that in a single unit of time $t,$ the algorithm has done one complete pass (in expectation) over the data set, i.e.\ SGD has completed one epoch.

\subsection{Doob-Meyer decomposition}
We compute the Doob decomposition for quasi-martingales \citep[Thm. 18, Chpt. 3]{protter2005stochastic} of $\nnu_{t}$ and of $\nu_{t,j}^2$, where $\nu_{t,j}$ is the $j$-th coordinate of $\nnu_t$ where $j$ ranges from $j = 1, \hdots, d$.  Here we let $\mathcal{F}_t$ be the $\sigma$-algebra of information available to the process at time $t \geq 0$.
So we take, for any $j \in [d],$
\[
\begin{aligned}
\bm{\mathcal{B}}_{t}
&\defas
\partial_t \E[\nnu_{t}~|~\mathcal{F}_t] =
\lim_{\epsilon \downarrow 0}
\epsilon^{-1}\E[\nnu_{t+\epsilon}-\nnu_{t}~|~\mathcal{F}_t] \\
\mathcal{A}_{t,j}
&\defas
\partial_t \E[\nu_{t,j}^2~|~\mathcal{F}_t] =
\lim_{\epsilon \downarrow 0}
\epsilon^{-1}\E[\nu_{t+\epsilon,j}^2-\nu_{t,j}^2~|~\mathcal{F}_t]. \\
\end{aligned}
\]
In terms of this random variable, we have a decomposition
\begin{equation}\label{eq:doob}
\begin{aligned}
\nnu_{t} &= \nnu_{0} + \int_0^t \bm{\mathcal{B}}_{s}\, \dif s + \widetilde{\MM}_{t}, \\
\nu_{t,j}^2 &= \nu_{0,j}^2 + \int_0^t \mathcal{A}_{s,j}\, \dif s + M_{t,j},
\end{aligned}
\end{equation}
where $(M_{t,j}, \widetilde{\MM_{t}}: t \geq 0)$ are $\mathcal{F}_t$--adapted martingales.

For the computation of $\mathcal{A}_{t,j}$ we observe that as $\epsilon \to 0,$ $\mathcal{A}_{t,j}$ is dominated by the contribution of a single Poisson point arrival; as in time $\epsilon$, the probability of having multiple Poisson point arrivals is $O(\beta^{-2}n^2 \epsilon^2),$ whereas the probability of having a single arrival is $1-e^{-\beta^{-1}n\epsilon} \sim \beta^{-1}n\epsilon$ as $\epsilon \to 0$. For notational simplicity, we let the projection matrix $\PP \in \mathbb{R}^{d \times d}$ be an i.i.d. copy of $\PP_{1}$, which is independent of all the randomness so far and we let $B$ be the corresponding random subset of $\{1,2,\hdots, n\}$ that defines $\PP$.
It follows 
\[
\begin{aligned}
\bm{\mathcal{B}}_{t}
&= \frac{n}{\beta}\E \biggl[
       \nnu_{t}
    -\gamma \SSigma^T
    \UU^T 
    \PP
    (\UU\SSigma \nnu_t - \eeta)
    - 
    \nnu_{t}
    ~\biggl\vert~
    \mathcal{F}_t
\biggr], \\
\mathcal{A}_{t,j}
&= \frac{n}{\beta}\E \biggl[
   \biggl(
       \nu_{t,j}
    -\gamma \mathbb{e}_j^T\SSigma^T
    \UU^T 
    \PP
    (\UU\SSigma \nnu_t - \eeta)
    \biggr)^2
    - 
    \nu_{t,j}^2
    ~\biggl\vert~
    \mathcal{F}_t
\biggr].
\end{aligned}
\]

The mean term of $\nnu_t$ simplifies significantly, and by no accident:  by construction, the SGD update rule has a conditional expectation which is proportional to the gradient of the objective function.  Observe that since
\[
\E[\PP] = \frac{\beta}{n} \sum_{i=1}^n \mathbb{e}_{i} \mathbb{e}_{i}^T = \frac{\beta}{n} \II,
\]
the previous equation simplifies to:
\begin{equation}\label{eq:Bt}
\bm{\mathcal{B}}_{t}
=
-\gamma
\SSigma^T
(\SSigma
\nnu_t - \UU^T\eeta)
\quad\text{and}\quad
\mathcal{B}_{t,j}
= -\gamma \sigma_j^2 \nu_{t,j} + \gamma \sigma_j(\UU^T\eeta)_j.
\end{equation}

We now turn to the evaluation of $\mathcal{A}_{t,j}.$ 
If we let $\PP = \sum_{i \in B} \mathbb{e}_i \mathbb{e}_i^T,$
\begin{equation}\label{eq:Atj}
\begin{aligned}
\mathcal{A}_{t,j}
&= -2\nu_{t,j}\gamma
   \mathbb{e}_j^T\SSigma^T
    \UU^T 
    (\UU\SSigma \nnu_t - \eeta) 
    +\gamma^2\frac{n}{\beta}
    \E
    \biggl(\mathbb{e}_j^T\SSigma^T
    \UU^T 
    \PP
    (\UU\SSigma \nnu_t - \eeta)
    ~\bigg\vert~ \mathcal{F}_t
    \biggr)^2
    \\
    &= 2\nu_{t,j} \mathcal{B}_{t,j}
    +\gamma^2\frac{n\sigma_j^2}{\beta}
    \E
    \biggl(
    \sum_{i\in B}
    \bigl(\mathbb{e}_j^T
    \UU^T 
    \mathbb{e}_{i}
    \bigr)
    \bigl(
    \mathbb{e}_{i}^T
    (\UU\SSigma \nnu_t - \eeta)\bigr)
    ~\bigg\vert~ \mathcal{F}_t
    \biggr)^2.
    \end{aligned}
\end{equation}
To compute this conditional expectation, we record the following lemma.
\begin{lemma}\label{lem:subsetsum}
    Suppose that $\uu$ and $\vv$ are fixed vectors in $\RR^n$.  Then
    \[
    \begin{aligned}
    \E\biggl(
    \sum_{i \in B} u_i v_i
    \biggr)^2
    &= \frac{\beta(\beta-1)}{n(n-1)}(\uu^T \vv)^2 + \biggl(\frac{\beta}{n} - \frac{\beta(\beta-1)}{n(n-1)}\biggr)\sum_{i=1}^n (u_iv_i)^2. \\
    \end{aligned}
    \]
\end{lemma}
\begin{proof}
This reduces to the two probabilities:
\[
\Pr( i \in B) = \frac{\beta}{n},
\quad
\text{and}
\quad
\Pr( i,\ell \in B) = \frac{\beta(\beta-1)}{n(n-1)}, 
\]
where $i\neq \ell$ are any fixed numbers in $\{1,2,\ldots,n\}$. The proof now follows by expanding both sides.
\end{proof}

Using Lemma \ref{lem:subsetsum}, we can therefore simplify \eqref{eq:Atj} by writing
\begin{equation}\label{eq:Atj2}
\mathcal{A}_{t,j}
    = 2\nu_{t,j} \mathcal{B}_{t,j}
    +\frac{\beta-1}{n-1}
    \bigl(
    \mathcal{B}_{t,j}
    \bigr)^2
    +\gamma^2 \sigma_j^2\biggl( 1-\frac{\beta-1}{n-1}\biggr)
    \sum_{i=1}^n
    \bigl(\mathbb{e}_j^T
    \UU^T 
    \mathbb{e}_{i}
    \bigr)^2
    \bigl(
    \mathbb{e}_{i}^T
    (\UU\SSigma \nnu_t - \eeta)\bigr)^2.
\end{equation}

\subsection{Constructing the approximate Volterra equation} \label{app:volterra_equation}
In this section, we derive an approximate Volterra equation. First, we can write the function values at the iterates $\xx_t$ in terms of $\nnu_t$,
\begin{equation} \begin{aligned} \label{eq:psidef}
    \psi_{\varepsilon}(t) \defas f(\xx_{N_t})
    &=
    \frac{1}{2n}\| \AA (\xx_{N_t}-\widetilde{\xx}) - \sqrt{n}\, \eeta \|^2
    =
    \frac{1}{2}\| \SSigma \nnu_t - \UU^T\eeta\|^2.\\
    &=\frac{1}{2}
\sum_{j=1}^d \sigma_j^{2}\nu_{t,j}^2
-\sum_{j=1}^{n \wedge d} \sigma_j\nu_{t,j}\, (\UU^T\eeta)_j 
+\frac{1}{2}\|\eeta\|^2.
\end{aligned}
\end{equation}
Hence the dynamics of $\psi_{\varepsilon}(t)$ are governed by the behaviors of $\nu_{t,j}$ and $\nu_{t,j}^2$ processes. We now return to $\mathcal{A}_{t,j}$. The \emph{key lemma} to simplifying \eqref{eq:Atj2} is that the $\bigl(\mathbb{e}_j^T\UU^T \mathbb{e}_{i}\bigr)^2$ expression in \eqref{eq:Atj2}  self-averages to $\frac{1}{n}.$ Furthermore, we are working in the regime when $\beta = o(n),$ and hence the terms with $\tfrac{\beta}{n}$ will vanish in the large-$n$ limit.  Thus we define
\begin{equation}\label{eq:Ahat}
\widehat{\mathcal{A}}_{t,j}
\defas
    2\nu_{t,j} \mathcal{B}_{t,j}
    +\gamma^2 \sigma_j^2
    \sum_{i=1}^n
    \frac{1}{n}
    \biggl(
    \mathbb{e}_{i}^T
    (\UU\SSigma \nnu_t - \eeta)\biggr)^2
    =
    -\gamma 2\sigma_j^2\nu_{t,j}^2 
    +\gamma 2\sigma_j\nu_{t,j}(\UU^T\eeta)_j 
    +\gamma^2\frac{2\sigma_j^2\psi_{\varepsilon}(t)}{n}.
\end{equation}
We will show that $\widehat{\mathcal{A}}_{t,j}$ is a good approximation for $\mathcal{A}_{t,j}$ in a suitably strong sense so that we can derive a deterministic Volterra equation description for $\psi_{\varepsilon}(t)$ in the large-$n$ limit. For the moment, let's group these error terms together. Define the c\`adl\`ag process
\begin{equation}\label{eq:totalerror}
\mathcal{E}_{t,j} \defas 
\nu_{t,j}^2 - \nu_{0,j}^2 - \int_0^t \widehat{\mathcal{A}}_{s,j}\, \dif s.
\end{equation}
In the following lemma, we get an expression for $\nu_{t,j}$. 
\begin{lemma}\label{lem:integrated}
For any $t \geq 0$ and for any $1 \leq j \leq d,$
\[
\begin{aligned}\nu_{t,j}^2
&=
e^{-2t\gamma\sigma_j^2}\nu_{0,j}^2 +
\int_0^t 
e^{-2(t-s)\gamma\sigma_j^2}
\biggl(\bigl(\gamma 2\sigma_j\nu_{s,j}(\UU^T\eeta)_j
+\tfrac{\gamma^22\sigma_j^2\psi_{\varepsilon}(s)}{n}\bigr) \dif s + \dif \mathcal{E}_{s,j}\biggr) \\
    \text{and} \qquad \nu_{t,j}
    &=
    e^{-\gamma \sigma_j^2t}\nu_{0,j} 
    +\int_0^t e^{-\gamma \sigma_j^2(t-s)}\bigl(\gamma \sigma_j(\UU^T\eeta)_j \dif s + \dif \widetilde{M}_{s,j}\bigr).
\end{aligned}
\]
\end{lemma}
\begin{proof}
We show the first equation.  The second follows by a similar argument.
Using the definition of $\mathcal{E}_{t,j}$, the following holds
\[
\nu_{t,j}^2
= 
\nu_{0,j}^2  + 
\int_0^t
\biggl(
-\gamma 2\sigma_j^2\nu_{s,j}^2
+\gamma 2\sigma_j\nu_{s,j}(\UU^T\eeta)_j 
+\gamma^2\frac{2\sigma_j^2\psi_{\varepsilon}(s)}{n}
\biggr)
\, \dif s
+\mathcal{E}_{t,j}.
\]
Using c\`adl\`ag differentiation, we get that
\[
\dif (e^{2t\gamma\sigma_j^2}\nu_{t,j}^2)
= 
e^{2t\gamma\sigma_j^2}\gamma 2\sigma_j\nu_{t,j}(\UU^T\eeta)_j
+e^{2t\gamma\sigma_j^2}\gamma^2\frac{2\sigma_j^2\psi_{\varepsilon}(t)}{n}
+e^{2t\gamma\sigma_j^2}\dif \mathcal{E}_{t,j}.
\]
Hence integrating both sides, one obtains
\[
\nu_{t,j}^2=
e^{-2t\gamma\sigma_j^2}
\biggl(\nu_{0,j}^2 + \int_0^t e^{2s\gamma\sigma_j^2}
\biggl(\gamma 2\sigma_j\nu_{s,j}(\UU^T\eeta)_j
+\gamma^2\frac{2\sigma_j^2\psi_{\varepsilon}(s)}{n}\biggr)\dif s
+\int_0^t e^{2s\gamma\sigma_j^2}\dif \mathcal{E}_{s,j}\biggr),
\]
which completes the proof.
\end{proof}
We then apply Lemma \ref{lem:integrated} to \eqref{eq:psidef} and we derive the \textit{approximate Volterra equation}
\begin{equation}\label{eq:psivolterra}
\begin{aligned}
\psi_{\varepsilon}(t)
&=
\frac{1}{2}
\sum_{j=1}^d \sigma_j^{2}
\biggl(e^{-2t\gamma\sigma_j^2}\nu_{0,j}^2 + \int_0^t e^{-2(t-s)\gamma\sigma_j^2}\gamma^2\frac{2\sigma_j^2\psi_{\varepsilon}(s)}{n}\dif s\biggr) \\ 
&+
\frac{1}{2}
\sum_{j=1}^d
\int_0^t 
e^{-2(t-s)\gamma\sigma_j^2}
\gamma 2\sigma_j^3\nu_{s,j}(\UU^T\eeta)_j \dif s
+\frac{1}{2}\|\eeta\|^2
    -\sum_{j=1}^{n \wedge d} \sigma_j\nu_{t,j}\, (\UU^T\eeta)_j \\
    &+\frac{1}{2}\sum_{j=1}^d 
    \sigma_j^{2}
    \int_0^t e^{-2(t-s)\gamma\sigma_j^2}\dif \mathcal{E}_{s,j}.
\end{aligned}
\end{equation}
In this expression, we have gathered the terms on each line 
that have different limit behaviors.  
On the first line, we have the terms, that due to the convergence of the empirical measure of singular values (Assumption \ref{assumption: spectral_density}), will have continuum limits.
The second line are those terms that survive in the limit due to the effect of noise $\eeta$. The third line are error terms that vanish in the limit.  We will make explicit the convergence in the first two lines in the following lemmata.

\subsection{Stability of the Volterra equation} \label{App: stability_volterra}
We begin by defining some Laplace transforms of the limiting spectral measures, for $k \in \N_0,$
\begin{equation}\label{eq:deterministic}
h_k(t) \defas \int_0^\infty 
    x^ke^{-2\gamma tx} \dif \mu(x).
\end{equation}
We begin by showing that the terms on the first line of \eqref{eq:psivolterra} converge to some finite limit. Under our Assumption~\ref{assumption: spectral_density},
\begin{lemma}\label{lem:modelconvergence}
Locally uniformly on compact sets of time,
\begin{align}
\frac{1}{2}\sum_{j=1}^d \sigma_j^{2} e^{-2t\gamma\sigma_j^2}\nu_{0,j}^2
\Prto[n] \frac{Rh_1(t)}{2} \label{eq:conv1}
\quad\text{and}\\
\sum_{j=1}^d \gamma^2\frac{\sigma_j^{4}}{n} e^{-2t\gamma\sigma_j^2}
\Prto[n] \gamma^2 r h_2(t). \label{eq:conv2}
\end{align}
\end{lemma}

\begin{proof}
We begin by showing that each term in \eqref{eq:conv1} and \eqref{eq:conv2} converge pointwise in probability. Note that the convergence of \eqref{eq:conv2} is trivial because of Assumption  \ref{assumption: spectral_density}. Hence, it only remains to show pointwise convergence of \eqref{eq:conv1}.

Under Assumption \ref{assumption: Vector} and using uniform distribution of $\VV$ we have that $\nnu_0 = \VV (\xx_0 - \widetilde{\xx})$ is a uniformly distributed vector on the sphere of norm $\sqrt{R}$.
Observe first that conditioned on $\SSigma$, the conditional expectation of the LHS of \eqref{eq:conv1} is given by
\[
\mathbb{E} \biggl[\frac{1}{2}\sum_{j=1}^d \sigma_j^{2} e^{-2t\gamma\sigma_j^2}\nu_{0,j}^2
\bigg| \SSigma
\biggr] = \frac{R}{2}\frac{1}{d}\sum_{j=1}^d \sigma_j^{2} e^{-2t\gamma\sigma_j^2}.
\]
The vector $\nnu_0^2$ follows the Dirichlet distribution, which is negatively associated.  In particular, 
\(
\mathbb{E}( \nu_{0,j}^2\nu_{0,j}^2) \leq \mathbb{E}( \nu_{0,j}^2) \mathbb{E}(\nu_{0,j}^2)
\)
. Further $\mathbb{E}( \nu_{0,j}^4) \leq 3R^2d^{-2},$ as the moments are strictly bounded by the normal moments.  Hence, the variance is bounded by 
\begin{align*}
\Var\biggl( \frac{1}{2}\sum_{j=1}^d \sigma_j^{2} e^{-2t\gamma\sigma_j^2}\nu_{0,j}^2
\bigg| \SSigma \biggr)
&=\mathbb{E} \biggl[\bigl(\frac{1}{2}\sum_{j=1}^d \sigma_j^{2} e^{-2t\gamma\sigma_j^2}\nu_{0,j}^2 -
\frac{R}{2}\frac{1}{d}\sum_{j=1}^d \sigma_j^{2} e^{-2t\gamma\sigma_j^2}
\bigr)^2
\bigg| \SSigma
\biggr]\\
&= \mathbb{E} \biggl[\bigl(\frac{1}{2d}\sum_{j=1}^d \sigma_j^{2} e^{-2t\gamma\sigma_j^2}
(d\nu_{0,j}^2 - R)
\bigr)^2
\bigg| \SSigma
\biggr]\\
&\leq \frac{1}{4d^2}\mathbb{E}
\biggl[ \sum_{j=1}^d (\sigma_j^{2} e^{-2t\gamma\sigma_j^2})^2
(d\nu_{0,j}^2 - R)^2
\bigg| \SSigma
\biggr]\\
&= \frac{1}{4d}\biggl[\frac{1}{d}\sum_{j=1}^d \sigma_j^{4} e^{-2\cdot2t\gamma\sigma_j^2}
\biggr]
(d^2\mathbb{E}[\nu_{0,j}^4] - R^2) = \mathcal{O}\biggl(\frac{1}{d}\biggr).
\end{align*}
Therefore, for $\epsilon>0$, conditional Chebyshev inequality gives
\[
\Pr \biggl(\biggl|
\frac{1}{2}\sum_{j=1}^d \sigma_j^{2} e^{-2t\gamma\sigma_j^2}\nu_{0,j}^2 -
\frac{R}{2}\frac{1}{d}\sum_{j=1}^d \sigma_j^{2} e^{-2t\gamma\sigma_j^2}
\biggr| > \epsilon
\bigg| \SSigma
\biggr) \le \frac{1}{\epsilon^2} \Var\biggl( \frac{1}{2}\sum_{j=1}^d \sigma_j^{2} e^{-2t\gamma\sigma_j^2}\nu_{0,j}^2
\bigg| \SSigma \biggr)
 { \xrightarrow[ n \to \infty]{}} 0.
\]
Applying the law of total probability to this and combining it with the weak convergence of ESM in probability (Assumption \ref{assumption: spectral_density}) gives
\begin{align*}
    \Pr \biggl( \biggl| \frac{1}{2}\sum_{j=1}^d \sigma_j^{2} e^{-2t\gamma\sigma_j^2}\nu_{0,j}^2
    &- \frac{Rh_1(t)}{2} \biggr| > \epsilon \biggr)\\
    &\le \Pr\biggl(\biggl|
    \frac{1}{2}\sum_{j=1}^d \sigma_j^{2} e^{-2t\gamma\sigma_j^2}\nu_{0,j}^2 -
    \frac{R}{2}\frac{1}{d}\sum_{j=1}^d \sigma_j^{2} e^{-2t\gamma\sigma_j^2}
    \biggr| > \frac{\epsilon}{2}
    \biggr)\\
    &+\Pr \biggl( \biggl|
    \frac{R}{2}\frac{1}{d}\sum_{j=1}^d \sigma_j^{2} e^{-2t\gamma\sigma_j^2}
    - \frac{Rh_1(t)}{2} \biggr| > \frac{\epsilon}{2} \biggr)
    { \xrightarrow[ n \to \infty]{}} 0.
\end{align*}

Now we show the uniform convergence of \eqref{eq:conv1} on time interval $[0,T]$ for a fixed time $T>0$ (the same argument applies to \eqref{eq:conv2}). Considering mesh points on $[0,T]$ with spacing, let us say $\lambda >0$, we can say that the pointwise convergence holds on those mesh points and so does the supremum convergence on them. For an arbitrary time $t\in [0,T]$, there exists a mesh point $t_0$ such that $|t-t_0|\le \lambda$. Then, since $e^{-2t\gamma\sigma_j^2}$ is a Lipschitz function on $[0,T]$ with some Lipschitz constant $C>0$, we have
\[
\sup_{0\le t\le T} \biggl| \frac{1}{2}\sum_{j=1}^d \sigma_j^{2} (e^{-2t\gamma\sigma_j^2} - e^{-2t_0\gamma\sigma_j^2})\nu_{0,j}^2 \biggr|
\le \frac{C\lambda}{2}\sum_{j=1}^d \sigma_j^2 \nu_{0,j}^2.
\]
Note that $\sum_{j=1}^d \sigma_j^2 \nu_0^2 \Prto[n] R \int_0^\infty x^2 d\mu(x) < \infty$ using a similar idea by conditioning on $\SSigma$ and applying Assumption \ref{assumption: spectral_density}. Then observe, applying triangle inequality and taking supremum on $t\in [0,T]$ gives
\begin{align*}
    \sup_{0\le t\le T} \biggl| \frac{1}{2}\sum_{j=1}^d \sigma_j^{2} e^{-2t\gamma\sigma_j^2}\nu_{0,j}^2
    - \frac{Rh_1(t)}{2} \biggr| &\le 
    \sup_{t_0\in [0,T]}\biggl| \frac{1}{2}\sum_{j=1}^d \sigma_j^{2} e^{-2t_0\gamma\sigma_j^2}\nu_{0,j}^2
    - \frac{Rh_1(t_0)}{2} \biggr|\\
    &\quad +\sup_{0\le t\le T} \biggl| \frac{1}{2}\sum_{j=1}^d \sigma_j^{2} (e^{-2t\gamma\sigma_j^2} - e^{-2t_0\gamma\sigma_j^2})\nu_{0,j}^2 \biggr|\\
    &\quad +  \sup_{t,t_0\in [0,T]} \frac{R}{2} \bigl| h_1(t_0) - h_1(t) \bigr|.
\end{align*}
Given that $\mu$ has a finite support, we have $\sup_{t,t_0\in [0,T]}|h_1(t)-h_1(t_0)| \le C' \lambda$ for some $C'>0$. Now the claim follows as $\lambda$ can be chosen as small as possible. 
\end{proof}
We can now recast \eqref{eq:psivolterra} as an approximate Volterra type integral equation, where
\begin{equation}\label{eq:approxvolterra}
\psi_\varepsilon(t) = 
\frac{Rh_1(t)}{2} + {\widetilde{R}}\,\cdot\frac{rh_0(t) +(1-r)}{2}
+ \varepsilon_1^{(n)}(t) + \int_0^t \biggl(\gamma^2 r h_2(t-s) + \varepsilon_2^{(n)}(t-s)\biggr)\psi_\varepsilon(s)\,\dif s,
\end{equation}
and where $\varepsilon_{i}^{(n)}$ are defined implicitly by comparison with \eqref{eq:psivolterra}.  In particular, $\varepsilon_2^{(n)}(t)$ is given by the difference
\[
\varepsilon_2^{(n)}(t)
=
\sum_{j=1}^d \gamma^2\frac{\sigma_j^{4}}{n} e^{-2t\gamma\sigma_j^2}
- \gamma^2 r h_2(t),
\]
which is therefore guaranteed to converge to $0$ by Lemma \ref{lem:modelconvergence}.  The other error $\varepsilon_1^{(n)}$ is substantially more complicated; we discuss it fully in  \eqref{eq:errorpieces}, Section~\ref{app: bounding_the_error_terms}.

However, all we need to show is that this error tends to $0,$
as Volterra equations are stable:
\begin{proposition}[Stability of the Volterra equation]\label{prop:volterrastability} Fix a constant $T>0$ and 
suppose $\psi_\varepsilon$ solves \eqref{eq:approxvolterra} with bounded error terms,
\[
\sup_{0 \leq t \leq T} |\varepsilon_1^{(n)}(t)| \Prto[n] 0
\quad
\text{and}
\quad
\sup_{0 \leq t \leq T} |\varepsilon_2^{(n)}(t)| \Prto[n] 0.
\]
 Suppose that $\psi_0$ solves \eqref{eq:approxvolterra} with $\varepsilon_1^{(n)} = \varepsilon_2^{(n)} = 0$.  Then for any fixed length of time $T$, the perturbed solution of the Volterra equation $\psi_{\varepsilon}$ converges uniformly, in probability, to the unperturbed solution of the Volterra equation, $\psi_0(t)$:
\[
\sup_{0 \leq t \leq T} |\psi_\varepsilon(t) - \psi_0(t)| \Prto[n] 0.
\]
\end{proposition}

\begin{proof} Throughout this proof, we set $\mathbb{R}^+ = \{ x \in \mathbb{R} \, : \, x \ge 0\}$ and we use the notation $L^1_{\text{loc}}(\mathbb{R}^+)$ to be the locally integrable functions on $\mathbb{R}^+$. We will supress the $n$ dependency in the error terms $\varepsilon_1^{(n)}$ and $\varepsilon_2^{(n)}$. We begin by defining some notation for solving Volterra equations, namely the kernel and forcing function respectively by
\begin{equation}
    k_{\varepsilon}(t) \defas - \big ( \gamma^2 rh_2(t) + \varepsilon_2(t) \big ) \quad \quad \text{and} \quad \quad  f_\varepsilon(t) \defas \frac{R h_2(t)}{2} + \widetilde{R} \cdot \frac{rh_0(t)+1-r}{2} + \varepsilon_1(t).
\end{equation}
Here we use the convention that $k_0$ and $f_0$ correspond to where $\varepsilon_1(t) = \varepsilon_2(t) = 0$. Under this notation, the Volterra equation in \eqref{eq:approxvolterra} becomes
\begin{equation} \label{eq:volterra_stuff} \psi_{\varepsilon}(t) + \int_0^t k_{\varepsilon}(t-s) \psi_{\varepsilon}(s) \, \dif s = f_{\varepsilon}(t).
\end{equation}
Now we check that $k_{\varepsilon}(t) \in L^1_{\text{loc}}(\mathbb{R}^+)$ with high probability. To see this we only need that $h_2(t) \in L^1_{\text{loc}}(\mathbb{R}^+)$ as the supremum condition on $\varepsilon_2(t)$ guarantees that the error term in $k_{\varepsilon}$ is bounded with high probability and therefore $\varepsilon_2(t)$ is in $L^1_{\text{loc}}(\mathbb{R}^+)$ with high probability. 
Since $h_2(t) \ge 0$, we can apply Tonelli's theorem 
\begin{align*}
    \int_0^\infty h_2(t) \, \dif t = \int_0^\infty \int_0^\infty x^2 e^{-2 \gamma t x} \, \dif t \, \dif \mu(x) = \int_0^\infty \frac{x}{2\gamma} \dif \mu(x) < \infty.
\end{align*}
Here we used that $\mu(x)$ is compactly supported to conclude the last integral. Hence it follows that $h_2(t) \in L^1(\mathbb{R}^+)$ which shows that $k_{\varepsilon}(t) \in L^1_{\text{loc}}(\mathbb{R}^+)$. 
To prove the conclusion of the proposition, we will use a stability theorem together with the existence and uniqueness for Volterra equations of convolution type kernels. The solutions of convolution kernel Volterra equations rely on a function defined through the kernel $k$ called the \textit{resolvent of the kernel $k$}. We define this resolvent as the function $r_{\varepsilon}: \mathbb{R}^+ \to \mathbb{R}$ such that
\begin{equation}
    r_{\varepsilon}(t) = \sum_{j=1}^\infty (-1)^{j-1}k^{*j}_{\varepsilon}(t),
\end{equation}
where the function $k^{*j}_{\varepsilon}(t)$, $j \ge 1$ is the $(j-1)$-fold convolution of the kernel $k_{\varepsilon}$ with itself. We want to show that a perturbed kernel $k_{\varepsilon}$ results in a perturbation of the resolvent. Since $k_{\varepsilon} \in L^1_{\text{loc}}(\mathbb{R}^+)$, the stability theorem for kernels, Theorem 3.1 in \cite{gripenberg1990volterra}, says that the resolvent $r_{\varepsilon} \in L^1_{\text{loc}}(\mathbb{R}^+)$ is unique and depends continuously on $k_{\varepsilon}$ in the $L_{\text{loc}}^1(\mathbb{R}^+)$ topology. As $\displaystyle \sup_{0 \le t \le T} |\varepsilon_2(t)| \Prto[n] 0$, we have that 
\begin{align*}
    \int_0^T |k_{\varepsilon}(t)-k_0(t)| \, \dif t = \int_0^T |\varepsilon_2(t)| \, \dif t \Prto[n] 0,
\end{align*}
so by continuity in $L^1_{\text{loc}}$, we get that
\begin{align*}
    \int_0^T |r_{\varepsilon}(t)-r_0(t)| \, \dif t \Prto[n] 0.
\end{align*}
Using this resolvent, the unique solution to the Volterra equation in \eqref{eq:volterra_stuff} \citep[Theorem 3.5]{gripenberg1990volterra} is give by 
\begin{equation}
    \psi_\varepsilon(t) = f_{\varepsilon}(t) - (r_{\varepsilon} * f_{\varepsilon})(t).
\end{equation}
A simple computation yields that
\begin{align*}
    |\psi_{\varepsilon}(t)-\psi_0(t)| &\le |f_{\varepsilon}(t)-f_0(t)| + \big | \int_0^t (r_0-r_{\varepsilon})(t-s)f_0(s) \, \dif s \big |\\
    & \qquad + \big | \int_0^t r_{\varepsilon}(t-s) [f_0(s)-f_{\varepsilon}(s)] \, \dif s \big |\\
    &\le |f_{\varepsilon}(t)-f_0(t)| + \sup_{0 \le s \le t} |f_0(t)| \int_0^t |(r_0-r_\varepsilon)|(t-s) \, \dif s\\
    & \quad + \sup_{0 \le s \le t} |f_0(s)-f_\varepsilon(s)| \int_0^t |r_\varepsilon(t-s)| \, \dif s.
\end{align*}
Since $h_k(t)$ is bounded, we clearly have that $\sup_{0 \le t \le T} |f_0(t)|$ is bounded. 
Working on the event that $r_{\varepsilon}(t) \in L^1_{\text{loc}}(\mathbb{R}^+)$ (\textit{i.e.} $\int_0^T |r_{\varepsilon}(t)| \, dt$ is bounded),  $\int_0^T |r_{\varepsilon}(t) - r_0(t)| \, dt$ is small, and $\sup_{0 \le t \le T} |\varepsilon_1(t)|$ is small, it follows that 
\begin{align*}
   \sup_{0 \le t \le T} |\psi_{\varepsilon}(t)-\psi_0(t)| &\le \sup_{0 \le t \le T} | \varepsilon_1(t)| + \sup_{0 \le t \le T} |f_0(t)| \int_0^T |r_0-r_{\varepsilon}|(t) \, \dif t\\
   & \quad + \sup_{0 \le t \le T} |\varepsilon_1(t)| \int_0^T |r_{\varepsilon}(t)| \, \dif t.
\end{align*}
Since every term on the RHS is small and the complement of the event on which we proved the inequality above has small probability, the result immediately follows.  
\end{proof}

This yields one of the main theorems of this paper which we restate for clarity (Theorem~\ref{thm: concentration_main}):

\begin{theorem}[Concentration of SGD] \label{thm: concentration} 
Suppose $\beta \in \mathbb{N}$ is a batch-size parameter such that $0 < \beta \le n^{1/5-\delta}$ for some $\delta > 0$ and the stepsize is $\gamma < \frac{2}{r} \big (\int_0^\infty x \, d\mu(x) \big )^{-1}$. Let the constant $T > 0$. Under Assumptions~\ref{assumption: Vector} and \ref{assumption: spectral_density}, the function values at the iterates of SGD converge to 
\begin{equation}
  \sup_{0 \le t \le T}  \big | f \big ( \xx_{ \lfloor \tfrac{n}{\beta} t \rfloor } \big ) -  \psi_0(t) \big | \Prto[n] 0,
\end{equation}
where the function $\psi_0$ is the solution to the Volterra equation
\begin{equation} \begin{gathered} \label{eq:Volterra_eq}
    \psi_0(t) = \tfrac{R}{2} h_1(t) + \tfrac{\widetilde{R}}{2} \big ( rh_0(t)+ (1-r) \big ) + \int_0^t \gamma^2 r h_2(t-s)\psi_0(s)\,\dif s,\\
    \quad \text{and} \quad h_k(t) = \int_0^\infty x^k e^{-2 \gamma tx} \, \dif \mu(x). 
\end{gathered}
\end{equation}
\end{theorem}

\begin{proof}
 By definition of $N_t$ and $\psi_\varepsilon(t)$, we have
\[
f \big ( \xx_{ \lfloor \tfrac{n}{\beta} t \rfloor } \big ) = f \big ( \xx_{N_{\tau_{\lfloor n\beta/t \rfloor}}} \big) = \psi_\varepsilon(\tau_{\lfloor \tfrac{n}{\beta} t \rfloor }).
\]
Also, Proposition \ref{prop:volterrastability} gives
\[
\sup_{0 \leq t \leq T} |\psi_\varepsilon(t) - \psi_0(t)| \Prto[n] 0.
\]
Therefore, triangle inequality gives
\[
\sup_{0 \leq t \leq T} \big | f \big ( \xx_{ \lfloor \tfrac{n}{\beta} t \rfloor } \big ) -  \psi_0(t) \big | 
\le \sup_{0 \leq t \leq T} \big | \psi_\varepsilon(\tau_{\lfloor \tfrac{n}{\beta} t \rfloor }) -  \psi_0(\tau_{\lfloor \tfrac{n}{\beta} t \rfloor }) \big | + 
\sup_{0 \leq t \leq T} \big | \psi_0(\tau_{\lfloor \tfrac{n}{\beta} t \rfloor }) - \psi_0(t) \big|,
\]
and by the continuity of $\psi_0(t)$, it would suffice to show 
\begin{equation}\label{eq:convtau}
\sup_{0 \leq t \leq T} \big| \tau_{\lfloor \tfrac{n}{\beta} t \rfloor } - t \big|
\Prto[n] 0.
\end{equation}
First, note that 
\begin{equation}\label{eq:convpoisson}
\sup_{0\le s \le T} \big| N_s\frac{\beta}{n} - s \big| \Prto[n] 0
\end{equation}
holds. For a fixed time $s\in [0,T]$, this comes from the strong law of large numbers, see \cite[(4.18)]{kingman1993}. And the result for the supremum on $[0,T]$ follows using monotonicity of $N_t$ on $[0,T]$ and the meshing arguments as used in proving Lemma \ref{lem:modelconvergence}.

Now for $t>0$, let $s>0$ be such that $\lfloor \tfrac{n}{\beta} t \rfloor = N_s$, or $t = N_s\frac{\beta}{n} + \frac{\beta}{n}r$ for $0\le r<1$. Therefore, observe
\begin{align*}
\big| \tau_{\lfloor \tfrac{n}{\beta} t \rfloor } - t \big|
 = \big| \tau_{N_s} - N_s\frac{\beta}{n} - \frac{\beta}{n}r \big | 
 &\le 
 \big | \tau_{N_s} - s \big| + \big| N_s\frac{\beta}{n} -s \big| + \frac{\beta}{n}r.
\end{align*}

The last term converges to 0 as $n\to \infty$, and so does the second term in probability, by \eqref{eq:convpoisson}. So, it is left to show the convergence in probability of the first term. By the definition of $N_s$, we have
\[
s-\Delta \le \tau_{N_s} \le s,
\]
where $\Delta$ denotes the largest spacing between adjacent jumps in [0,T]. Note that $N_s \le N_T \le \frac{2Tn}{\beta}$ with overwhelming probability \cite[Prop. 1]{klar2000}. Recalling again that $\tau_k - \tau_{k-1}, k\in\mathbb{N}$, follow $\text{Exp}(\frac{n}{\beta})$ independently on $k\in \mathbb{N}$, we have for $u>0$,
\begin{align*}
    \Pr\bigl((\Delta > u)\cap (N_T \le \frac{2Tn}{\beta})\bigr) &= 1 - \Pr\bigl((\Delta \le u) \cap (N_T \le\frac{2Tn}{\beta})\bigr)\\
    &\le 1 - (1-e^{-\frac{n}{\beta}u})^{N_T}\\
    &\le N_T e^{-\frac{n}{\beta}u} \le \frac{2Tn}{\beta}e^{-\frac{n}{\beta}u}.
\end{align*}
This implies
\[
\Pr\bigl(\Delta > u\bigr) \le \Pr\bigl((\Delta > u)\cap (N_T \le \frac{2Tn}{\beta})\bigr) + \Pr(N_T > \frac{2Tn}{\beta}) \to 0
\]
as $n\to \infty$ and we obtain the claim.
\end{proof}


\subsection{Bounding the errors \texorpdfstring{$\varepsilon^{(n)}_1$}{Lg}} \label{app: bounding_the_error_terms}
In this section, we give a high--level overview of the errors and how they converge to $0$. We will have the following error pieces:
\begin{equation}\label{eq:errorpieces}
\varepsilon^{(n)}_1(t)
\defas 
\varepsilon^{(n)}_{\operatorname{IC}}(t)
+
\varepsilon^{(n)}_{\operatorname{KL}}(t)
+
\varepsilon^{(n)}_{\operatorname{M}}(t)
+
\varepsilon^{(n)}_{\operatorname{beta}}(t)
+
\varepsilon^{(n)}_{\operatorname{eta}}(t).
\end{equation}
We define these terms momentarily and we will verify that $\varepsilon^{(n)}_1$ is indeed equal to these pieces in Lemma \ref{lem:errorpieces}.  We remark that before controlling the errors, we will need to make an \emph{a priori} estimate that (effectively) shows the function values remain bounded.  Thus, we define the stopping time, for any fixed $\theta >0$, by
\begin{equation}\label{eq:vartheta}
\vartheta \defas \inf \left\{ t \ge 0 :  \|\UU\SSigma \nnu_t - \eeta\| > n^{\theta} \right\}.
\end{equation}
We then show:
\begin{lemma}\label{lem:psibound}
For any $\theta >0$, and for any $T >0,$ $\vartheta > T$ with high probability.
\end{lemma}
\noindent This is achieved by a simple martingale-type estimate, which is similar to the standard convergence arguments for SGD. The proof is given in Section \ref{sec:psibound}. We will need it in what follows.  \emph{We will also condition on $\SSigma$ going forward.}

\subsubsection{Errors from the convergence of the initial conditions}
The error $\varepsilon^{(n)}_{\operatorname{IC}}(t)$ arises due to convergence errors in the signal and initialization. It was already essentially discussed in Lemma \ref{lem:modelconvergence}.  We define it by
\begin{equation}\label{eq:errorIC}
\varepsilon^{(n)}_{\operatorname{IC}}(t)
\defas 
\frac{1}{2}\sum_{j=1}^d \sigma_j^{2} e^{-2t\gamma\sigma_j^2}\nu_{0,j}^2
-\frac{Rh_1(t)}{2}.
\end{equation}
It accounts for the convergence of the initialization in the large $d$ limit and relies on the convergence of the empirical spectral distribution.  Due to Lemma \ref{lem:modelconvergence}, we have already shown it converges to $0$.

\subsubsection{Errors which vanish due to the key lemma} \label{app:error_vanish_key_heuristic}
The vanishing of the error $\varepsilon^{(n)}_{\operatorname{KL}}(t)$ is the \emph{key lemma}.  To explain why we call it this: let us specialize to the case of $\eeta =0$ and $\beta = 1$. If we were content to evaluate the \emph{expected} function values, when averaging over the randomness inherent in the SGD algorithm, then this would be the only error that we would need to control.  Thus in some sense, it can be viewed as the minimal estimate that needs to be shown to prove the Volterra equation holds.  This error is given by 
\begin{equation}\label{eq:errorKH}
\varepsilon^{(n)}_{\operatorname{KL}}(t)
\defas 
\frac{1}{2}\sum_{j=1}^d 
    \sigma_j^{2}
    \int_0^t e^{-2(t-s)\gamma\sigma_j^2}
 \biggl(\sum_{i=1}^n
    \biggl(\bigl(\mathbb{e}_j^T
    \UU^T 
    \mathbb{e}_{i}
    \bigr)^2-\frac{1}{n}\biggr)
    \bigl(
    \mathbb{e}_{i}^T
    (\UU\SSigma \nnu_s - \eeta)\bigr)^2\biggr) \, \dif s.
\end{equation}
After interchanging the order of summation, it suffices to show:
\begin{lemma}[Key lemma] \label{prop:errorKH}
For any $T > 0$ and for any $\epsilon > 0,$ with overwhelming probability
\[
\max_{1 \leq i \leq n}
\max_{0 \leq t \leq T}
\biggl|
\sum_{j=1}^d 
    \sigma_j^{2}
    e^{-2t\gamma\sigma_j^2}
    \biggl(\bigl(\mathbb{e}_j^T
    \UU^T 
    \mathbb{e}_{i}
    \bigr)^2-\frac{1}{n}\biggr)
\biggr|
\leq n^{\epsilon-1/2}.
\]
\end{lemma}
\noindent This we show in Section \ref{sec:uboundsapplications}. Note that by combining this with \eqref{eq:errorKH} and Lemma \ref{lem:psibound}, we conclude that for any $\epsilon > 0$ and any $T,$ with high probability
\[
\max_{0 \leq t \leq T}|\varepsilon^{(n)}_{\operatorname{KL}}(t)|
\leq n^{2\epsilon-1/2}\int_0^T \, \frac{1}{2} \, \dif s \to 0.
\]

\subsubsection{Martingale errors}

The martingale errors are due to the randomness in the algorithm itself.  They in part are small because the singular vector matrix $\UU$ is delocalized, in that its offdiagonal entries in any fixed orthogonal basis are $n^{\epsilon-1/2}$ with overwhelming probability.  The martingale errors are given by
\begin{equation}\label{eq:errorM}
\varepsilon^{(n)}_{\operatorname{M}}(t)
\defas
\frac{1}{2}\sum_{j=1}^d 
    \sigma_j^{2}
    \int_0^t e^{-2(t-s)\gamma\sigma_j^2} \, \dif M_{s,j}.
\end{equation}

Estimating this error requires a substantial build-up.  The most important technical input, which we will use in multiple places, is that the function values do not concentrate too heavily in any coordinate direction.  As an input, we will use Lemma \ref{lem:psibound}, and so we work with the stopped process defined for any $t \geq 0$ by $\nnu_t^\vartheta \defas \nnu_{t \wedge \vartheta}$. In some sense, this is the most challenging and important technical statement that we prove:
\begin{proposition}\label{prop:delocalization}
For any  $T>0$, any $\epsilon >0$, there is a sufficiently small $\theta >0$ so that 
\[
\sup_{0 \leq t \leq T}
\sup_{1 \leq i \leq n}
\bigl(\mathbb{e}_{i}^T(\UU\SSigma \nnu^\vartheta_t - \eeta)\bigr)^2
\leq \beta n^{\epsilon-1}
\]
with overwhelming probability.
\end{proposition}
\noindent We expect that the upper bound on $\beta$ in Theorem~\ref{thm: concentration} is a limitation of our method, and that similar statements should hold for larger $\beta$. This proposition is proven in Section \ref{sec:mglesapplication}.

With Proposition \ref{prop:delocalization} in hand, we can then bound the martingale errors.
\begin{proposition}\label{prop:errorM}
For any $T>0,$ with overwhelming probability,
\[
\sup_{0 \leq t \leq T} 
|\varepsilon^{(n)}_{\operatorname{M}}(t \wedge \vartheta)|
\Prto[n] 0.
\]
\end{proposition}
\noindent This is proven in Section \ref{sec:mglesapplication}.  Having done Proposition \ref{prop:delocalization}, the proof of Proposition \ref{prop:errorM} reduces to standard martingale techniques.

\subsubsection{Errors due to minibatching}
The Volterra equation that we prove \eqref{eq:Volterra_eq_main} importantly does not depend on the minibatching size.  Naturally, the dynamics do depend on $\beta,$ and so there are error terms which must be controlled and which are in part small due to the minibatching parameter $\beta$ satisfying $\beta/n \to 0$.  These errors are given by
\begin{equation}\label{eq:errorbeta}
\varepsilon^{(n)}_{\operatorname{beta}}(t)
\defas
\frac{1}{2}\sum_{j=1}^d 
    \sigma_j^{2}
    \int_0^t e^{-2(t-s)\gamma\sigma_j^2}
    \biggl(\frac{\beta-1}{n-1}
    \mathcal{B}_{s,j}^2
    -\frac{\beta-1}{n-1}
    \sum_{i=1}^n
    \bigl(\mathbb{e}_j^T
    \UU^T 
    \mathbb{e}_{i}
    \bigr)^2
    \bigl(
    \mathbb{e}_{i}^T
    (\UU\SSigma \nnu_s - \eeta)\bigr)^2\biggr) \, \dif s.
\end{equation}
\noindent Note in particular that when $\beta=1$ this vanishes identically.  

Much of this error term is controlled using delocalization of $\UU$ and Lemma \ref{lem:psibound}.  However, there is one error $\beta (\mathcal{B}_{s,j})^2$ which requires extra work.  This we would like to tend to $0$ in a sufficiently strong sense.  On consideration of \eqref{eq:Bt}, we see that this in turn requires that $\nnu_{s}$ itself be delocalized in the sense that $|\nnu_{s,j}| \leq n^{\epsilon -1/2}$ with overwhelming probability.

\begin{proposition}\label{prop:nut}
For any $\epsilon > 0$ and $T>0,$ with overwhelming probability,
\[
\sup_{0 \leq t \leq T} \max_{1 \leq j \leq d} |\nu^\vartheta_{t,j}| \leq \beta n^{\epsilon-1/2}.
\]
\end{proposition}
\noindent The dependence on $\beta$ is only through Proposition \ref{prop:delocalization}, on which this relies.
The proof is found in Section \ref{sec:mglesapplication}. Now Proposition \ref{prop:nut} with eigenvector delocalization gives the following proposition.

\begin{proposition}\label{prop:errorbeta}
For any $\epsilon > 0$ and any $T>0,$ with overwhelming probability,
\[
\sup_{0 \leq t \leq T}
|\varepsilon^{(n)}_{\operatorname{beta}}(t)|
\leq n^{\epsilon - 1/2}.
\]
\end{proposition}
\noindent This is proven in Section \ref{sec:propsixteen}.

\subsubsection{Errors due to the model noise} \label{app:errors_due_to_the_model_noise}
Finally, there are errors that arise due to the noise $\eeta$. The model noise $\eeta$ in fact induces a change in the dynamics of the algorithm. This change is reflected in an additional forcing term that appears in the Volterra equation.  This forcing term is controlled (in some sense) by the mean behavior of $\nu_{t,j}$.  The model noise error is defined by
\begin{equation}\label{eq:erroreta}
\varepsilon^{(n)}_{\operatorname{eta}}(t)
\defas
\sum_{j=1}^d
\int_0^t 
e^{-2(t-s)\gamma\sigma_j^2}
\gamma \sigma_j^3\nu_{s,j}(\UU^T\eeta)_j \dif s
+
\frac{1}{2}\|\eeta\|^2 -
\sum_{j=1}^{n \wedge d} \sigma_j \nu_{t,j} (\UU^T\eeta)_j 
-
{\widetilde{R}}\,\cdot\frac{rh_0(t) +(1-r)}{2}.
\end{equation}

The fundamental identity that needs to be shown here is that averages of $\nu_{s,j}(\UU^T\eeta)_j$ converge.  Within $\varepsilon^{(n)}_{\operatorname{eta}}(t)$ there are many such averages, and so we formulate a general claim to this effect.
\begin{proposition}\label{prop:noisecovariance}
Let $\{ c_j \}_1^n$ be a deterministic sequence with $|c_j| \leq 1$ for all $j$ and define $\vartheta$ as in Lemma~\ref{lem:psibound}.  Then for any $t > 0$ and any $\epsilon > 0,$
\[
\biggl|\sum_{j=1}^n c_j \sigma_j \nu_{t,j}^\vartheta (\UU^T\eeta)_j
-
\frac{\|\eeta\|^2}{n}\sum_{j=1}^n c_j (1-e^{-(t \wedge \vartheta) \gamma \sigma_j^2})
\biggr|
\leq \|\eeta\|^2 \sqrt{\beta} n^{\epsilon-1/2},
\]
with overwhelming probability.
\end{proposition}
\noindent This is proven in Section \ref{sec:mglesapplication}.

Using a mesh argument, and appealing to the convergence of the empirical spectrum, we can then show that $\varepsilon^{(n)}_{\operatorname{eta}}(t)$ tends to $0$.
\begin{proposition}\label{prop:erroreta}
For any $T > 0$,
\[
\max_{0 \leq t \leq T}
|\varepsilon^{(n)}_{\operatorname{eta}}(t)|
\Prto[n] 0.
\]
\end{proposition}
\noindent This is proven in Section \ref{sec:propeighteen}.

\subsubsection{Verification of \texorpdfstring{\eqref{eq:errorpieces}}{Lg}}

The combination of Propositions \ref{prop:errorKH}, \ref{prop:errorM}, \ref{prop:errorbeta}, and \ref{prop:erroreta} together show all remaining errors are small in \eqref{eq:errorpieces}.  Before proceeding, we connect \eqref{eq:errorpieces} to the previous sections to demonstrate this is truly the sum of errors that must be controlled.  

\begin{lemma}\label{lem:errorpieces}
Equation \eqref{eq:errorpieces} holds.
\end{lemma}
\begin{proof}
We recall the approximate Volterra equation in \eqref{eq:psivolterra}.  The error $\varepsilon^{(n)}_{1}(t)$ is defined implicitly in \eqref{eq:approxvolterra}.  By using \eqref{eq:errorIC} and \eqref{eq:erroreta}, we conclude that
\begin{equation}\label{eq:errorpieces2}
\varepsilon^{(n)}_{1}(t)
-
\varepsilon^{(n)}_{\operatorname{IC}}(t)
-
\varepsilon^{(n)}_{\operatorname{eta}}(t)
=
\frac{1}{2}\sum_{j=1}^d 
    \sigma_j^{2}
    \int_0^t e^{-2(t-s)\gamma\sigma_j^2} \dif \mathcal{E}_{s,j}.
\end{equation}
Recall from \eqref{eq:totalerror} and \eqref{eq:doob}
\[
\mathcal{E}_{t,j} =
\int_0^t \mathcal{A}_{s,j}\, \dif s + M_{t,j}
- \int_0^t \widehat{\mathcal{A}}_{s,j}\, \dif s.
\]
Using \eqref{eq:Atj2} and \eqref{eq:Ahat}, we can express
\begin{equation*}
\begin{aligned}
\dif \mathcal{E}_{t,j} 
&= \dif M_{t,j} \\
    &
    +
    \biggl(\frac{\beta-1}{n-1}
    \bigl(
    \mathcal{B}_{t,j}
    \bigr)^2
    -\frac{\beta-1}{n-1}
    \sum_{i=1}^n
    \bigl(\mathbb{e}_j^T
    \UU^T 
    \mathbb{e}_{i}
    \bigr)^2
    \bigl(
    \mathbb{e}_{i}^T
    (\UU\SSigma \nnu_t - \eeta)\bigr)^2\biggr) \dif t \\
    &+
    \biggl(\sum_{i=1}^n
    \biggl(\bigl(\mathbb{e}_j^T
    \UU^T 
    \mathbb{e}_{i}
    \bigr)^2-\frac{1}{n}\biggr)
    \bigl(
    \mathbb{e}_{i}^T
    (\UU\SSigma \nnu_t - \eeta)\bigr)^2\biggr) \dif t.
\end{aligned}
\end{equation*}
Each of these three lines, on substituting into \eqref{eq:errorpieces2}, produces $\varepsilon^{(n)}_{\operatorname{M}}(t),$
$\varepsilon^{(n)}_{\operatorname{beta}}(t),$
and 
$\varepsilon^{(n)}_{\operatorname{KL}}(t),$
respectively.
\end{proof}

\subsubsection{Proof organization}
We organize the remainder of the proof as follows.  We begin by proving Lemma \ref{lem:psibound} in Section \ref{sec:psibound} using standard martingale techniques.  Arguments along this line are well--known in the context of analysis of SGD, and this argument is similar (and in fact easier) than convergence arguments for SGD.

In Section \ref{sec:ubounds}, we introduce standard machinery for concentration of Lipschitz functions on the orthogonal group.  In this section, we then make the error estimates that follow from this type of estimate.   In particular we prove that the key lemma, Lemma~\ref{prop:errorKH} holds.  We also show Proposition \ref{prop:errorbeta} and Proposition \ref{prop:erroreta} hold.  Note these latter propositions depend on some estimates that require other martingale techniques.

In Section \ref{sec:mgles}, we give the estimates that depend heavily on martingale concentration techniques.  In Section \ref{sec:mglesabstract}, we outline the general martingale concentration techniques that we need.  These extend general martingale techniques in ways that are appropriate to our setting.  In Section \ref{sec:mglesapplication}, we prove the remaining propositions, beginning with the main technical proposition Proposition \ref{prop:delocalization}.
We then give the bounds that prove Propositions~\ref{prop:errorM}, \ref{prop:nut}, and \ref{prop:noisecovariance}.

\subsection{ An \emph{a priori} bound for the objective function values} \label{sec:psibound}

Here we combine some of the estimates already developed to give a simple starting bound for the function values $\psi_{\varepsilon}(t)$ in the proof of Lemma \ref{lem:psibound}.  We will need this starting bound in many of our future estimates.  We will do this by constructing an appropriate supermartingale, which we then use to control the evolution of $\psi_{\varepsilon}$.

\begin{proof}[Proof of Lemma \ref{lem:psibound}]
For convenience, we will set $\psi_{\varepsilon} = \psi$. We recall from \eqref{eq:psidef} that
\begin{equation*}
\begin{aligned}
\psi(t)
&=
\frac{1}{2}
\sum_{j=1}^d \sigma_j^{2}\nu_{t,j}^2
-\sum_{j=1}^{n \wedge d} \sigma_j\nu_{t,j}\, (\UU^T\eeta)_j 
+\frac{1}{2}\|\eeta\|^2.
\end{aligned}
\end{equation*}
Hence using \eqref{eq:Atj} and \eqref{eq:Bt},
\[
\begin{aligned}
\lim_{\epsilon \downarrow 0} \, 
&\epsilon^{-1}
\mathbb{E}
[
\psi(t+\epsilon)-\psi(t)
~\vert~\mathcal{F}_t
]
=\frac{1}{2}
\sum_{j=1}^d \sigma_j^{2}\mathcal{A}_{t,j}
-\sum_{j=1}^{n \wedge d} \sigma_j\mathcal{B}_{t,j}\, (\UU^T\eeta)_j \\
&= 
\frac{1}{2}
\sum_{j=1}^d
\sigma_j^2\biggl(
-2\nu_{t,j}\gamma
   \mathbb{e}_j^T\SSigma^T
    \UU^T 
    (\UU\SSigma \nnu_t - \eeta) 
    +\gamma^2\frac{n}{\beta}
    \E
    \biggl(\mathbb{e}_j^T\SSigma^T
    \UU^T 
    \PP
    (\UU\SSigma \nnu_t - \eeta)
    ~\bigg\vert~ \mathcal{F}_t
    \biggr)^2
\biggr)\\
&
+\gamma
\sum_{j=1}^n 
(\eeta^T \UU \SSigma \mathbb{e}_j)
\mathbb{e}_j^T
\SSigma^T
(\SSigma
\nnu_t - \UU^T\eeta) \\
&=
-\gamma
(\SSigma
\nnu_t - \UU^T\eeta)^T
\SSigma \SSigma^T
(\SSigma
\nnu_t - \UU^T\eeta) 
+ \gamma^2\frac{n}{2\beta}
\sum_{j=1}^n
    \E
    \biggl(\mathbb{e}_j^T\SSigma \SSigma^T
    \UU^T 
    \PP
    (\UU\SSigma \nnu_t - \eeta)
    ~\bigg\vert~ \mathcal{F}_t
    \biggr)^2.
\end{aligned}
\]
Using Lemma \ref{lem:subsetsum}, we can give the expression
\[
\begin{aligned}
\lim_{\epsilon \downarrow 0}
&\epsilon^{-1}
\mathbb{E}
[
\psi(t+\epsilon)-\psi(t)
~\vert~\mathcal{F}_t
]
=
-\gamma
(\SSigma
\nnu_t - \UU^T\eeta)^T
\SSigma \SSigma^T
(\SSigma
\nnu_t - \UU^T\eeta)  \\
&+ \gamma^2\frac{\beta-1}{2(n-1)}
\sum_{j=1}^d
    \biggl(\mathbb{e}_j^T\SSigma \SSigma^T
    \UU^T 
    (\UU\SSigma \nnu_t - \eeta)
    \biggr)^2 \\
&+ \gamma^2\frac{1}{2}\biggl(1-\frac{\beta-1}{n-1}\biggr)
\sum_{j=1}^d
\sum_{i=1}^n
    \biggl(\mathbb{e}_j^T\SSigma \SSigma^T
    \UU^T 
    \mathbb{e}_i\biggr)^2
    \biggl(\mathbb{e}_i^T
    (\UU\SSigma \nnu_t - \eeta)
    \biggr)^2.
\end{aligned}
\]
All three terms have the interpretation as a quadratic form $\xx^T \widehat{\AA} \xx$ for some matrix $\widehat{\AA}$ and the vector $\xx = \UU\SSigma \nnu_t - \eeta$. Specifically, we have
\[
\widehat{\AA} = 
-\gamma \UU \SSigma \SSigma^T \UU^T
+\frac{\gamma^2(\beta-1)}{2(n-1)}
\UU\SSigma\SSigma^T
\SSigma \SSigma^T\UU^T
+\frac{\gamma^2}{2}\biggl(1-\frac{\beta-1}{n-1}\biggr)
\sum_{i=1}^n 
\mathbb{e}_i \mathbb{e}_i^T \| \SSigma \SSigma^T\UU^T \mathbb{e}_i\|^2.
\]
As $\SSigma$ is bounded, we can let $\rho_\star$ be the largest eigenvalue of $\widehat{\AA},$ which is symmetric, and which can be bounded solely in terms of the norm of $\SSigma$.  Then we conclude that
\[
\lim_{\epsilon \downarrow 0}
\epsilon^{-1}
\mathbb{E}
[
\psi(t+\epsilon)-\psi(t)
~\vert~\mathcal{F}_t
]
\leq 2\rho_\star \psi(t).
\]
It follows immediately that
\[
X_t \defas e^{-2\rho_\star t}\psi(t)
\]
is a positive supermartingale.  Hence by optional stopping, for any $T > 0$
\[
\Pr[ \sup_{0 \leq t \leq T} X_t \geq \lambda ~\vert~ \mathcal{F}_0] 
\leq \frac{\psi(0)}{\lambda}.
\]
Hence,
\[
\Pr[ \sup_{0 \leq t \leq T} \psi(t) \geq \lambda ~\vert~ \mathcal{F}_0] 
\leq \frac{\psi(0)e^{2\rho_\star T}}{\lambda}.
\]
Taking $\lambda=n^{2\theta}/2$ completes the proof.
\end{proof}

\section{Estimates based on concentration of measure on the high--dimensional orthogonal group}
\label{sec:ubounds}

\subsection{Generalities}
We recall a few properties of Haar measure on the orthogonal group.  We endow the orthogonal group $O(n)$ with the metric given by the Frobenius norm, so that $d(\OO,\UU) = \|\OO-\UU\|_F$.  Say that a function $F : O(n) \to \RR$ is Lipschitz with constant $L$ if 
\[
|F(\OO)-F(\UU)| \leq L \|\OO-\UU\|_F.
\]
Recall that the orthogonal group can be partitioned as two disconnected copies of the special orthogonal group $\operatorname{SO}(n),$ which we endow with the same metric.  These are given as the preimages of $\{\pm 1\}$ under the determinant map.
Haar measure on the special orthogonal group enjoys a strong concentration of measure property.
\begin{theorem}
    \label{thm:lipschitz}
    Suppose that $F : \operatorname{SO}(n) \to \RR$ is Lipschitz with constant $L$.  Then for all $t \geq 0,$
    \[
    \Pr[
    |F(\UU)-\E[F(\UU)]| > t
    ] \leq 2e^{-c n t^2/L^2},
    \]
    where $c>0$ is an absolute constant.
\end{theorem}
\noindent See \cite[Theorem 5.2.7]{vershynin2018high} or \cite[Theorem 5.17]{Meckes} for precise constants.
We can derive concentration for even functions of the orthogonal group automatically:
\begin{corollary}
    \label{cor:lipschitz}
    Suppose that $F : O(n) \to \RR$ is Lipschitz with constant $L$ and suppose that 
    \[
    \E[ F(\UU)\det \UU]
    =\E [F(\UU)] = 0.
    \]
    Then for all $t \geq 0,$
    \[
    \Pr[
    |F(\UU)-\E[F(\UU)]| > t
    ] \leq 2e^{-c n t^2/L^2},
    \]
    where $c>0$ is an absolute constant.
\end{corollary}
\begin{proof}
Under the assumption, the mean of $F$ conditioning on either $\det \UU = 1$ or $\det \UU = -1$ is $0$.  Hence, by conditioning, we achieve the desired concentration around $0,$ which is the mean $\E[F(\UU)].$
\end{proof}

As a useful illustration, an entry of $\UU$, which is a Haar--distributed random matrix on $O(n)$ is concentrated.
\begin{corollary}\label{cor:entry}
For any $n > 1,$ there is an absolute constant $c>0$ so that for all $t \geq 0$ and all $i,j$ in $\{1,2,...,n\}$
\[
\Pr[ |U_{ij}| > t] \leq 2e^{-cnt^2}.
\]
The same statement holds for any generalized entry $\xx^T \UU \yy$ for fixed unit vectors $\xx,\yy$, i.e.
\[
\Pr[ |\xx^T \UU \yy| > t] \leq 2e^{-cnt^2}.
\]
\end{corollary}
\begin{proof}
The entry map $\UU \mapsto U_{ij}$ is $1$--Lipschitz.  Moreover, it has mean $0,$ restricted to either component, as
\[
\E[ U_{ij} \det(\UU)] = 0 = \E[U_{ij}].
\]
Note that by distributional invariance, negating row $i$ of $\UU$ leaves the distribution of $\UU$ invariant.  Doing so shows the second equality.  Negating any row except for $i$ (which exists as $n > 1$) shows the first equality.

For the generalized entry, by the linearity of the expectation,
\[
\E[ \xx^T \UU \yy \det(\UU)]
=\sum_{i,j} x_iy_j \E [U_{ij} \det(\UU)]
=0
=\sum_{i,j} x_iy_j \E[U_{ij}]
=\E[ \xx^T \UU \yy].
\]
Using that $\UU \mapsto \xx^T \UU \yy$ is $1$--Lipschitz, the proof follows.
\end{proof}

\subsection{Applications to the Volterra equation errors}\label{sec:uboundsapplications}

More to the point, we need concentration of random combinations of functions weighted by entries of $\UU.$
\begin{lemma}\label{lem:FU}
Let $T > 0$ be fixed, and suppose that $\{ g_i\}$ are functions from $[0,T] \to \RR$ which are bounded by $L>0$ and have Lipschitz constant $1$. 
Then, for any $\epsilon > 0,$ and any fixed unit vectors $\aa$ and $\bb$ in $\mathbb{R}^n$, with overwhelming probability
\[
\max_{0 \leq t \leq T}
\biggl|
\sum_{j=1}^n 
    g_j(t)
    \bigl((\aa^T \UU)_j (\bb^T \UU)_j
    -\E [(\aa^T \UU)_j (\bb^T \UU)_j]\bigr)
\biggr|
\leq n^{\epsilon-1/2}.
\]
\end{lemma}
\begin{proof} We first prove the claim for a fixed $t\in[0,T]$ and  generalize the result for any $t\in[0,T]$ later using a mesh points argument. In proving this, we can take advantage of Corollary \ref{cor:lipschitz}. For $t\in [0,T]$, let $F_t: O(n) \to \mathbb{R}$ be
\begin{equation}
    F_t(\UU) \defas  \sum_{j=1}^n g_j(t)
    \bigl((\aa^T \UU)_j (\bb^T \UU)_j
    -\E [(\aa^T \UU)_j (\bb^T \UU)_j]\bigr).
\end{equation}
We can show that $F_t$ is a Lipschitz function on $O(n)$. Indeed, for $\UU,
\VV \in O(n)$, 
\[\mathbb{E}[(\aa^T\UU)_j (\bb^T\UU)_j]=\mathbb{E}[(\aa^T\VV)_j (\bb^T\VV)_j] = \frac{1}{n} \langle \aa, \bb \rangle \] 
and 
\begin{align*}
    |F_t&(\UU) - F_t(\VV)| \\
    &= \bigl|\sum_{j=1}^n g_j(t) 
    \bigl((\aa^T \UU)_j-(\aa^T \VV)_j\bigr)(\bb^T \UU)_j + \sum_{j=1}^n g_j(t)(\aa^T \VV)_j
    \bigl( (\bb^T \UU)_j - (\bb^T \VV)_j\bigr) \bigr|\\
    &\le \sqrt{\sum_{j=1}^n (\aa^T(\UU-\VV))^2}\sqrt{\sum_{j=1}^n g_j^2(t)(\bb^T \UU)_j^2}
    + \sqrt{\sum_{j=1}^n (\bb^T(\UU-\VV))^2}\sqrt{\sum_{j=1}^n g_j^2(t)(\aa^T \VV)_j^2}\\
    &\le L\|\aa^T(\UU-\VV)\|_2 \|\bb^T\UU\|_2 + L\|\bb^T(\UU-\VV)\|_2\|\aa^T\VV\|_2\\
    &\le 2L\|\UU-\VV\|_F.
\end{align*}
Therefore,  we conclude that $F_t$ is a Lipschitz function of Lipschitz constant $2L$. For $j\in\{1,\cdots,n\}$, let 
\[f_j\defas (\aa^T \UU)_j (\bb^T \UU)_j
    -\E [(\aa^T \UU)_j (\bb^T \UU)_j].
\]
Then conditioning on $\det \UU = 1$ and $\det\UU = -1$, negating any column of $\UU$ leaves the distribution of $f_j$ invariant, which gives 
\[
\E[ f_j(\UU)\det \UU ] = \E [f_j(\UU)] = 0,
\]
and thus, using linearity, 
\[
    \E[ F_t(\UU)\det \UU]
    =\E [F_t(\UU)] = 0.
    \]
Now Corollary \ref{cor:lipschitz} gives, for $s\ge 0$,
\[
    \Pr[
    |F_t(\UU)-\E[F_t(\UU)]| > s
    ] \leq 2e^{-c n s^2},
\]
where $c>0$ is an absolute constant. Or, replacing $s=n^{\epsilon-1/2}$ gives the claim for a fixed time $t\in[0,T]$.

Now we generalize the result to any time in $[0,T]$. Assume that the claim is attained for mesh points on $[0,T]$ with arbitrarily small spacing, say $\lambda$.  Then for any $t\in[0,T]$, there exists a mesh point $t_0 \in [0,T]$ such that $|t-t_0| \le \lambda$. The assumption that $\{g_j\}$ are Lipschitz functions with Lipschitz constant 1 implies that $|g_j(t) - g_j(t_0)| \le |t-t_0| \le \lambda.$ Then we see
\begin{align*}
\big|&\sum_{j=1}^n (g_j(t)-g_j(t_0))
    \bigl((\aa^T \UU)_j (\bb^T \UU)_j
    -\E [(\aa^T \UU)_j (\bb^T \UU)_j]\bigr)\big|\\ 
    &\le 
    \sum_{j=1}^n |g_j(t)-g_j(t_0)|
    \big|\bigl((\aa^T \UU)_j (\bb^T \UU)_j
    -\E [(\aa^T \UU)_j (\bb^T \UU)_j]\bigr)\big| 
    \le 2\lambda n.
\end{align*}
Note that $\lambda$ can be arbitrarily small. Thus we have
\begin{align*}
\biggl|
\sum_{j=1}^n 
    g_j(t)
    \bigl((\aa^T \UU)_j (\bb^T \UU)_j
    &-\E [(\aa^T \UU)_j (\bb^T \UU)_j]\bigr)
\biggr|\\
&\le 
\biggl|
\sum_{j=1}^n 
    g_j(t_0)
    \bigl((\aa^T \UU)_j (\bb^T \UU)_j
    -\E [(\aa^T \UU)_j (\bb^T \UU)_j]\bigr)
\biggr|\\
& +  \biggl|
\sum_{j=1}^n 
    (g_j(t) - g_j(t_0))
    \bigl((\aa^T \UU)_j (\bb^T \UU)_j
    -\E [(\aa^T \UU)_j (\bb^T \UU)_j]\bigr)
\biggr|\\
&\le n^{\epsilon-1/2} + 2n\lambda < n^{\epsilon-1/2}
\end{align*}
with overwhelming probability and with small enough $\epsilon$ in the last part. All in all, we have with overwhelming probability
\[
\max_{0 \leq t \leq T}
\biggl|
\sum_{j=1}^n 
    g_j(t)
   \bigl((\aa^T \UU)_j (\bb^T \UU)_j
    -\E [(\aa^T \UU)_j (\bb^T \UU)_j]\bigr)
\biggr|
\leq n^{\epsilon-1/2}.
\]
\end{proof}

\noindent As $\epsilon > 0$ is arbitrary, Lemma \ref{prop:errorKH} follows immediately.
\begin{proof}[Proof of Lemma~\ref{prop:errorKH}]
We just need to apply Lemma \ref{lem:FU} with 
\[
g_j(t) = \sigma_j^2 e^{-2\gamma \sigma_j^2 t} \qquad \text{and} \qquad \aa = \bb=\mathbb{e}_i.
\]
Note that we are conditioning on $\SSigma$ so that the expectation in the statement of Lemma \ref{lem:FU} is only taken over $\UU$. By the boundedness of $\sigma_j,$ by dividing by a sufficiently large constant depending on $\SSigma,$ the Lemma applies.
\end{proof}

\subsection{ Control of the beta errors}\label{sec:propsixteen}
Next, we prove Proposition~\ref{prop:errorbeta} provided that Proposition \ref{prop:nut} holds.
\begin{proof}[Proof of Proposition \ref{prop:errorbeta}]
For $t\in [0,T]$, by equation \eqref{eq:Bt}, we have
\[
(\mathcal{B}_{s,j})^2 = (-\gamma \sigma_j^2 \nu_{t,j}^\vartheta + \gamma \sigma_j(\UU^T\eeta)_j)^2
\le 2(\gamma^2\sigma_j^4\beta + \gamma^2\sigma_j^2)n^{2\epsilon-1},
\]
with overwhelming probability. Here Proposition \ref{prop:nut} was used to bound $\nu_{t,j}^\vartheta$, and Assumption \ref{assumption: Vector} and Corollary \ref{cor:entry} imply $|(\UU^T\eeta)_j| \le n^{\epsilon-1/2}$ w.o.p.
by observing
\[
\bigl( (\UU^T \eeta)_j :1 \leq j \leq n \bigr) \law \| \eeta\| \bigl( U_{1,j} :1 \leq j \leq n \bigr).
\]
We recall the definition of $\vartheta$ from Lemma~\ref{lem:psibound} as 
\[ \vartheta = \inf \big \{ t \ge 0 \, : \, \|\UU \SSigma \nnu_t- \eeta\| > n^{\theta} \big \}, \]
where $\theta <\epsilon/2$. By applying Corollary \ref{cor:entry} and the definition of $\vartheta$, we get
\[\sum_{i=1}^n
    \bigl(\mathbb{e}_j^T
    \UU^T 
    \mathbb{e}_{i}
    \bigr)^2
    \bigl(
    \mathbb{e}_{i}^T
    (\UU\SSigma \nnu_t^\vartheta - \eeta)\bigr)^2 \le n^{\epsilon-1}\|\UU\SSigma \nnu_t^\vartheta - \eeta\|^2
    \le n^{2\epsilon -1},
\]
with overwhelming probability. Therefore,
\begin{align*}
     |&\varepsilon^{(n)}_{\operatorname{beta}}(t)| \\
     &= \frac{1}{2}\biggl| \sum_{j=1}^d 
    \sigma_j^{2}
    \int_0^t e^{-2(t-s)\gamma\sigma_j^2}
     \biggl(\frac{\beta-1}{n-1}
    \bigl(
    \mathcal{B}_{s,j}
    \bigr)^2
    -\frac{\beta-1}{n-1}
    \sum_{i=1}^n
    \bigl(\mathbb{e}_j^T
    \UU^T 
    \mathbb{e}_{i}
    \bigr)^2
    \bigl(
    \mathbb{e}_{i}^T
    (\UU\SSigma \nnu_s - \eeta)\bigr)^2\biggr) \dif s \biggr|\\
    &\le  \frac{1}{2}\sum_{j=1}^d 
    \sigma_j^{2}
   \left ( \int_0^t e^{-2(t-s)\gamma\sigma_j^2} \dif s \right )
    \left( \frac{\beta}{n} \cdot 2 \big (\gamma^2 \sigma_j^4  + \gamma^2\sigma_j^2 \big ) n^{2\epsilon-1} + \frac{\beta}{n}\cdot n^{2\epsilon-1} \right )\\
    &= \frac{1}{2}\sum_{j=1}^d \frac{1}{2\gamma} \cdot \frac{\beta}{n} \cdot n^{2\epsilon-1} (2\gamma^2\sigma_j^4 + 2\gamma^2\sigma_j^2 + 1)\\
    & \le n^{\epsilon-1/2},
\end{align*}
with small enough $\epsilon$ in the last line, given our assumption on $\beta \le n^{1/5-\epsilon}$.
\end{proof}

\subsection{ Control of the eta errors}\label{sec:propeighteen}
We now give the proof of Proposition \ref{prop:erroreta} provided that Proposition \ref{prop:noisecovariance} holds.
\begin{proof}[Proof of Proposition \ref{prop:erroreta}]

Note that the proof of Proposition \ref{prop:noisecovariance} is based on conditioning on $\SSigma$. Therefore, Proposition \ref{prop:noisecovariance} by substituting $c_j = \int_0^t 
e^{-2(t-s) \gamma\sigma_j^2} \gamma \sigma_j^2 \dif s$ and $c_j = 1$, respectively, implies with overwhelming probability
\[
\biggl|
\sum_{j=1}^d \int_0^t 
e^{-2(t-s) \gamma\sigma_j^2}
\gamma \sigma_j^3\nu_{s,j}^\vartheta(\UU^T\eeta)_j
- \frac{\|\eeta\|^2}{n}\sum_{j=1}^d
\int_0^t 
e^{-2(t-s) \gamma\sigma_j^2}\gamma \sigma_j^2(1-e^{-(s\wedge \vartheta) \gamma \sigma_j^2}) 
\dif s
\biggr|
\le \|\eeta\|^2\sqrt{\beta} n^{\epsilon -1/2}
\]
and
\[
\biggl| \sum_{j=1}^{n \wedge d} \sigma_j \nu_{t,j}^\vartheta (\UU^T\eeta)_j - \frac{\|\eeta\|^2}{n}\sum_{j=1}^{n\wedge d} (1-e^{-(t\wedge \vartheta) \gamma \sigma_j^2}) \biggr| \le \|\eeta\|^2\sqrt{\beta} n^{\epsilon -1/2}.
\]
 Therefore, it suffices to show

\begin{equation}\label{eq:etaerrorapprox}
\frac{\|\eeta\|^2}{n} \sum_{j=1}^d \int_0^t 
e^{-2(t-s)\gamma\sigma_j^2}
\gamma \sigma_j^2(1-e^{-s \gamma \sigma_j^2}) 
\dif s
+\frac{1}{2}\|\eeta\|^2 
- \frac{\|\eeta\|^2}{n}\sum_{j=1}^{n\wedge d} (1-e^{-t \gamma \sigma_j^2})
-
{\widetilde{R}}\,\cdot\frac{rh_0(t) +(1-r)}{2}
\end{equation}
converges to 0 in probability as $n\to \infty$.


Observe
\begin{align*}
\frac{\|\eeta\|^2}{n} \sum_{j=1}^d \int_0^t 
e^{-2(t-s) \gamma\sigma_j^2}&
\gamma \sigma_j^2(1-e^{-s \gamma \sigma_j^2}) 
\dif s \\
&= \frac{\|\eeta\|^2}{n} 
\sum_{j=1}^d \gamma\sigma_j^2 e^{-2t\gamma\sigma_j^2} \int_{0}^t e^{2s\gamma\sigma_j^2}(1-e^{-s\gamma\sigma_j^2}) \dif s\\
&= \frac{\|\eeta\|^2}{n} 
\sum_{j=1}^d e^{-2t\gamma\sigma_j^2} \biggl[ \frac{1}{2}(e^{2t\gamma\sigma_j^2}-1) - (e^{t\gamma\sigma_j^2 }-1) \biggr]\\
&= \frac{\|\eeta\|^2}{n}
\biggl[
\frac{d}{2} + \frac{1}{2}\sum_{j=1}^d e^{-2\gamma\sigma_j^2 t} - \sum_{j=1}^d e^{-t\gamma\sigma_j^2}
\biggr].
\end{align*}
Note that $\frac{1}{2}\|\eeta\|^2 
\Prto[n]
\frac{\widetilde{R}}{2}, \frac{d\|\eeta\|^2}{2n} \Prto[n] \frac{r\tilde{R}}{2}$ by Assumptions \ref{assumption: Vector} and $\frac{\|\eeta\|^2}{2n}\sum_{j=1}^d e^{-2\gamma\sigma_j^2 t} \Prto[n] \frac{\tilde{R}r}{2}h_0(t)$ by Assumptions \ref{assumption: spectral_density}. Moreover, $\sum_{j=1}^{n\wedge d} (1-e^{-t \gamma \sigma_j^2}) = \sum_{j=1}^{d} (1-e^{-t \gamma \sigma_j^2})$ always holds because $\sigma_j = 0$ for $j > n\wedge d$ and this cancels out $\frac{\|\eeta\|^2}{n}\sum_{j=1}^d e^{-t\gamma\sigma_j^2}$ in \eqref{eq:etaerrorapprox}.

\end{proof}

\section{Estimates based on martingale concentration}
\label{sec:mgles}
\subsection{General techniques}
\label{sec:mglesabstract}
We recall that the martingales $M_{t,j}$ and $\widetilde{M}_{t,j}$ are defined in \eqref{eq:doob}.  Both of these martingales need to be controlled, but only after summing them in a specific way.  First, we do not need these martingales directly, but only certain integrals against these martingales. These are defined for all $t \geq 0$ and all $j \in \{1,2,\hdots, d\},$
\begin{equation}\label{eq:othermgles}
\begin{aligned}
\widetilde{X}_{t,j} &=
\int_0^t e^{s \gamma \sigma_j^2} \dif \widetilde{M}_{{s},j} \\
{X}_{t,j} &=
\int_0^t e^{2s \gamma \sigma_j^2} \dif {M}_{{s},j}, \\
\end{aligned}
\end{equation}
which are again c\`adl\`ag, finite variation martingales.  We will need to show concentration for sums of these martingales such as $\sum_{j=1}^d c_j \widetilde{X}_{t,j}$ and $\sum_{j=1}^d c_j {X}_{t,j}$ for bounded coefficients $\left\{ c_j \right\}.$

We formulate some general concentration lemmas for c\`adl\`ag, finite variation martingales $Y_t$ with jumps given exactly by $\{ \tau_k : k \geq 0\}$.  For such a process, the jumps entirely determine its fluctuations. We will define for any c\`adl\`ag process $Y,$
\[
\Delta Y_t \defas Y_t - Y_{t-},
\]
which is $0$ for all $t$ except $\{ \tau_k : k \geq 0\}$.  For concreteness and for reference, we record that the jumps of $\widetilde{X}$ and $X$ are given by
\begin{equation}
  \begin{aligned}
  \Delta \widetilde{X}_{\tau_k,j}
    &=
    e^{\gamma \sigma_{j}^2\tau_k}
    \bigl(
    \mathbb{e}_{j}^T
    \gamma \SSigma^T
    \UU^T \PP_{k-1}
    (\UU\SSigma {\nnu}_{\tau_k-} - \eeta)
    \bigr), \\
    \Delta {X}_{\tau_k,j}
    &=
    e^{2\gamma \sigma_{j}^2\tau_k}
    \bigl(\nu_{\tau_k,j}+\nu_{\tau_k-,j} \bigr)
    \bigl(
    \mathbb{e}_{j}^T
    \gamma \SSigma^T
    \UU^T \PP_{k-1}
    (\UU\SSigma {\nnu}_{\tau_k-} - \eeta)
    \bigr), \\
  \end{aligned}
  \label{eq:QV}
\end{equation}

To control the fluctuations of these martingales, we need to control their quadratic variations. The \emph{quadratic variation} $[Y_t]$ is the sum of squares of all jumps of the process, and hence
\[
[Y_t] \defas \sum_{k=1}^{N_{t}} (\Delta Y_{\tau_k})^2.
\]
Likewise the \emph{predictable quadratic variation} $\langle Y_t \rangle$ is
\[
\langle Y_t \rangle \defas \sum_{k=1}^{N_{t}} \E\bigl[ (\Delta Y_{\tau_k})^2 ~| \mathcal{F}_{\tau_k-}\bigr].
\]
Moreover, for some of the martingales we consider here, it is possible to find good events on which the quadratic variation or the predictable quadratic variations are in control.  Then it is a relatively standard fact that the fluctuations of these processes are in control:
\begin{lemma}\label{lem:mgleconcentration}
Suppose that $(Y_t : t\geq 0)$ is a c\`adl\`ag finite variation martingale.
Suppose there is an event $\mathcal{G}$ which is measurable with respect to $\mathcal{F}_0$ that holds with overwhelming probability, and so that for some $T > 0$
\[
  (i)
  \quad
[Y_T]\mathbb{1}_{\mathcal{G}} \leq \tfrac{\beta}{nT} N_T;
\qquad
\text{or}
\qquad
(ii)
\quad
\langle Y_T \rangle \mathbb{1}_{\mathcal{G}} \leq \tfrac{\beta}{nT} N_T
\quad
\text{and}
\quad
\max_{0 \leq t \leq T} |Y_{t} - Y_{t-}| \mathbb{1}_{\mathcal{G}} \leq 1. 
\]
Then for any $\epsilon > 0$ with overwhelming probability
\[
\sup_{0 \leq t \leq T} |Y_t| \leq n^{\epsilon}.
\]
\end{lemma}
\begin{proof}
  We begin with the proof of \textit{(i)}.
 Using the Burkholder--Davis--Gundy inequalities (see \cite[Theorem IV.49]{protter2005stochastic}), for any $p > 1$ there is a constant $C_p$ so that
\[
\E \biggl( \sup_{0 \leq t \leq T} |Y_t| \mathbb{1}_{\mathcal{G}} \biggr)^p
\leq C_p \E \bigl[ [Y_T]^p \mathbb{1}_{\mathcal{G}}\bigr] \le C_p\bigl( \tfrac{\beta}{nT} \bigr)^p\E [N_T^p ].
\]
There is an absolute constant $C > 0$ so that
\[
\E( N_T^p ) \leq C p! (\E N_T)^p,
\]
and so we conclude that
\[
\E \biggl( \sup_{0 \leq t \leq T} |Y_t| \mathbb{1}_{\mathcal{G}} \biggr)^p
\leq CC_p.
\]
Using Markov's inequality, we conclude that
\[
\Pr( \{ \sup_{0 \leq t \leq T} |Y_t| \geq n^{\epsilon}\})
\leq \Pr(\mathcal{G}^c) + CC_p n^{-\epsilon p}.
\]
Hence letting $p$ tend slowly to infinity with $n,$ this concludes the proof of \textit{(i)}.

We turn to the proof of \textit{(ii)}.  We need a tail bound for martingales (see \cite[Appendix B.6 Inequality 1]{ShorackWellner}), which states that
\[
  \Pr( \{\sup_{0 \leq t \leq T} |Y_t| > s\} \cap \{\langle Y_T \rangle \leq r\} \cap \mathcal{G})
  \leq 
  2\exp\left( -\frac{s^2}{\tfrac{2s}{3} + 2r} \right).
\]
Taking $s=r=n^{\epsilon},$ this vanishes faster than any power of $n.$  The probability that $N_T > n^{\epsilon}(\E [N_T])$ additionally decays faster than any power of $n,$ so that we conclude that on $\mathcal{G},$ $\sup_{0 \leq t \leq T} |Y_t| \leq n^{\epsilon}$ with overwhelming probability. 
\end{proof}

We will need an extension of this standard type of concentration, which allows for exceptional jumps.  Suppose we can decompose the jumps $\left\{ \tau_k\right\}$ of $(Y_t: t\geq 0)$ into two types $\left\{\tau_{k,1}, \tau_{k,2} \right\}$.  In our application, we shall pick the jumps of the second type to be those for which a fixed coordinate $1 \in B_k$ and the first type to be all that remains.  Thus by properties of the Poisson process, the two processes $\left\{\tau_{k,1}, \tau_{k,2} \right\}$ are independent Poisson processes.
\begin{lemma}\label{lem:raremartingales}
Suppose that $(Y_t : t \geq 0)$ is a c\`adl\`ag finite variation martingale with jumps given by $\{ \tau_k \}$.  Suppose these jumps are divided into two groups $\left\{\tau_{k,1}, \tau_{k,2} \right\}$ by a rule depending only on $(k,\PP_k)$.  Let $N_{t,1}$ and $N_{t,2}$ be the counting functions of the number of jumps from either type.  Suppose that the jumps of $Y_t$ of type $1$ (the typical ones) satisfy
\[
\mathbb{E} [ \bigl(\Delta Y_{\tau_{k,1}}\bigr)^2~\vert~\mathcal{F}_{\tau_{k,1}-}]
\leq \frac{\beta}{n}
\quad
\text{and}
\quad
|\Delta Y_{\tau_{k,1}}| \leq 1.
\]
For the jumps of the second type, suppose that for some $T> 0$ there is a constant $C>1$ so that 
\(
\mathbb{E} [N_{T,2} ] \leq C
\)
and a constant $\delta \in (0,1)$ so that
\[
(i). \quad
|\Delta Y_{\tau_{k,2}}| \leq \delta|Y_{\tau_{k,2}-}|+1
\qquad\text{or}\qquad
(ii). \quad 
|Y_{\tau_{k,2}}| \leq \delta|Y_{\tau_{k,2}-}|+1.
\]
Then for any $\epsilon > 0$ with overwhelming probability
\[
\sup_{0 \leq t \leq T} |Y_t| \leq n^{\epsilon}.
\]
\end{lemma}
\begin{proof}
Let $n_{t,1}$ and $n_{t,2}$ be the L\'evy measures for the jumps of $Y$ of types $1$ and $2$; i.e.\ the measures so that for any bounded continuous function $f$ and $\ell \in \{1,2\},$
\[
\sum_{k = 1}^{N_{t,\ell}}
f\bigl( \Delta Y_{\tau_{k,\ell}}\bigr)
-\int f(x) n_{t,\ell}(\dif x)
\]
is a martingale.  We decompose the martingale $(Y_t :t \geq 0)$ into pieces.
Define 
\[
Y_{t,\ell}=\sum_{k = 1}^{N_{t,\ell}}\Delta Y_{\tau_{k,\ell}}-\int x n_{t,\ell}(\dif x).
\]
Then $Y_t = Y_{t,1}+Y_{t,2}$ for all $t \geq 0$.

We use two different versions of the exponential martingale.  The first, which we believe originates with \cite{Yor} (c.f.\ \cite[Lemme 2]{Lepingle}) is
\[
\widehat{Z}_{t,1} \defas
\exp\biggl( \lambda Y_{t,1} - \int \bigl(e^{\lambda x} -1 - \lambda x\bigr)n_{t,1}(\dif x) \biggr),
\]
which is a martingale.
The second is the Dol\'eans exponential, which is the more commonly cited (\cite[II. Theorem 37]{protter2005stochastic}, \cite{Yor}), and which shows
\[
\widehat{Z}_{t,2} \defas
\exp\bigl( \lambda Y_{t,2} \bigr)\prod_{k=1}^{N_{t,2}}f\bigl( \lambda(\Delta Y_{\tau_{k,2}})\bigr) 
\quad
\text{where}
\quad f(x) = (1+x)e^{-x}.
\]
As both processes are finite variation and have no common jumps, their product remains a martingale.  Thus
\[
\widehat{Z}_t \defas 
\exp\biggl( \lambda Y_{t,1} + \lambda Y_{t,2} - \int \bigl(e^{\lambda x} -1 - \lambda x\bigr)n_{t,1}(\dif x) \biggr)
\prod_{k=1}^{N_{t,2}}f\bigl( \lambda\Delta Y_{\tau_{k,2}}\bigr)
\]
is a martingale.  Note that the two martingales combine to form $Y_t.$  By the assumption, the jumps of $Y_{t,1}$ are less than or equal to $1$.  Hence the measure $n_{t,1}(\dif x)$ is supported on $[-1,1]$.  For $|u| \leq 1,$
\[
e^u - 1 - u \leq \tfrac{u^2}{e-2}.
\]
So we define a supermartingale $(Z_t : t \geq 0)$ for any $\lambda \leq 1$ by
\[
\widehat{Z}_t \geq {Z}_t \defas 
\exp\biggl( \lambda Y_{t} - \int \frac{\lambda^2 x^2}{e-2}n_{t,1}(\dif x) \biggr)
\prod_{k=1}^{N_{t,2}}f\bigl( \lambda \Delta Y_{\tau_{k,2}}\bigr).
\]
The integral $\int x^2 n_{t,1}(\dif x)$ is the predictable quadratic variation
\[
\langle Y_{t,1}\rangle
= \sum_{k=1}^{N_{t,1}} \mathbb{E} \bigl[ ( \Delta Y_{\tau_{k,1}})^2~|~\mathcal{F}_{\tau_{k,1}-} \bigr]
\leq \tfrac{\beta}{n}N_{t,1}.
\]

Now we fix a parameter $r > \frac{1}{1-\delta}$ and let 
\[
\vartheta = \inf\{ t \geq 0 ~:~ |Y_t| \geq r \}.
\]
By optional stopping for any bounded stopping time $\rho \geq 0,$
\begin{equation}\label{eq:l14_0}
\E\bigl[ Z_{\vartheta \wedge \rho} \mathbb{1}_{\vartheta \leq \rho}\bigr] \leq  \E [Z_{\vartheta \wedge \rho}] \leq \E [Z_0] = 1.
\end{equation}
So, for $\lambda \in (0,1),$
\[
\begin{aligned}
Z_{\vartheta \wedge \rho} \mathbb{1}_{\vartheta \leq \rho} \geq 
&\exp\bigl( \lambda r - \tfrac{\lambda^2}{e-2} \tfrac{\beta}{n}N_{\rho,1}\bigr)\prod_{k=1}^{N_{\rho,2}}f\bigl( \lambda \Delta Y_{\tau_{k,2}})\bigr)\mathbb{1}_{\{ Y_{\vartheta \wedge \rho} \geq r\}}\\
-&\exp\bigl( -\lambda r - \tfrac{\lambda^2}{e-2}\tfrac{\beta}{n}N_{\rho,1}\bigr)\prod_{k=1}^{N_{\rho,2}}|f\bigl( \lambda \Delta Y_{\tau_{k,2}}\bigr)|
\mathbb{1}_{\{Y_{\vartheta \wedge \rho} \leq -r\}}.\\
\end{aligned}
\]
We produce a similar bound on taking $-\lambda \in (0,1)$ although with the roles reversed.

The product may in principle be negative or $0$.  So we consider taking $\rho = T \wedge \tau_{1,2}$, for some fixed $T > 0$. Then if $\vartheta < \tau_{1,2},$ we have an empty product.  Otherwise, we have $\vartheta = \tau_{1,2},$ in which case the product contains a single term.  

If assumption (ii) is in force, then either the jump decreases the absolute value of $Y_{\tau_{1,2}}$ as it is opposite sign from $Y_{\tau_{1,2}-}$ and does not cross $0$ 
or the second condition is in force. 
In that case, since $|Y_{\tau_{1,2}-}| \leq r,$ and since
\[
r \leq |Y_{\tau_{1,2}}| \leq \delta r + 1,
\]
we would have $r \leq \frac{1}{1-\delta}.$  However, we have chosen $r$ large enough that this is not the case.  So, we conclude that when assumption (ii) is in force, we could not have had $\vartheta = \tau_{1,2}.$ We conclude in the case of assumption (ii) that
\begin{equation}\label{eq:l14_2}
Z_{\vartheta \wedge \rho} \mathbb{1}_{\{\vartheta \leq \rho\}} \geq 
\exp\bigl( \lambda r - \tfrac{\lambda^2}{e-2}\tfrac{\beta}{n}N_{\rho,1}\bigr)\mathbb{1}_{\{Y_{\vartheta \wedge \rho} \geq r\}}
-\exp\bigl( -\lambda r - \tfrac{\lambda^2}{e-2}\tfrac{\beta}{n}N_{\rho,1}\bigr)
\mathbb{1}_{\{Y_{\vartheta \wedge \rho} \leq -r\}}.\\
\end{equation}

If assumption (i) is in force,
then if $Y_\vartheta \geq r,$ the jump of $Y$ at $\tau_{1,2}$ is necessarily positive, as this is the first time the martingale jumped above some level.  As assumption (i) is in force, then $Y_{\tau_{1,2}-} > 0$ as well, and so the jump of type $2$ must satisfy
\[
\Delta Y_{\tau_{1,2}}
\leq 
\delta Y_{\tau_{1,2}-}+1
\leq \delta r + 1.
\] 
We conclude that when $Y_\vartheta \geq r$ and assumption (i) holds,
\[
f\bigl( \lambda\Delta Y_{\tau_{1,2}}\bigr)
\geq e^{-\lambda\Delta Y_{\tau_{1,2}}}
\geq e^{-\lambda(\delta r+1)}.
\]
If on the other hand $Y_\vartheta \leq -r,$ then the jump must have been negative, and we conclude similarly that
\[
|f\bigl( \lambda\Delta Y_{\tau_{1,2}}\bigr)|
\leq (1-\lambda\Delta Y_{\tau_{1,2}})e^{-\lambda\Delta Y_{\tau_{1,2}}}
\leq (1+\lambda (\delta r + 1))e^{\lambda(\delta r+1)}.
\]
Hence
\begin{equation}\label{eq:l14_1}
\begin{aligned}
Z_{\vartheta \wedge \rho} \mathbb{1}_{\{\vartheta \leq \rho\}} 
&\geq 
\exp\bigl( -1+\lambda (1-\delta)r - \tfrac{\lambda^2}{e-2}\tfrac{\beta}{n}N_{\rho,1}\bigr)\mathbb{1}_{\{Y_{\vartheta \wedge \rho} \geq r\}}\\
&-
(1+\lambda (\delta r + 1))
\exp\bigl( 1-\lambda (1-\delta)r - \tfrac{\lambda^2}{e-2}\tfrac{\beta}{n}N_{\rho,1}\bigr)
\mathbb{1}_{\{Y_{\vartheta \wedge \rho} \leq -r\}}.\\
\end{aligned}
\end{equation}

In either case of \eqref{eq:l14_2} or \eqref{eq:l14_1},  using \eqref{eq:l14_0} and the boundedness of 
\[
r \mapsto \lambda(\delta r + 1)e^{-\lambda(1-\delta)r},
\]
there is a constant $C_\delta>0$ so that
\[
\mathbb{E}
\biggl(
\exp\bigl( \lambda (1-\delta) r - \tfrac{\lambda^2}{e-2}\tfrac{\beta}{n}N_{\rho,1}\bigr)\mathbb{1}_{\{Y_{\vartheta \wedge \rho} \geq r\}}
\biggr)
\leq C_\delta.
\]
With overwhelming probability $\tfrac{\beta}{n}N_{\rho,1} \leq \tfrac{\beta}{n}N_{T}\leq 2T$, and hence on the event $\mathcal{E}$ that $\tfrac{\beta}{n}N_{\rho,1} \leq 2T,$
\begin{equation}\label{eq:l14_3}
\mathbb{E}
\biggl(
\exp\bigl( \lambda (1-\delta) r - \tfrac{\lambda^2}{e-2}2T\bigr)\mathbb{1}_{\{Y_{\vartheta \wedge \rho}\mathbb{1}_{\mathcal{E}} \geq r\}}
\biggr)
\leq C_\delta.
\end{equation}
Thus taking $\lambda = 1$ and $r=(1-\delta)^{-1}(\log n)^2,$ we conclude that
\[
e^{(\log n)^2}
\Pr( Y_{\vartheta \wedge \rho} \geq n^{\epsilon} \cap \mathcal{E})= O(1),
\]
and hence $Y_{\vartheta \wedge \rho} \leq n^{\epsilon}$ with overwhelming probability.

By applying the same argument to $-Y_t$ which is again a martingale satisfying the same assumptions, we can conclude with overwhelming probability that \[
\sup\bigl\{
|Y_t| 
: 0 \leq t \leq (T \wedge \tau_{1,2}) 
\bigr\}
\leq 2(1-\delta)^{-1}(\log n)^2.
\]
Now we suppose that with overwhelming probability, we have shown for some $\ell \in \N$
\[
\sup\bigl\{
|Y_t| 
: 0 \leq t \leq (T \wedge \tau_{\ell,2}) 
\bigr\}
\leq 2^{\ell}(1-\delta)^{-\ell}(\log n)^2.
\]
We now apply the same bounds to
\[
Z_t/Z_{t \wedge \tau_{\ell,2}}
\defas \exp\biggl( \lambda (Y_{t}-Y_{t \wedge \tau_{\ell,2}}) - \int \frac{\lambda^2 x^2}{e-2}n_{t,1}(\dif x)
\biggr)
\prod_{k=\ell+1}^{N_{t,2}}f\bigl( \lambda \Delta Y_{\tau_{k,2}}\bigr).
\]
In particular taking the conditional expectation, with the same $\vartheta$ and with $\rho = T \wedge \tau_{\ell+1,2}$
\[
\mathbb{E}
[
Z_{\vartheta \wedge \rho}/Z_{\vartheta \wedge 
\rho \wedge \tau_{\ell,2}}
\mathbb{1}_{\{\vartheta \leq \rho\}}
~\vert~ \mathcal{F}_{\tau_{k,2}}
]
\leq 1.
\]
Rearranging, we conclude
\[
\mathbb{E}
\biggl(
\exp\biggl( \lambda Y_{\vartheta \wedge \rho} - \int \frac{\lambda^2 x^2}{e-2}n_{t,1}(\dif x)
\biggr)
\prod_{k=\ell+1}^{N_{\vartheta \wedge \rho,2}}f\bigl( \lambda \Delta Y_{\tau_{k,2}}\bigr)
~\vert~ \mathcal{F}_{\tau_{\ell,2}}
\biggr)
\leq
\exp\biggl(
\lambda Y_{\vartheta \wedge \rho \wedge \tau_{\ell,2}}
\biggr).
\]
Hence following the same line of argument that leads to \eqref{eq:l14_3},
\[
\E
\biggl(
\exp\bigl(
 \lambda (1-\delta) r - \tfrac{\lambda^2}{e-2}2T\bigr)
\bigr)
\mathbb{1}_{\{Y_{\vartheta \wedge \rho} \geq r\}}
\mathbb{1}_{\mathcal{E}}
~\vert~ \mathcal{F}_{\tau_{\ell,2}}
\biggr)
\leq 
C_\delta \exp\biggl(
\lambda Y_{\vartheta \wedge \rho \wedge \tau_{\ell,2}}
\biggr).
\]
Taking $\lambda = 1$ and $r = 2^{\ell+1} (1-\delta)^{-\ell-1}$ and restricting to the event in the inductive hypothesis, 
\[
\E
\biggl(
\exp\bigl(
 2^{\ell+1}(1-\delta)^{-\ell}
 (\log n)^2
\bigr)
\mathbb{1}_{\{Y_{\vartheta \wedge \rho} \geq r\}}
\mathbb{1}_{\mathcal{E}}
~\vert~ \mathcal{F}_{\tau_{\ell,2}}
\biggr)
\leq 
C_\delta \exp\biggl(
 2^{\ell}(1-\delta)^{-\ell} (\log n)^2
\biggr).
\]
In particular, with overwhelming probability,
\[
\sup\bigl\{
|Y_t| 
: 0 \leq t \leq (T \wedge \tau_{\ell+1,2}) 
\bigr\}\leq 2^{\ell}(1-\delta)^{-\ell}(\log n)^2.
\]

The number of type-2 jumps before $T$ is $N_{T,2},$ which is Poisson with mean $C$.  Hence with overwhelming probability, for any $\epsilon > 0,$ $N_{T,2} \leq \bigl(\log \tfrac{2}{1-\delta}\bigr)^{-1} \frac{\epsilon}{2} \log n.$
Hence, we conclude that with overwhelming probability,
\[
\sup\bigl\{
|Y_t| 
: 0 \leq t \leq T  
\bigr\}\leq 
\bigl( \tfrac{2}{1-\delta} \bigr)^{N_{T,2}}
(\log n)^2
\leq n^{\epsilon} (\log n)^2.
\]
As $\epsilon >0$ may be picked as small as desired, the proof follows.
\end{proof}

\subsection{Applications to the control of errors in the Volterra equation}
\label{sec:mglesapplication}

\subsubsection{Delocalization of the function values: the proof of Proposition \ref{prop:delocalization}}

\begin{proof}[Proof of Proposition \ref{prop:delocalization}]
It suffices to prove that for a fixed $i$ and for any $T >0$ and any $\epsilon > 0,$
\[
\sup_{0 \leq t \leq T}
\bigl(\mathbb{e}_{i}^T(\UU\SSigma \nnu^\vartheta_t - \eeta)\bigr)^2
\leq \beta n^{\epsilon-1}
\]
with overwhelming probability.

Using Lemma \ref{lem:integrated}, we have the representation
\begin{equation*}
\nu_{t,j}
    =
    e^{-\gamma \sigma_j^2t}\nu_{0,j} 
    +\int_0^t e^{-\gamma \sigma_j^{2}(t-s)}\gamma \sigma_j(\UU^T\eeta)_j \dif s +  e^{-\gamma \sigma_j^2 t} \widetilde{X}_{t,j}.
\end{equation*}
Note that $\nu_{t,j}^{\vartheta}$ has the same representation by replacing $t \to t \wedge \vartheta$. Observe that each of the first two terms has the contribution of $\mathcal{O}(n^{\epsilon-1/2})$ to $|\mathbb{e}_{i}^T(\UU\SSigma \nnu^\vartheta_t - \eeta)|$ with overwhelming probability. Indeed, we have
\begin{align}\label{eq:nonMGerror}
\begin{split}
    \mathbb{e}_{i}^T(\UU\SSigma \nnu^\vartheta_t - \eeta) - \sum_{j=1}^d U_{ij}\sigma_j e^{-\gamma t\sigma_j^2} \widetilde{X}_{t,j} 
    &= - \eta_i + \sum_{j=1}^d U_{ij}\sigma_je^{-\gamma \sigma_j^2t}\nu_{0,j}\\ 
    &\quad+\sum_{j=1}^d U_{ij}\sigma_j\int_0^t e^{-\gamma \sigma_j^{2}(t-s)}\gamma \sigma_j(\UU^T\eeta)_j \dif s\\
    &= \mathcal{O}(n^{\epsilon-1/2}).
    \end{split}
\end{align}
Here the order of first term comes from Assumption \ref{assumption: Vector}. Corollary \ref{cor:lipschitz} gives the order of the second term, by defining $F(U) := \sum_{j=1}^d U_{ij}\sigma_j e^{-\gamma\sigma_j^2 t} \nu_{0,j}$ with conditioning on $\nnu_0$. The order of the last term is obtained from Lemma \ref{lem:FU} with setting $\aa=\mathbb{e}_i, \bb = \eeta$ and conditioning on $\SSigma$. Indeed, note that $\mathbb{E}[U_{ij}(\UU^T\eeta)_j] = \frac{1}{n}\eta_i$ so that w.o.p.,
\[
\E \bigg[\sum_{j=1}^d \gamma U_{ij}\sigma_j^2 (\UU^T\eeta)_j \int_0^{t} e^{-(t-s)\gamma \sigma_j^2}  \dif s \bigg] = \bigg( \frac{1}{n} \sum_{j=1}^d \gamma\sigma_j^2 \int_0^{t} e^{-(t-s)\gamma \sigma_j^2} \bigg) \eta_i = \mathcal{O}(n^{\epsilon-1/2}).
\]

Therefore, it would suffice to bound
\begin{equation}\label{eq:Yt}
    Y_t \defas \sum_{j=1}^d U_{ij}\sigma_j e^{-\gamma q\sigma_j^2} \widetilde{X}_{t,j} \qquad  0\le t\le q,
\end{equation}
for some fixed $q \ge 0$. Note that as we did in Lemma \ref{lem:modelconvergence}, showing the bound for a fixed $q\in[0,T]$ should be sufficient, considering mesh points on $[0,T]$ with spacing, let us say, $\lambda = \lambda(n) >0$, which depends on $n$. Since the process $\nnu_t$ is constant between jumps, the only cases which cannot be covered by mesh points are having multiple jumps between two adjacent mesh points. However, as the possibility of such events is given by $\mathcal{O}(\beta^{-2}n^2\lambda)$, which can be smaller than any power of polynomial of $n$, we conclude that every jump can be covered by mesh points with overwhelming probability.  Each jump for the process $Y_{(\cdot)}$ is given by
\begin{equation}\label{eq:DeltaYdef}
    \Delta Y_{\tau_{k+1}} \defas Y_{\tau_{k+1}} - Y_{\tau_{{(k+1)}-}} = -\sum_{j=1}^d U_{ij}\sigma_j e^{-(q-\tau_{k+1})\gamma \sigma_j^2} \mathbb{e}_j^T \gamma \SSigma^T \UU^T \PP_k(\UU\SSigma \nnu_{\tau_{{(k+1)}-}}^\vartheta- \eeta).
\end{equation}
Note that there are two different types of jumps, i.e. \begin{enumerate}
    \item $B_k$ does not include the index $i$.
    \item $B_k$ includes the index $i$.
\end{enumerate}
Replacing $\PP_k$ by $\sum_{l\in B_k} \mathbb{e}_l\mathbb{e}_l^T$, the jump $\Delta Y_{\tau_{k+1}}$ can be translated as
\begin{equation}
    \Delta Y_{\tau_{k+1}} = -\sum_{l \in B_k} \bigl[\sum_{j=1}^d \gamma U_{ij} U_{lj}\sigma_j^2 e^{-(q-\tau_{k+1}) \gamma \sigma_j^2}\bigr] \mathbb{e}_l^T (\UU\SSigma \nnu_{\tau_{(k+1)-}}^\vartheta- \eeta).
\end{equation}
Let 
\begin{equation}
    \Phi_{i,l}(t) \defas \sum_{j=1}^d \gamma U_{ij} U_{lj}\sigma_j^2 e^{-(q-t) \gamma \sigma_j^2},
\end{equation}
and let $\mathcal{G} = \mathcal{G}(\theta)$ for $\theta >0$ be the event defined as
\begin{equation}
   \mathcal{G} \defas \biggl\{ \sup_{1\le l\le n, l\neq i}\max_{0 \le t \le T} |\Phi_{i,l}(t) | \le
    n^{\theta - 1/2},\  \max_{0 \le t \le T} |\Phi_{i,i}(t)| < 2 \biggr\}.
\end{equation}
Note that this holds with overwhelming probability, by Lemma \ref{lem:FU} and condition on the stepsize $\gamma$, see Theorem \ref{thm:main_critical_stepsize}.  Furthermore, in order to apply the \textit{bootstrap} argument, let us define, for $\aleph \in [-\epsilon/2,1/2)$, 
\begin{equation}
    \hbar \defas \inf \{ t \le \vartheta: \max_{1\le l \le n} |\mathbb{e}_l^T (\UU \SSigma \nnu_t - \eeta)| > n^{-\aleph} \}.
\end{equation}

Now we are ready to apply Lemma  \ref{lem:raremartingales} to prove the claim.\\[5pt]

\textbf{Case 1.} When $B_k$ does not include the index $i$: we need to control $\mathbb{E}[(\Delta Y_{\tau_{k+1,1}}^\hbar)^2| \mathcal{F}_{\tau_{(k+1)}-}]$ and $|\Delta Y_{\tau_{k+1,1}}^\hbar|$. Observe,
\begin{equation}\label{eq:DeltaYhbar}
    |\Delta Y_{\tau_{k+1,1}}^\hbar| =
    |\sum_{l \in B_k} \Phi_{i,l} \mathbb{e}_l^T (\UU\SSigma \nnu_{\tau_{k+1,1}-}^\hbar- \eeta)| 
    \le \beta n^{\theta - 1/2 - \aleph}.
\end{equation}
On the other hand,
\begin{align}\label{eq:DeltaYhbarsquared}
\begin{split}
    \mathbb{E}[(\Delta Y_{\tau_{k+1,1}}^\hbar)^2| \mathcal{F}_{\tau_{(k+1)-}}]&= \frac{\beta(\beta - 1)}{(n-1)(n-2)} \bigl[ \sum_{l=1}^n \Phi_{i,l} \mathbb{e}_l^T (\UU\SSigma \nnu_{\tau_{k+1,1}-}^\hbar- \eeta) \bigr]^2\\ 
    &\quad + \biggl(\frac{\beta}{n-1} - \frac{\beta(\beta-1)}{(n-1)(n-2)} \biggr) \sum_{l=1}^n \Phi_{i,l}^2 (\mathbb{e}_l^T (\UU\SSigma \nnu_{\tau_{k+1,1}-}^\hbar- \eeta))^2\\
    & \le \frac{\beta^2}{n^2}n^{4\theta} + \frac{\beta}{n}n^{4\theta -1}\\
    & \le \frac{\beta}{n} (\beta n^{4\theta -1} + n^{4\theta -1}).
    \end{split}
\end{align}
Here Cauchy-Schwarz inequality as well as the definition of $\vartheta$ were used for the inequality. \\[5pt]

\noindent \textbf{Case 2.} When $B_k$ includes the index $i$: In this case, we want to have the following: for some $T> 0$, there is a constant $C>1$ so that 
\(
\mathbb{E} N_{T,2} \leq C
\)
and a constant $\delta \in (0,1)$ so that
\[
(i). \quad
\,
|\Delta Y_{\tau_{k+1,2}}^\hbar| \leq \delta|Y_{\tau_{k+1,2}-}^\hbar|+1
\qquad\text{or}\qquad
(ii). \quad 
|Y_{\tau_{k+1,2}}^\hbar| \leq \delta|Y_{\tau_{k+1,2}-}^\hbar|+1.
\]
First recall that $N_t$ has the distribution of $\text{Poisson}(\tfrac{n}{\beta}t)$. Since $B_k$ contains the index $i$ with probability $\binom{n-1}{\beta-1}/\binom{n}{\beta} = \frac{\beta}{n}$, we have 
\[
\mathbb{E}[N_{T,2}] = \frac{\beta}{n}\frac{nT}{\beta} = T<\infty.
\]
Now observe, with $\mathfrak{t} \defas \tau_{k+1,2} \wedge \hbar$,
\begin{align}\label{eq:deltaYh2}
    \begin{split}
        \Delta Y_{\tau_{k+1,2}}^\hbar 
        &= -\sum_{l \in B_k} \bigl[\sum_{j=1}^d \gamma U_{ij} U_{lj}\sigma_j^2 e^{-(q-\mathfrak{t}) \gamma \sigma_j^2}\bigr] \mathbb{e}_l^T (\UU\SSigma \nnu_{\tau_{k+1,2}-}^\hbar- \eeta)\\
        &= -\sum_{j=1}^d \gamma U_{ij}^2 \sigma_j^2 e^{-(q-\mathfrak{t}) \gamma \sigma_j^2} \mathbb{e}_i^T (\UU\SSigma \nnu_{\tau_{k+1,2}-}^\hbar- \eeta)\\
        &\quad -\sum_{l \in B_k, l\neq i} \bigl[\sum_{j=1}^d \gamma U_{ij} U_{lj}\sigma_j^2 e^{-(q-\mathfrak{t}) \gamma \sigma_j^2}\bigr] \mathbb{e}_l^T (\UU\SSigma \nnu_{\tau_{k+1,2}-}^\hbar- \eeta).\\
    \end{split}
\end{align}
We will see that the first term will lead to the one including $Y_{\tau_{k+1,2}-}^\hbar$ with errors. From Lemma \ref{lem:integrated}, we have
\[
    \nu_{t,j}
    =
    e^{-\gamma \sigma_j^2t}\nu_{0,j} 
    +\int_0^t e^{-\gamma \sigma_j^2(t-s)}\gamma \sigma_j(\UU^T\eeta)_j \dif s + e^{-\gamma \sigma_j^2t}\widetilde{X}_{t,j},
\]
and this gives with overwhelming probability
\begin{align*}
    \begin{split}
        Y_{\tau_{{k+1},2}-}^\hbar &= \sum_{j=1}^d U_{ij}\sigma_j e^{-\mathfrak{t}\gamma\sigma_j^2} \widetilde{X}_{(\tau_{{k+1,2}-}),j}^\hbar\\
        &= \sum_{j=1}^d U_{ij}\sigma_j \biggl[
        \nu_{(\tau_{{k+1,2}-}),j}^\hbar - e^{-\mathfrak{t}\gamma \sigma_j^2}\nu_{0,j}
        -\int_0^{\mathfrak{t}} e^{-(\mathfrak{t}-s)\gamma \sigma_j^2}\gamma \sigma_j(\UU^T\eeta)_j \dif s \biggr]\\
        & = \mathbb{e}_i^T \UU\SSigma \nnu_{\tau_{{k+1,2}-}}^\hbar - \sum_{j=1}^d U_{ij}\sigma_j  e^{-\mathfrak{t}\gamma \sigma_j^2}\nu_{0,j}
       - \sum_{j=1}^d \gamma U_{ij}\sigma_j^2 (\UU^T\eeta)_j \int_0^{\mathfrak{t}} e^{-(\mathfrak{t}-s)\gamma \sigma_j^2}  \dif s\\
        &= \mathbb{e}_i^T \UU\SSigma \nnu_{\tau_{{k+1,2}-}}^\hbar + \mathcal{O}(n^{\theta-1/2}).
    \end{split}
\end{align*}
In the last line to get the order, we used Corollary \ref{cor:lipschitz} for the second term and Lemma \ref{lem:FU} for the last term with setting $\aa = \mathbb{e}_i, \bb = \eeta$ and conditioning on $\SSigma$. See the arguments after \eqref{eq:nonMGerror} for detail.
Thus, from \eqref{eq:deltaYh2} we have with overwhelming probability
\begin{align}\label{eq:DeltaY}
    \begin{split}
        \Delta Y_{\tau_{k+1,2}}^\hbar
         &= -\sum_{j=1}^d \gamma U_{ij}^2 \sigma_j^2 e^{-(q-\mathfrak{t}) \gamma \sigma_j^2} (Y_{\tau_{{k+1},2}-}^\hbar + \mathcal{O}(n^{\theta-1/2}))\\
         &\quad -\sum_{l \in B_k, l\neq i} \bigl[\sum_{j=1}^d \gamma U_{ij} U_{lj}\sigma_j^2 e^{-(q-\mathfrak{t}) \gamma \sigma_j^2}\bigr] \mathbb{e}_l^T (\UU\SSigma \nnu_{\tau_k}^\hbar - \eta)\\
         &= -\sum_{j=1}^d \gamma U_{ij}^2 \sigma_j^2 e^{-(q-\mathfrak{t}) \gamma \sigma_j^2} Y_{\tau_{k+1,2}-}^\hbar + \mathcal{O}(n^{\theta - 1/2}) + \mathcal{O}(\beta n^{\theta - 1/2 - \aleph}),
        \end{split}
\end{align}
where Lemma \ref{lem:FU} was used again in the last line; when $l\neq i$,
\[
\E \big[ \sum_{j=1}^d \gamma U_{ij} U_{lj}\sigma_j^2 e^{-(q-\mathfrak{t}) \gamma \sigma_j^2} \big] = 0.
\]
Note that condition $(ii)$ is satisfied from \eqref{eq:DeltaY} after some appropriate scaling of $Y_t^\hbar$, since
\begin{align}\label{eq:Yt_bound}
    \begin{split}
    Y_{\tau_{k+1,2}}^\hbar &= Y_{\tau_{k+1,2}-}^\hbar + 
    \Delta Y_{\tau_{k+1,2}}^\hbar\\
    &= \biggl(1 - \sum_{j=1}^d \gamma U_{ij}^2 \sigma_j^2 e^{-(q-\mathfrak{t}) \gamma \sigma_j^2} \biggr)Y_{\tau_{k+1,2}-}^\hbar + \mathcal{O}(n^{\theta-1/2})+ \mathcal{O}(\beta n^{\theta - 1/2 - \aleph}),
    \end{split}
\end{align}
and $\biggl|1 - \sum_{j=1}^d \gamma U_{ij}^2 \sigma_j^2 e^{-(q-\mathfrak{t}) \gamma \sigma_j^2} \biggr| < 1$ on $\mathcal{G}$.

Now, in view of \eqref{eq:DeltaYhbar}, \eqref{eq:DeltaYhbarsquared} and \eqref{eq:Yt_bound}, scaling $Y_t^\hbar$ by $\max\{\sqrt{\beta}n^{2\theta -1/2}, \beta n^{\theta - 1/2 - \aleph}, n^{\theta - 1/2} \}$ makes every condition for cases 1 and 2 valid, so Lemma \ref{lem:raremartingales} gives 
\begin{equation}\label{eq:supY}
    \sup_{0\le t\le T} |Y_t^\hbar| \le n^{2\theta-1/2} \max\{\sqrt{\beta}n^\theta, \beta n^{-\aleph }, 1 \}.
\end{equation}
We summarize the following conclusion: if we let, for any $\epsilon > 0,$
$\psi_i^{(T)} \defas \underset{0\le t\le T}{\max}|\mathbb{e}_i^T (\UU \SSigma \nnu_t - \eeta)|$, 
\[
  \psi_i^{(T)} \leq n^{-\aleph}
  \;\text{w.o.p.}
  \;\implies
  \;
  \psi_i^{(T)} \leq n^{2\theta-1/2}\max\{\sqrt{\beta}n^\theta, \beta n^{-\aleph }, 1 \}
  \;\text{w.o.p.}
\]
Thus under the assumption that $\beta \leq n^{1/5-\delta}\le n^{1/2-\delta}$, and picking $\theta < \delta/4$  we conclude 
\[
  \psi_i^{(T)} \leq n^{-\aleph}
  \;\text{w.o.p.}
  \;\implies
  \;
  \psi_i^{(T)} \leq \max\{\sqrt{\beta}n^{3\theta-1/2}, n^{-\aleph-\delta/2}, n^{2\theta-1/2}\}
  \;\text{w.o.p.}
\]
Hence by iterating this inequality finitely many times, $\max\{\sqrt{\beta}n^{3\theta-1/2}, n^{-\aleph-\delta/2}, n^{2\theta-1/2}\}$ becomes $\sqrt{\beta}n^{3\theta-1/2}$ and the conclusion follows with the choice of $\theta < \min\{\delta/4, \epsilon/6\}$. The only thing left to check is to bound $\psi_i^{(T)}$ with the initial condition $\aleph = -\epsilon/2$, \textit{i.e.}, \[
\max_{0 \leq t \leq T}  |\mathbb{e}_i^T(\UU\SSigma\nnu_t^\vartheta~-~\eeta)|~\leq~ n^{\epsilon/2}
\]
with overwhelming probability. But this was already given by Lemma \ref{lem:psibound}.
\end{proof}

\subsubsection{Delocalization of the spectral weights: the proof of Proposition \ref{prop:nut}}

\begin{proof}[Proof of Proposition \ref{prop:nut}]
It is sufficient to prove the same claim for a fixed $j$.  To take advantage of the main technical assumption Proposition \ref{prop:delocalization}, we will introduce a stopping time $\hbar,$ defined as (for some $\alpha \in (0,\tfrac 12)$)
\begin{equation}\label{eq:hbar}
\hbar \defas \inf \left\{ t \leq \vartheta: \max_{1 \leq i \leq n} \bigl(\mathbb{e}_{i}^T(\UU\SSigma \nnu_t - \eeta)\bigr)^2 > \beta n^{-2\alpha} \right\}.
\end{equation}
As with overwhelming probability this does not occur, it suffices to show a bound for the stopped processes $\nu^{\hbar}_{t,j} \defas \nu_{t \wedge \hbar,j}.$

Using Lemma \ref{lem:integrated}, we have the representation
\[
    \nu_{t,j}
    =
    e^{-\gamma \sigma_j^2t}\nu_{0,j} 
    +\int_0^t e^{-\gamma \sigma_j^2(t-s)}\gamma \sigma_j(\UU^T\eeta)_j \dif s + e^{-\gamma \sigma_j^2t}\widetilde{X}_{t,j}.
\]
By replacing $t \to t \wedge \hbar,$ we have the same representation for $\nu^{\hbar}_{t,j}.$
We let $\mathcal{G}$ be the event that
\begin{equation}\label{eq:G0}
|(\UU^T\eeta)_j| \leq n^{\epsilon/2-1/2}
\quad\text{and}\quad
\max_{i}
\bigl|
    \mathbb{e}_{j}^T
    \gamma \SSigma^T
    \UU^T \mathbb{e}_{i}
    \bigr|
    \leq n^{\epsilon/2-1/2}.
\end{equation}
By Corollary \ref{cor:entry} this holds with overwhelming probability. 
The first two terms is $n^{\epsilon-1/2}$ with overwhelming probability, using Assumption \ref{assumption: Vector}.
Hence it suffices to show that for any $\epsilon > 0,$
\[
\sup_{0 \leq t \leq T} |\widetilde{X}^\hbar_{t,j}| \leq  
\beta n^{\epsilon-1/2}
\]
with overwhelming probability. 

The quadratic variation is, from \eqref{eq:QV}
\[
\begin{aligned}
    [\widetilde{X}^{\hbar}_{t,j}] 
    &=
    \sum_{k=1}^{N_{t\wedge \hbar}}
    e^{2\tau_{k+1} \gamma \sigma_{j}^2}
    \bigl(
    \mathbb{e}_{j}^T
    \gamma \SSigma^T
    \UU^T \PP_{k-1}
    (\UU\SSigma {\nu}_{\tau_k}^\hbar - \eeta)
    \bigr)^2.
    \end{aligned}
\]
We observe that for $\tau_k \leq \hbar$ on $\mathcal{G},$
\[
\bigl(
    \mathbb{e}_{j}^T
    \SSigma^T
    \UU^T \PP_{k-1}
    (\UU\SSigma {\nu}_{\tau_k} - \eeta)
    \bigr)^2
    \leq 
    \beta^3\max_{i}
    \bigl|
    \mathbb{e}_{j}^T
    \gamma \SSigma^T
    \UU^T \mathbb{e}_{i}
    \bigr|^2
    n^{-2\alpha}
    \leq C(T,\SSigma) \beta^3 n^{\epsilon-1-2\alpha}.
\]
Using part \textit{(i)} of Lemma \ref{lem:mgleconcentration}, we have that $\max_{0 \leq t \leq T}|\nu_{t,j}^\hbar| \leq \beta n^{\epsilon-\alpha}$ with overwhelming probability.
\end{proof}

\subsubsection{Concentration of the function values: the proof of Proposition \ref{prop:errorM}}
\begin{proof}[Proof of Proposition \ref{prop:errorM}]
Recall that for fixed $q \in [0, T]$
\[
\varepsilon^{(n)}_{\operatorname{M}}(q)
\defas
\frac{1}{2}\sum_{j=1}^d 
    \sigma_j^{2}
    \int_0^q e^{-2(q-s)\gamma\sigma_j^2}\dif M_{s,j}.
\]
Hence we can write this as
\[
\varepsilon^{(n)}_{\operatorname{M}}(q)
=
\frac{1}{2}\sum_{j=1}^d 
    \sigma_j^{2}
    e^{-2q\gamma \sigma_j^2}
    {X}_{q,j}.
\]
We consider the martingale
\[
Y_t
\defas 
\frac{1}{2}\sum_{j=1}^d 
    \sigma_j^{2}
    e^{-2q\gamma \sigma_j^2}
    {X}_{t,j},
\]
and we show concentration for $Y_t,\ 0\le t\le q$. As in the proof of Proposition \ref{prop:delocalization}, it would suffice to bound $Y_q$ for fixed $q\in [0,T]$ using the mesh arguments because the probability of having multiple jumps between two adjacent mesh points converges to zero faster than any polynomial order. Also, Proposition \ref{prop:delocalization} allows us to adopt a stopping time $\hbar,$ defined as (for some $\alpha \in (0,\tfrac 12)$)
\begin{equation}\label{eq:hbar2}
\hbar \defas \inf \left\{ t \leq \vartheta: \max_{1 \leq i \leq n} \bigl(\mathbb{e}_{i}^T(\UU\SSigma \nnu_t - \eeta)\bigr)^2 > \beta n^{-2\alpha} \right\}.
\end{equation}
As with overwhelming probability this does not occur, it suffices to show a bound for the stopped processes $\nu^{\hbar}_{t,j} \defas \nu_{t \wedge \hbar,j}.$

The jumps of this martingale are given by (see \eqref{eq:QV})
\[
\Delta Y_{\tau_k}^\hbar
=
\frac{1}{2}\sum_{j=1}^d 
    \sigma_j^{2}
    e^{-2\gamma(q-\tau_k\wedge \hbar) \sigma_j^2}
    \bigl(\nu_{\tau_k,j}^\hbar+\nu_{\tau_k-,j}^\hbar \bigr)
    \bigl(
    \mathbb{e}_{j}^T
    \gamma \SSigma^T
    \UU^T \PP_{k-1}
    (\UU\SSigma {\nnu}_{\tau_k-}^\hbar - \eeta)
    \bigr).
\]
Therefore, the quadratic variation is
\begin{align*}
    [Y_{q}^\hbar] &= \sum_{k=1}^{N_q} (\Delta Y_{\tau_k}^\hbar)^2\\
    &=\sum_{k=1}^{N_q}
    \biggl[
    \frac{1}{2}\sum_{j=1}^d 
    \sigma_j^{2}
     e^{-2\gamma(q-\tau_k\wedge \hbar) \sigma_j^2}
    \bigl(\nu_{\tau_k,j}^\hbar+\nu_{\tau_k-,j}^\hbar \bigr)
    \bigl(
    \mathbb{e}_{j}^T
    \gamma \SSigma^T
    \UU^T \PP_{k-1}
    (\UU\SSigma {\nnu}_{\tau_k-}^\hbar - \eeta)
    \bigr)
    \biggr]^2\\
    &= \frac{1}{4}\sum_{k=1}^{N_q} 
    \biggl[
    \sum_{j=1}^d 
    \sigma_j^{2}
    e^{-2\gamma(q-\tau_k\wedge \hbar) \sigma_j^2}
    \bigl(2\nu_{\tau_{k-},j}^\hbar+
    \mathbb{e}_{j}^T
    \gamma \SSigma^T \UU^T \PP_{k-1} (\UU\SSigma {\nnu}_{\tau_k-}^\hbar - \eeta)\bigr)\\
    &\qquad\qquad\qquad\cdot\bigl(\mathbb{e}_{j}^T
    \gamma \SSigma^T \UU^T \PP_{k-1} (\UU\SSigma {\nnu}_{\tau_k-}^\hbar - \eeta)
    \bigr)
    \bigg]^2\\
    &\le \frac{1}{2}\sum_{k=1}^{N_q}
    \biggl[
     \sum_{i\in B_{k-1}} \bigg( \sum_{j=1}^d \sigma_j^{2}
    e^{-2\gamma(q-\tau_k\wedge \hbar) \sigma_j^2}2\nu_{\tau_k-,j}^\hbar
    (
    \mathbb{e}_{j}^T
    \gamma \SSigma^T
    \UU^T \mathbb{e}_i )
    \biggr)
    \bigl(\mathbb{e}_i^T
    (\UU\SSigma {\nnu}_{\tau_k-}^\hbar - \eeta)
    \bigr)
    \biggr]^2 \\
    &\quad + \frac{1}{2}\sum_{k=1}^{N_q}
    \biggl[
    \sum_{j=1}^d  \sigma_j^{2}
    e^{-2\gamma(q-\tau_k\wedge\hbar) \sigma_j^2} \bigl(
    \mathbb{e}_{j}^T
    \gamma \SSigma^T
    \UU^T \PP_{k-1}
    (\UU\SSigma {\nnu}_{\tau_k-}^\hbar - \eeta)
    \bigr)^2
    \biggr]^2.
\end{align*}
Note that the second term is bounded as
\begin{align*}
    \frac{1}{2}&\sum_{k=1}^{N_q}
    \biggl[
    \sum_{j=1}^d  \sigma_j^{2}
    e^{-2\gamma(q-\tau_k\wedge \hbar) \sigma_j^2} \bigl(
    \sum_{i\in B_{k-1}}
    \bigl(\mathbb{e}_{j}^T
    \gamma \SSigma^T
    \UU^T \mathbb{e}_i\bigr)
    \bigl(
    \mathbb{e}_i^T(\UU\SSigma {\nnu}_{\tau_k-}^\hbar - \eeta)
    \bigr)^2
    \biggr]^2\\
    &\le \frac{N_q}{2}
    \bigl[
    n\cdot (\beta n^{\epsilon-1/2} \sqrt{\beta} n^{-\alpha})^2
    ]^2 
    = \frac{N_q\beta}{2n}\beta^5n^{4\epsilon-4\alpha+1}.
\end{align*}
Note that Corollary \ref{cor:entry} was used to bound $\mathbb{e}_{j}^T
    \gamma \SSigma^T
    \UU^T \mathbb{e}_i$. As for the first term, define
\begin{equation}\label{eq:Wbound}
W_{q,v,s,i}
\defas
\sum_{j=1}^d
 \nu_{s,j}^\hbar e^{-2\gamma(q-v) \sigma_j^2}
 \mathbb{e}_j^T \SSigma^T \UU^T \mathbb{e}_i,
\end{equation}
for $0\le v\le q, 0\le s\le v$. Note that it suffices to bound $\max_{0 \le s \le v} W_{q,v,s,i}$ for a fixed $v$, because we can apply the union bound to $0\le v \le q$ using the same meshing arguments as the ones used after \eqref{eq:Yt}.

Observe,
\begin{align*}
W_{q,v,s,i} &= \sum_{j=1}^d e^{-2\gamma(q-v) \sigma_j^2} \sigma_j U_{ij} \bigl(
e^{-\gamma \sigma_j^2 s}\nu_{0,j} 
    +\int_0^s e^{-\gamma \sigma_j^2
    (s-u)}\gamma \sigma_j(\UU^T\eeta)_j \dif u +  e^{-\gamma \sigma_j^2 s} \widetilde{X}_{s,j}^\hbar
\bigr)\\
&= \sum_{j=1}^d e^{-\gamma (2q-2v+s) \sigma_j^2} \sigma_j U_{ij}\nu_{0,j} 
+ \sum_{j=1}^d \int_0^{s}
e^{-\gamma (2q-2v+s-u) \sigma_j^2} \gamma \sigma_j^2 U_{ij}(\UU^T\eeta)_j \dif u \\
&\quad + \sum_{j=1}^d e^{-\gamma 2(q-v) \sigma_j^2} e^{-\gamma s \sigma^2}
\sigma_j U_{ij}\widetilde{X}_{s,j}^\hbar.
\end{align*}
Note that the first and second terms are of $\mathcal{O}(n^{\epsilon-1/2})$ with overwhelming probability by the arguments after \eqref{eq:nonMGerror}. Also, the last term is bounded by $\sqrt{\beta}n^{\epsilon-1/2}$ w.o.p. by the same arguments used in showing the bound for $Y_t$ defined in \eqref{eq:Yt}. It is crucial that the additional coefficients $e^{-\gamma 2(q-v) \sigma_j^2}$ are less than 1 so that the same arguments from \eqref{eq:DeltaYdef} to \eqref{eq:supY} work. 

Then the quadratic variation is bounded by
\[
[Y_q^\hbar] \le C N_q \bigl[
\beta \sqrt{\beta} n^{\epsilon -1/2} \sqrt{\beta} n^{-\alpha}
\bigr]^2 + \frac{N_q\beta}{2n}\beta^5n^{4\epsilon-4\alpha+1} 
\le \frac{CN_q\beta}{nT}\beta^5 n^{\epsilon-4\alpha+1},
\]
with some $C>0$ and small enough $\epsilon >0$ in the last part, which will be an enough bound to apply Lemma \ref{lem:mgleconcentration}. Hence we conclude, with the same meshing arguments used in the proof of Proposition \ref{prop:delocalization} and assumption on $\beta \le n^{1/5-\delta}$,
\[
\sup_{0\le t\le T} |Y_t^\hbar| \le \beta^{5/2} n^{\epsilon-2\alpha+1/2} \le n^{\epsilon - 5\delta/2 +(1-2\alpha)},
\]
and the claim follows by choosing $\alpha < 1/2$ sufficiently close to $1/2$.
\end{proof}

\subsubsection{Concentration of cross-variation of the model noise with the spectral weights: the proof of Proposition \ref{prop:noisecovariance}}

\begin{proof}[Proof of Proposition \ref{prop:noisecovariance}]
We recall from the assumptions of the Proposition that we let $\{ c_j \}_1^n$ be a deterministic sequence with $|c_j| \leq 1$ for all $j$.  We should show that for any $t > 0$ and for some $\epsilon > 0$
\[
\biggl|\frac{1}{\|\eeta\|^2}\sum_{j=1}^n c_j \sigma_j\nu_{t,j}^\vartheta (\UU^T\eeta)_j
-
\frac{1}{n}\sum_{j=1}^n c_j \bigl(1- e^{- \gamma \sigma_j^2(t \wedge \vartheta)}\bigr)
\biggr|
\leq \sqrt{\beta} n^{\epsilon-1/2}
\]
with overwhelming probability.  

We begin again by using Lemma \ref{lem:integrated},  due to which we have the representation
\[
    \nu_{t,j}
    =
    e^{-\gamma \sigma_j^2t}\nu_{0,j} 
    +\int_0^t e^{-\gamma \sigma_j^2(t-s)}\gamma \sigma_j(\UU^T\eeta)_j \dif s + e^{-\gamma \sigma_j^2t}\widetilde{X}_{t,j}.
\]
We replace this expression into the sum we wish to control, and observe that
\[
\sum_{j=1}^n c_j \sigma_j\nu_{t,j}^\vartheta (\UU^T\eeta)_j
=
\sum_{j=1}^n c_j \sigma_j (\UU^T\eeta)_j \biggl(e^{-\gamma \sigma_j^2t}\nu_{0,j} 
    +\int_0^t e^{-\gamma \sigma_j^2(t-s)}\gamma \sigma_j(\UU^T\eeta)_j \dif s + e^{-\gamma \sigma_j^2t}\widetilde{X}_{t,j}\biggr).
\]
Under Assumption \ref{assumption: Vector} and Corollary \ref{cor:lipschitz}, the first sum vanishes with overwhelming probability.  By independence of $\eeta$ from $\UU,$ we have that 
\[
\bigl( (\UU^T \eeta)_j :1 \leq j \leq n \bigr) \law \| \eeta\| \bigl( U_{1,j} :1 \leq j \leq n \bigr).
\]
Hence, by Lemma \ref{lem:FU}, for any $\epsilon > 0,$ with overwhelming probability, 
\[
\biggl|
\sum_{j=1}^n c_j \frac{(\UU^T\eeta)^2_j}{\|\eeta\|^2} \int_0^t e^{-\gamma \sigma_j^2(t-s)}\gamma \sigma_j^2 ds
-
\sum_{j=1}^n \frac{c_j}{n} \int_0^t e^{-\gamma \sigma_j^2(t-s)}\gamma \sigma_j^2 ds
\biggr|\leq n^{\epsilon-1/2}.
\]
Using that $\int_0^t e^{-\gamma \sigma_j^2(t-s)}\gamma \sigma_j^2 ds = (1-e^{-\gamma \sigma_j^2 t}),$ we have reduced the problem to showing that for any $\epsilon > 0$ with overwhelming probability
\[
\bigl|Y_t\bigr| \leq \sqrt{\beta} n^{\epsilon - 1/2}
\quad\text{where}\quad
Y_s \defas 
\frac{1}{\|\eeta\|^2}
\sum_{j=1}^n c_j \sigma_j(\UU^T\eeta)_je^{-\gamma \sigma_j^2t}\widetilde{X}_{s,j}
\quad \text{for all } s\leq t.
\]
To leverage Proposition \ref{prop:delocalization}, we again use the stopping time $\hbar$ \eqref{eq:hbar}. We will again apply Lemma \ref{lem:mgleconcentration}.  
The jumps of $Y_t^{\hbar}$ are given by, for any $\tau_k \leq s,$
\[
    \Delta Y^{\hbar}_{\tau_k} 
    =
    -\frac{1}{\|\eeta\|^2}
    \sum_{j=1}^n c_j \sigma_j(\UU^T\eeta)_je^{-\gamma \sigma_j^2t}
    \bigl(
    \mathbb{e}_{j}^T
    \gamma \SSigma^T
    \UU^T \PP_{k-1}
    (\UU\SSigma {\nu}_{\tau_k-}^\hbar - \eeta)
    \bigr).
\]
If we let $\DD$ be the diagonal matrix with entries $-\gamma c_je^{-\gamma \sigma_j^2 t},$ then we have the representation
\[
    \Delta Y^{\hbar}_{\tau_k}
    =
    \frac{1}{\|\eeta\|^2}
    \eeta^T \UU \DD
    \SSigma\SSigma^T
    \UU^T \PP_{k-1}
    (\UU\SSigma {\nu}_{\tau_k-}^\hbar - \eeta).
\]
Using that $\frac{1}{\|\eeta\|^2}
\eeta^T \UU \DD
\SSigma\SSigma^T
\UU^T$ has a norm bounded only in terms of $t,\SSigma$ and $\gamma,$ it follows that
\begin{equation}\label{eq:sqincr0}
|\Delta Y^{\hbar}_{\tau_k}| \leq C(t,\SSigma,\gamma)\beta n^{-\alpha}.
\end{equation}

We turn to bounding the predictable quadratic variation.  Using Lemma \ref{lem:subsetsum},
\[
\begin{aligned}
\E ( (\Delta Y^{\hbar}_{\tau_k})^2~|~ \mathcal{F}_{\tau_k-})
&= \frac{\beta(\beta-1)}{n(n-1)} 
\biggl(\frac{1}{\|\eeta\|^2}
    \eeta^T \UU \DD
    \SSigma\SSigma^T
    \UU^T
    (\UU\SSigma {\nu}_{\tau_k-}^\hbar - \eeta)\biggr)^2 \\
    &+
    \biggl(
    \frac{\beta}{n} 
    -\frac{\beta(\beta-1)}{n(n-1)} \biggr)
    \sum_{i=1}^n
    \biggl(
    \frac{1}{\|\eeta\|^2}
    \eeta^T \UU \DD
    \SSigma\SSigma^T
    \UU^T \mathbb{e}_i\biggr)^2
    \biggl(
    \mathbb{e}_i^T
    (\UU\SSigma {\nu}_{\tau_k-}^\hbar - \eeta)
    \biggr)^2.
    \end{aligned}
\]
The first line we bound using that $\UU\SSigma {\nu}_{\tau_k-}^\hbar - \eeta$ is norm at most $n^{\epsilon}$ and that $\UU \DD
    \SSigma\SSigma^T
    \UU^T$ has a norm bounded only by some $C(t,\SSigma,\gamma).$
The second line we bound using that 
\(
(\mathbb{e}_i^T
(\UU\SSigma {\nu}_{\tau_k-}^\hbar - \eeta)
)^2 \leq \beta n^{-2\alpha}.
\)  
Together, these bounds give that
\[
\E ( (\Delta Y^{\hbar}_{\tau_k})^2~|~ \mathcal{F}_{\tau_k-})
\leq C(t,\SSigma,\gamma)\biggl( \beta^2 n^{-2+\epsilon} + \beta^2 n^{-1-2\alpha}\biggr).
\]
Hence the conclusion follows using Lemma \ref{lem:mgleconcentration}.
\end{proof}

\section{Analyzing the Volterra equation} \label{app: Avg_case}
\subsection{General analysis of the Volterra equation} \label{app: general_analysis_volterra}
In this section, we analyze the solution of the Volterra equation \eqref{eq:Volterra_eq_main} and give some basic properties of its solution for general limiting spectral measures $\mu$ which is a compactly supported measure on $[0,\infty)$.
We recall for convenience that the Volterra equation is given by
\begin{equation}\label{eq:Volterra2}
\begin{gathered}
    \psi_0(t) = \tfrac{R}{2} h_1(t) + \tfrac{\widetilde{R}}{2} \big ( rh_0(t)+ (1-r) \big ) + \gamma^2r \int_0^t  h_2(t-s)\psi_0(s)\,\dif s,\\
    \quad \text{and} \quad h_k(t) = \int_0^\infty x^k e^{-2\gamma tx} \, \dif \mu(x),
\end{gathered}
\end{equation}
where $\gamma > 0$ is a stepsize parameter.  When convenient, we will simply write $z(t)$ for the forcing function
\begin{equation}\label{eq:zt}
z(t) \defas \tfrac{R}{2} h_1(t) + \tfrac{\widetilde{R}}{2} \big ( rh_0(t)+ (1-r) \big ).
\end{equation}
The parameter $r \in (0,\infty)$ is fixed, but we may consider limits of the Volterra equation under various limits.  The parameters $R$ and $\widetilde{R}$ are both non-negative, but to avoid trivialities, we should assume at least one is positive.  As the Volterra equation is linear, we may without loss of generality assume $R+\widetilde{R} = 1$.

The equation \eqref{eq:Volterra2} appears frequently in the probability literature as the \emph{renewal equation} \cite[(3.5.1)]{Resnick}, \cite[(2.1)]{Asmussen}; it appears naturally in renewal theory and in the Lotka population model, amongst others, which are neatly described in the references just mentioned.  It allows, for example, $\psi_0$ to be given the amusing interpretation as the expected size of a population which evolves in times (c.f.\ \cite[Example 3.5.2]{Resnick} or \cite[Example 2.2]{Asmussen}).  Much of the behavior of the equation is determined by the properties of the function $\gamma^2 r h_2(t)$.  We will let $\lambda^{-}$ be the leftmost endpoint of the support of $\mu$ restricted to $(0,\infty)$, and we record the following elementary computation.  For any $\alpha \in \RR,$
\begin{equation}\label{eq:h2mass}
\int_0^\infty e^{2\gamma \alpha t} \gamma^2r h_2(t)\, \dif t
=
\begin{cases}
\frac{\gamma r}{2}\int_0^\infty \frac{x^2}{x-\alpha} \dif \mu(x),
& \text{if } \alpha \leq \lambda^- \\
\infty &\text{otherwise.}
\end{cases}
\end{equation}
We begin by observing some elementary properties of the equation:
\begin{theorem}\label{thm:VolterraTrivialities}
    There is a unique, positive solution to \eqref{eq:Volterra_eq_main} which exists for all time.  The solution is bounded if and only if $\gamma < \tfrac 2r \bigl( \int_0^\infty x \dif \mu(x) \bigr)^{-1}$ 
    in which case
    \[
    \psi_0(\infty)
    \defas
    \lim_{t \to \infty} \psi_0(t) = \frac{\widetilde{R}}{2} \cdot \frac{r\mu(\{0\}) + (1-r)}{1-\tfrac{\gamma r}{2}\bigl( \int_0^\infty x \dif \mu(x) \bigr)}.
    \]
\end{theorem}
\begin{proof}
In standard renewal notation (c.f.\  \cite[(3.5.1)]{Resnick}, \cite[(2.1)]{Asmussen}), we would write \eqref{eq:Volterra2}
\[
\psi_0 = z + \psi_0 * F
\]
where $F$ is the function
\[
F(t) = \int_0^t \gamma^2 r h_2(s) \dif s.
\]
The existence and uniqueness is now standard (compare with Proposition \ref{prop:volterrastability}), see \cite[Theorem 3.5.1]{Resnick} or \cite[Theorem 2.4]{Asmussen}.  By \eqref{eq:h2mass},
\[
F(\infty) = \frac{\gamma r}{2}\int_0^\infty x \dif \mu(x).
\]
Hence when this is bigger than $1,$ the solution $\psi_0(t)$ tends to infinity exponentially fast \cite[Theorem 7.1]{Asmussen} or \cite[Proposition 3.11.1]{Resnick}.  In the case that 
\(
\frac{\gamma r}{2}\int_0^\infty x \dif \mu(x)=1,
\)
by Blackwell's Renewal theorem, $\psi_0(t)$ is asymptotic to a positive multiple of $t$ (see \cite[Theorem 3.10.1]{Resnick} or \cite[Theorem 4.4]{Asmussen}) and hence still diverges.  Finally, in the case that 
\(
\frac{\gamma r}{2}\int_0^\infty x \dif \mu(x)<1
\)
by \cite[Section 3.11]{Resnick} or \cite[Proposition 7.4]{Asmussen},
\begin{equation}\label{eq:psi0infty}
\lim_{t \to \infty} \psi_0(t) = \frac{\lim_{t \to \infty} z(t)}{1- \frac{\gamma r}{2}\int_0^\infty x \dif \mu(x)},
\end{equation}
which is the claimed result.
\end{proof}

\subsubsection*{Two phases}

Hence, we will assume going forward that $\gamma < \gamma_0 \defas \bigl( \tfrac r2 \int_0^\infty x \dif \mu(x) \bigr)^{-1}$.
We shall see that it is possible to say more about the rate of convergence in general.  
Define the \emph{Malthusian exponent} $\lambda^*$ as the solution of
\begin{equation}
    \int_0^\infty e^{2\gamma \lambda^* t} \gamma^2 r h_2(t)\, \dif t
    = 1,
\end{equation}
when it exists.  Note by virtue of \eqref{eq:h2mass}, if this exponent exists, it can just as well be defined as the solution of
\begin{equation}\label{eq:Malthusian2}
\frac{r}{2}\int_0^\infty \frac{x^2}{x-\lambda^*} \dif \mu(x) = \frac{1}{\gamma},
\end{equation}
and we necessarily have that $\lambda^* \leq \lambda^{-}$.  Define
\begin{equation}
    \gamma_* = \frac{1}{\frac{r}{2}\int_0^\infty \frac{x^2}{x-\lambda^{-}} \dif \mu(x)}
\end{equation}
which exists and is positive exactly when $\int_0^\infty \frac{x^2}{x-\lambda^{-}} \dif \mu(x) < \infty$.  Note that $\gamma_*$ is strictly less than $\gamma_0$ if and only if $\lambda^- > 0$.  Moreover, we can completely give the asymptotic behavior of the Volterra equation on either side of the critical point.  Although, to do this for $\gamma < \gamma_*,$ we will need some further assumptions on $\mu$. 

Recall that a function $f : (0,\infty) \to \RR$ is \emph{slowly varying} if $f(tx)/f(x) \to 1$ as $t\to \infty$ for any $x > 0$. A function $f : (0,\infty) \to \RR$ is \emph{regularly varying} if $f(t) = g(t) t^{\alpha}$ for a slowly varying function $g.$   
We will say that $\mu$ is \emph{left-edge-regular} if there exists a regularly varying function $L$ and $\alpha > 0$ so that 
\begin{equation}
t\mapsto \mu((\lambda^-, \lambda^-+t]) \sim t^\alpha L(\tfrac{1}{t}), \quad\text{as }t \to \infty,
\end{equation}
which for example is satisfied by Marchenko-Pastur \eqref{eq:MP}.
We show the following:
\begin{theorem}\label{thm:GVolterra_stronglyconvex}
For $\gamma \in (\gamma_*,\gamma_0),$ the Malthusian exponent $\lambda^*$ exists and is the unique solution of \eqref{eq:Malthusian2}.  The function $\psi_0(t)$ satisfies that for some explicit constant $c(R,\widetilde{R},\mu)$,
\[
\psi_0(t) - \psi_0(\infty)
\sim \frac{c(R,\widetilde{R},\mu)}{\gamma} e^{-2\gamma (\lambda^*)t}. 
\]
If in addition $\gamma_* > 0$ and $\mu$ is left-edge-regular, then for $\gamma \in (0, \gamma_*)$
\[
\psi_0(t) - \psi_0(\infty)
\sim e^{-2\gamma (\lambda^-)t} g(t)
\]
where $g(t)$ is some explicit regularly varying function.
\end{theorem}
\noindent Thus if one considers varying the stepsize $\gamma$ from $0$ up to $\gamma_*$ the process undergoes a transition in behavior when $\gamma=\gamma_*$.  For small $\gamma,$ the exponential rate of change is frozen on the smallest eigenvalue of the Hessian $\lambda^{-}$. However as $\gamma$ passes the transition point $\gamma_*,$ the logarithm of the rate becomes a smooth function.  This is strongly reminiscent of a \emph{freezing transition}, which is often seen in the free energies of random energy models.  See for example \cite{FyodorovBouchaud}.

The proof is essentially an automatic consequence of established theory for Volterra equations.  As an input to the case of $\gamma < \gamma_*$, we need the following asymptotics of the functions $h_k$, which are the main way in which left-edge-regularity:
\begin{lemma}\label{lem:hs}
Suppose that $\mu$ has left-edge-regularity, meaning that there is an $\alpha \geq 0$ and slowly varying function $L$ so that 
\[
\mu((\lambda^-, \lambda^-+t]) \sim t^{\alpha}L(\tfrac{1}{t})
\quad\text{as}\quad t \to 0.
\]
If $\lambda^- > 0$ then for any $k \geq 0$
\[
h_{k}(t) - (\lambda^{-})^k e^{-2\gamma \lambda^{-}t}\mu( \{\lambda^{-}\}) \sim e^{-2\gamma \lambda^{-}t}t^{-\alpha} L(t) \Gamma(\alpha+1) (\lambda^{-})^k. 
\]
If $\lambda^- = 0$ then for any $k \geq 0,$
\[
h_{k}(t)- \mathbb{1}_{k=0} \mu( \{0\}) \sim t^{-k-\alpha} L(t) \Gamma(k+\alpha+1). 
\]
\end{lemma}
This is a standard exercise, and we do not show its proof.
\begin{proof}[Proof of Theorem \ref{thm:GVolterra_stronglyconvex}]
\noindent \textbf{The case of $\gamma \in (\gamma_*,\gamma_0)$.} We follow the notation of \cite[Theorem 7.1]{Asmussen} (see also \cite[Proposition 3.11.1]{Resnick}).  Before beginning, we observe that the Malthusian exponent does exist for this region, as the function
\(
\alpha \mapsto \int_0^\infty \frac{x^2}{x-\alpha} \dif \mu(x),
\)
is an increasing continuous function on $(-\infty, \lambda^{-})$.  Hence by the definition of $\gamma_*,$ the image of this function applied to $[0,2\gamma \lambda^{-})$ is all of $[\gamma_0^{-1}, \gamma_*^{-1})$.  From \cite[Proposition 7.6]{Asmussen}
\[
\lim_{t\to \infty} e^{2\gamma \lambda^{*}t}(\psi_0(t)-\psi_0(\infty))
=
\frac{ \int_0^\infty e^{2\gamma \lambda^{*}t}(z(t)-z(\infty))\, \dif t - \tfrac{z(\infty)}{\beta}}
{
\int_0^\infty te^{2\gamma \lambda^{*}t}\gamma^2rh_2(t)\, \dif t
}.
\]
We evaluate these two integrals for convenience.  Using the definition of $z$ in \eqref{eq:zt}
\[
\int_0^\infty e^{2\gamma \lambda^{*}t}(z(t)-z(\infty))\, \dif t
-\frac{z(\infty)}{\beta}
=
\frac{R}{4\gamma}\int_0^\infty \frac{x\dif\mu(x)}{x-\lambda^*} 
+
\frac{\widetilde{R}r}{4\gamma}\int_{0}^\infty \frac{ \dif \mu(x)}{x-\lambda^*}
-
\frac{\widetilde{R}}{4\gamma\lambda^*}(1-r).
\]
For the denominator,
\[
\begin{aligned}
\int_0^\infty te^{2\gamma \lambda^{*}t}\gamma^2rh_2(t)\, \dif t
&=
\gamma^2 r \int_0^\infty\int_0^\infty x^2te^{2\gamma (\lambda^{*}-x)t} \dif t \dif \mu(x) \\
&=
\frac{r}{4}\int_0^\infty \frac{x^2}{(x-\lambda^{*})^2}\dif \mu(x).
\end{aligned}
\]

\noindent \textbf{The case of $\gamma \in (0,\gamma_*)$.} By assumption we have that $\gamma_* > 0$, and hence 
\[
\int_0^\infty \frac{x^2}{x-\lambda^{-}} \dif \mu(x) < \infty.
\]
Set $F(t) = \int_0^t \gamma^2 rh_2(s)\, \dif s$. 
Using that $\psi_0(\infty)=\frac{z(\infty)}{1-F(\infty)}$ (see \eqref{eq:psi0infty}),
\[
\psi_0(\infty) = z(\infty)\biggl(\tfrac{1-F(t)}{1-F(\infty)}\biggr) + F(t)\psi_0(\infty),
\]
and hence the constant function $\psi_0$ solves the Volterra equation \eqref{eq:Volterra2} with forcing function $z(\infty)\biggl(\tfrac{1-F(t)}{1-F(\infty)}\biggr)$.  It follows that 
\begin{equation}\label{eq:Volterra3}
\psi_0(t)-\psi_0(\infty) = z(t) - z(\infty)\biggl(\tfrac{1-F(t)}{1-F(\infty)}\biggr)
+\gamma^2 r\int_0^t h_2(t-s) (\psi_0(s)-\psi_0(\infty))\, \dif s.
\end{equation}
Define
\[
\widehat{Z}(t) = e^{2\gamma \lambda^{-} t}\bigl( \psi_0(t) - \psi_0(\infty)\bigr),
\quad
\widehat{z}(t) = e^{2\gamma \lambda^{-} t}\bigl( z(t) - z(\infty)\tfrac{1-F(t)}{1-F(\infty)}\bigr)
\quad
\text{and}
\quad
\frac{\dif \widehat{F}(t)}{\dif t} = \gamma^2 r e^{2\gamma \lambda^{-} t} h_2(t).
\]
Then using \eqref{eq:Volterra3},
\[
\widehat{Z}(t)
=
\widehat{z}(t)
+\int_0^t \frac{\dif \widehat{F}(s)}{\dif s} \widehat{Z}(t-s)\, \dif s.
\]

By the assumption that $\gamma < \gamma_*,$ we have that
\[
\theta \defas \int_0^\infty \gamma^2 r e^{2\gamma \lambda^{-} t} h_2(t) \, \dif t < 1.
\]
In preparation to apply \cite[Theorem 5]{Asmussen2003}, 
we need to evaluate the ratio of the limits of the densities
\begin{equation}\label{eq:prelimit0}
\lim_{t \to \infty} 
\frac{ \widehat{z}(t)}{\widehat{F}'(t)}
=
\lim_{t \to \infty} 
\frac{e^{2\gamma \lambda^{-} t}\bigl( z(t) - z(\infty)\tfrac{1-F(t)}{1-F(\infty)}\bigr)}{\gamma^2 r e^{2\gamma \lambda^{-} t} h_2(t)}
=
\lim_{t \to \infty} 
\biggl(\frac{z(t) - z(\infty)}{\gamma^2 r h_2(t)}
+
\frac{z(\infty)}{\gamma^2 r}\frac{F(t)-F(\infty)}{h_2(t)(1-F(\infty))}\biggr).
\end{equation}
To simplify this, we observe that
\[
F(\infty) - F(t) = \frac{\gamma r}{2} \int_0^\infty x e^{-2\gamma t x} \dif \mu(x) = \frac{\gamma r}{2} h_1(t).
\]
We also observe that
\[
z(t) - z(\infty)=\frac{R}{2} h_1(t) + \frac{\widetilde{R}r}{2} h_{0^+}(t)
\quad\text{where}\quad
h_{0^+}(t)=\lim_{\epsilon\downarrow 0} \int_\epsilon^\infty e^{-2\gamma t x} \dif \mu(x).
\]
We conclude that 
\begin{equation}\label{eq:prelimit1}
\lim_{t \to \infty} 
\frac{ \widehat{z}(t)}{\widehat{F}'(t)}
=
\lim_{t \to \infty} 
\biggl(\frac{R h_1(t) + {\widetilde{R}r} h_{0^+}(t)}{2\gamma^2 r h_2(t)}
-
\frac{\psi_0(\infty)}{2\gamma}\frac{h_1(t)}{h_2(t)}\biggr).
\end{equation}
Hence we have from \eqref{eq:prelimit1}
\begin{equation}\label{eq:prelimit2}
\lim_{t \to \infty} 
\frac{ \widehat{z}(t)}{\widehat{F}'(t)}
=
\frac{(R-{\gamma r}\psi_0(\infty)) \lambda^{-} + {\widetilde{R}r}}{2\gamma^2 r (\lambda^-)^2}
\defas c_*.
\end{equation}

By the assumption on $\alpha,$ it can be checked that $\int_0^\infty \widehat{z}(t)dt < \infty.$
Moreover it follows that $\widehat{z}(t)$ is subexponential as $\alpha > 0$ (see \cite[Section 3]{Asmussen2003}), and 
hence by \cite[Theorem 5 (ii)]{Asmussen2003},
\[
\begin{aligned}
\widehat{Z}(t) 
&\sim 
\biggl(\frac{\int_0^\infty \widehat{z}(t)\dif t}{(1-\theta)^2} + \frac{c_*}{1-\theta} \biggr) \widehat{F}'(t)
\sim
\biggl(\frac{\int_0^\infty \widehat{z}(t)\dif t}{(1-\theta)^2} + \frac{c_*}{1-\theta} \biggr)
\gamma^2r e^{2\gamma \lambda^{-} t} h_2(t). \\
&\sim
\biggl(\frac{\int_0^\infty \widehat{z}(t)\dif t}{(1-\theta)^2} + \frac{c_*}{1-\theta} \biggr)
\gamma^2r \biggl( (\lambda^-)^2\mu( \{\lambda^{-}\}) + t^{-\alpha} L(t) \Gamma(\alpha+1)(\lambda^{-})^2\biggr). \\
\end{aligned}
\]
The second line follows from Lemma \ref{lem:hs}.
\end{proof}

We finish by observing that when $\lambda^{-} = 0,$ another behavior takes hold.
\begin{theorem}\label{thm:GVolterra_convex}
Suppose that $\lambda^{-} = 0,$ and that the measure $\mu$ has left-edge-regularity, meaning that there is an $\alpha > 0$ and slowly varying function $L$ so that 
\[
\mu((\lambda^-, \lambda^-+t]) \sim t^{\alpha}L(\tfrac{1}{t})
\quad\text{as}\quad t \to 0.
\]
Then
\[
\psi_0(t)-\psi_0(\infty) \sim
\frac{1}{1-\frac{\gamma r}{2}\int_0^\infty x \dif \mu(x)}
\begin{cases}
\frac{\widetilde{R}r}{2}
t^{-\alpha} L(t) \Gamma(1+\alpha) &\text{if }\widetilde{R} > 0, \\ 
\frac{R}{2}
t^{-1-\alpha} L(t) \Gamma(2+\alpha) &\text{if }\widetilde{R} = 0. 
\end{cases}
\]
\end{theorem}
\begin{proof}
We again apply \cite[Theorem 5]{Asmussen2003}, and so we use the same change of variables as in Theorem \ref{thm:GVolterra_stronglyconvex}.  We once more must compute \eqref{eq:prelimit1}, which by Lemma \ref{lem:hs} is now equal to $\infty$.
Since $\gamma < \gamma_0,$
\[
\theta \defas \int_0^\infty \gamma^2 r h_2(t)\, \dif t = \frac{\gamma r}{2}\int_0^\infty x\dif \mu(x) < 1.
\]
Hence by \cite[Theorem 5 (iii)]{Asmussen2003},
\[
\widehat{Z}(t) 
\sim \frac{1}{1-\theta} \widehat{z}(t)
\sim
\frac{1}{1-\theta}
\begin{cases}
\frac{\widetilde{R}r}{2}\int_{0+}^\infty e^{-2\gamma tx} \dif \mu(x) &\text{if }\widetilde{R} > 0, \\ 
\frac{R}{2}\int_{0}^\infty xe^{-2\gamma tx} \dif \mu(x) &\text{if }\widetilde{R} = 0.
\end{cases}
\]
Then by Lemma \ref{lem:hs}, the proof is complete.
\end{proof}

 \subsection{Explicit solution of the Volterra equation for Isotropic Features} \label{app: explicit_sol_for_volterra}
In this section, we solve the Volterra equation for $\psi_0(t)$ in \eqref{eq:Volterra_eq_main} when $\dif \mu$ satisfies the Marchenko-Pastur law in \eqref{eq:MP}. Throughout this section we use the following change of variables
\[ \widehat{\psi}_0(t) \defas 2 \psi_0\big ( \tfrac{t}{2 \gamma} \big ). \]
Under this change of variables, the Volterra equation in \eqref{eq:Volterra_eq_main} becomes 
\begin{equation} \begin{aligned} \label{eq:Old_Volterra_Equation}
\widehat{\psi}_0(t) &= R \cdot h_1 \big ( \tfrac{t}{2 \gamma} \big ) + \widetilde{R} \big ( r h_0 \big ( \tfrac{t}{2 \gamma} \big ) + (1-r) \big ) + \frac{r \gamma}{2} \int_0^t h_2 \big ( \tfrac{1}{2\gamma} (t-s) \big ) \widehat{\psi}_0(s) \, \dif s\\
&= R \cdot \widehat{h}_1(t) + \widetilde{R} \big (r \widehat{h}_0(t) + (1-r) \big ) +  \int_0^t k(t-s) \widehat{\psi}_0(s) \, \dif s,  
\end{aligned} \end{equation}
where we set $\widehat{h}_k$ a scaled version of $h_k$ and the kernel $k$ as
\begin{equation} \label{eq:k_and_h}
    \widehat{h}_k(t) \defas h_k \big ( \tfrac{t}{2\gamma} \big ) \quad \text{and} \quad k(t) \defas \frac{r\gamma}{2} h_2 \big ( \tfrac{t}{2 \gamma} \big ).
\end{equation}
Volterra equations of convolution type can. be solved trivially by using Laplace transforms which conveniently in the case of Marchenko-Pastur, we do have. Explicit formulas for the Laplace transforms of $h_k(t)$ via the Stieltjes transform of $\MP$ exist. 

We now solve for $\widehat{\psi}_0$ using Laplace transforms. We let $\Psi(p)$ and $K(p)$ be the Laplace transforms of $\widehat{\psi}_0(t)$ and $k(t)$ respectively. We can relate $\widehat{h}_1(t)$ in \eqref{eq:k_and_h} to the function $k$ and hence it's Laplace transform by the following
\begin{equation}
    \partial_t \widehat{h}_1(t) = -R \int_0^\infty x^2 e^{-tx} \, \dif \MP(x) = - \frac{2R}{r \gamma} k(t) \quad \text{and} \quad \mathcal{L} \{ \widehat{h}_1(t)\} = \frac{R \left (1-  \frac{2}{r \gamma} K(p) \right )}{p},
\end{equation}
where we used that the first moment of $\MP$ is $1$ \citep{bai2010spectral}. 
We now define the function $T(t)$ to be the Laplace transform of Marchenko-Pastur and the Laplace transform of $T$ (\textit{i.e.}, the Laplace transform of the Laplace transform of $\mu_{\MP}$), otherwise known as the Stieltjes transform, as the following
\begin{equation} \label{eq:SGD_cp_6}
   T(t) \defas \int_0^\infty e^{-xt} \dif \MP \quad \text{and} \quad \mathcal{L}\{T(t)\}(p) = \frac{-p + r-1 - \sqrt{(-p-r -1)^2 - 4r}}{2rp}.
\end{equation}
It is clear that the Laplace transform for $\widehat{h}_0$ is given by using  $\mathcal{L}\{T(t)\}$. From the Volterra equation \eqref{eq:Old_Volterra_Equation} and the function $\widehat{h}_0(t)$ \eqref{eq:k_and_h}, we get the following expression for $\Psi(p):$
\begin{equation} 
 \label{eq:SGD_cp_2} \begin{gathered}
    \Psi(p) = \frac{R \left (1-  \frac{2}{r\gamma} K(p) \right )}{p} + K(p) \Psi(p) + \widetilde{R} \bigg ( r \mathcal{L}\{T(t)\}(p) +  \frac{(1-r)}{p}\bigg ) \\
    \Psi(p) =  \frac{\frac{R \left (1-  \frac{2}{r\gamma} K(p) \right )}{p}+ \widetilde{R} \left (r\mathcal{L}\{T(t)\}(p) +  \frac{(1-r)}{p}  \right ) }{1-K(p)}. 
\end{gathered} \end{equation}
We now turn to giving an explicit expression for $\Psi(p)$. We begin with the following lemma relating the integral, $\int \frac{x}{x+p} \, \dif \MP$, to the Stieltjes transform of the semi-circle law. We let $m$ be the Stieltjes transform for semi-circle law
\begin{equation} \label{eq:m}
m(z)\defas \frac{1}{2\pi}\int_{-2}^2 
\frac{\sqrt{4-y^2}}{y - z} \, \dif y,
\end{equation}
and we record a few well-known properties of this Stieltjes transform:
\begin{lemma}\label{lem:semicirclestieltjes}
The Stieltjes transform $m$ can be expressed as
\[
m(z) = \frac{-z + \sqrt{z^2-4}}{2},
\]
for $z \in \mathbb{C}$ with $\Im z > 0,$ and it maps to the upper half plane (in fact to the upper half-disk). This can be extended to $z \in \mathbb{R} \setminus [-2,2]$ by continuity, and to the lower half plane by conjugation symmetry.
The function $m$ is the solution of the functional equation
\begin{equation}\label{eq:stfixedpt}
m(z) + \frac{1}{m(z)} = -z
\quad \text{for all}\quad \Im z > 0.
\end{equation}
Hence, we define the conjugate of $m$ as 
\begin{equation}\label{eq:stconjugate}
\widehat{m}(z) \defas \frac{1}{m(z)}
= \frac{-z - \sqrt{z^2-4}}{2}.
\end{equation}
Moreover the Stieltjes transform of $m$ is related to the Marchenko-Pastur by the identity for $p \in \mathbb{C} \setminus [\lambda^-, \lambda^+]$ 
\begin{equation} \label{eq:m_and_MP}
    \int_0^\infty \frac{x}{x + p} \, \dif \MP(x) = \frac{m(q)}{\sqrt{r}},
\end{equation}
where we set $q \defas - \frac{p+1+r}{\sqrt{r}}$. 
\end{lemma}

\begin{proof} The results regarding the Stieltjes transform $m$ are well-known and we refer the reader to \citep{bai2010spectral}. It remains to prove \eqref{eq:m_and_MP} relating $m$ to the Marchenko-Pastur. First we observe that
\[ \int_0^\infty \frac{x}{x + p} \, \dif \MP(x) = \frac{1}{2 \pi r} \int_{\lambda^-}^{\lambda^+} \frac{\sqrt{(x-\lambda^-)(\lambda^+-x)}}{x+p} \, \dif x. \]
We recenter and rescale by sending $x = \frac{\lambda^- + \lambda^+}{2} + \frac{\lambda^+-\lambda^-}{4} y$, so that the follow holds
\[ \sqrt{(x-\lambda^-)(\lambda^+-x)} = \sqrt{r} \cdot \sqrt{4-y^2}. \]
Using this change of variables and noting that $\tfrac{\lambda^+-\lambda^-}{4} = \sqrt{r}$, $\dif x = \sqrt{r} \, \dif y$, and $\tfrac{\lambda^+ + \lambda^-}{2} = 1+r$, we deduce that
\begin{align*}
    \int_0^\infty \frac{x}{x+p} \, \dif \MP(x) &= \frac{1}{2\pi} \int_{-2}^2 \frac{\sqrt{4-y^2}}{p + \left ( \tfrac{\lambda^-+\lambda^+}{2} + \frac{\lambda^+-\lambda^-}{4} y \right )} \, \dif y\\
    &= \frac{1}{2 \pi \sqrt{r}} \int_{-2}^2 \frac{\sqrt{4-y^2}}{y - \left (\tfrac{-p-(1+r)}{\sqrt{r}} \right )} \, \dif y.
\end{align*}
The result follows after noting the definition of $m(z)$ and $q$.
\end{proof}

By exploiting the relationship between Marchenko-Pastur and the Stieltjes transform $m$ defined in \eqref{eq:m}, we can evaluate some expressions against Marchenko-Pastur. To do so, it will be important to define the following quantities
\begin{equation} \label{eq:rho_omega}
\varrho = \frac{1+r}{2} \left ( 1- \frac{r \gamma}{2} \right ) \quad \text{and} \quad \omega = \frac{1}{4} \left (1-\frac{r\gamma}{2} \right )^2 \left (\frac{8}{\gamma} - (1+r)^2 \right ) . 
\end{equation}

\begin{lemma}[Marchenko-Pastur Integrals] \label{lem:magical} Suppose the constants $\varrho$ and $\omega$ are as in \eqref{eq:rho_omega} and fix the stepsize $0< \gamma < \frac{2}{r}$. Define a critical stepsize $\gamma_*$ as
\begin{equation} \label{eq:critical_gamma_1}
\gamma_* \defas \frac{2}{\sqrt{r}(r-\sqrt{r}+1)}.
\end{equation}
It then follows that 
\begin{equation}
  \int_0^\infty \frac{x}{x+p_*} \, \dif \MP(x) = \begin{cases}
  \left (1 - \frac{r\gamma}{2} \right )^{-1} \frac{\gamma(\varrho +i \sqrt{\omega})}{2}, & \text{if $p_* = -\varrho - i \sqrt{\omega}$ and $\gamma < \tfrac{2}{r}$}\\
   \left (1 - \frac{r\gamma}{2} \right )^{-1} \frac{\gamma(\varrho -i \sqrt{\omega})}{2}, & \text{if $p_* = -\varrho + i \sqrt{\omega}$ and $\gamma \le \gamma_*$}\\
   \left (1 - \frac{r\gamma}{2} \right ) \frac{2(\varrho + i \sqrt{\omega})}{r\gamma(\varrho^2 + \omega)},& \text{if $p_* = - \varrho + i \sqrt{\omega}$ and $\gamma_* < \gamma < \tfrac{2}{r}$.}
\end{cases} 
\end{equation}
\end{lemma}

\begin{proof} First suppose that $\omega < 0$,  then the follow holds
\[
  \varrho + \sqrt{|\omega|}
  = \frac{1}{2}
  \left( 1-\frac{r \gamma }{2} \right) \left( 1+r+\sqrt{(1+r)^2-\tfrac{8}{\gamma}}\right).
\]
We wish to show exactly when this quantity is equal to $(1-\sqrt{r})^2$ as this will give us the critical $\gamma_*$.
Let $x=\sqrt{(1+r)^2-\tfrac{8}{\gamma}}$ and observe that $1+r-x \geq 0.$
Hence we have that
\[
\begin{aligned}
2(\varrho + \sqrt{|\omega|} - (1-\sqrt{r})^2)(1+r - x)
&=\left( \tfrac{8}{\gamma}-4r \right)
-2(1-\sqrt{r})^2(1+r-x) \\
&= \left( (1+r)^2-4r-{x^2} \right)
-2(1-\sqrt{r})^2(1+r-x) \\
&= -(x - (1-\sqrt{r})^2)^2.
\end{aligned}
\]
Thus $\varrho+\sqrt{|\omega|} < (1-\sqrt{r})^2$ except at a single value of $\gamma$ at which
\begin{equation}\label{eq:criticalpoint}
(1+r)^2 - \tfrac{8}{\gamma} = (1-\sqrt{r})^4
\quad\Longleftrightarrow\quad \gamma = \gamma_* \defas \frac{2}{\sqrt{r}(r-\sqrt{r}+1)}.
\end{equation}

Now we let $p_*$ be either of $-\varrho \pm i\sqrt{\omega},$ where we take the branch of the square root continuous in the closed upper half plane (where $\omega \in \mathbb{C}$). 
We use the identity for $p$ in $\mathbb{C} \setminus [\lambda^-,\lambda^+]$,
\[
\int_{0}^\infty \frac{x}{x+p} \, \dif \MP(x)
=\frac{m(q)}{\sqrt{r}},
\]
where we recall $-q=\frac{p+1+r}{\sqrt{r}}.$
We will apply this at $p_*,$ and we set $-q_*=\frac{p_*+1+r}{\sqrt{r}}$.

We use the identity
\[
\frac{\gamma}{2}[(p+\varrho)^2+\omega]
=
\bigl(1- \frac{r\gamma}{2} + \frac{r\gamma p}{2\sqrt{r}}m(q)\bigr)
\bigl(
1- \frac{r\gamma}{2} + \frac{r\gamma p}{2\sqrt{r}}\widehat{m}(q)
\bigr).
\]
In particular, evaluating this identity at $p=p_*$ the left-hand-side is $0$, and we conclude that either
\begin{equation}\label{eq:thechoice}
\bigl(1- \frac{r\gamma}{2} + \frac{r\gamma p_*}{2\sqrt{r}}m(q_*)\bigr)
=0
\quad\text{or}\quad
\bigl(1- \frac{r\gamma}{2} + \frac{r\gamma p_*}{2\sqrt{r}}\frac{1}{m(q_*)}\bigr)
=0
\end{equation}
When $\omega > 0,$ then as $1-\frac{r\gamma}{2} > 0,$ the correct choice is dictated by having either $p_*m(q_*) < 0$ or $\tfrac{p_*}{m(q_*)} < 0.$
If $\Im p_* = -\Im q_* < 0,$ we have $\Im m(q_*) > 0,$ and so it must be the
second of these two identities in \eqref{eq:thechoice}. 
Likewise, if $\Im p_* = -\Im q_* > 0,$ then $\Im m(q_*) < 0,$ and it is again the second of these identities.

For $\omega < 0,$ we argue by continuity.
For $p_* = -\varrho - i\sqrt{\omega},$ the mapping $\tfrac{1}{\gamma} \mapsto p_*$ is a continuous function for $\gamma \le \frac{2}{r}$ and we have that $p_* = -\varrho + \sqrt{|\omega|}$.  For $\omega \geq 0$ (i.e. $\tfrac{8}{\gamma} \geq (1+r)^2$), the second identity in \eqref{eq:thechoice} holds.  If for some value of $\gamma,$ there were a transition in which identity in \eqref{eq:thechoice} holds, then by continuity there would need to be a value of $\gamma$ so that both hold, which occurs if and only if
\[
\bigl(1- \frac{r\gamma}{2} + \frac{r\gamma p_*}{2\sqrt{r}}m(q_*)\bigr)
=0
=\bigl(1- \frac{r\gamma}{2} + \frac{r\gamma  p_*}{2\sqrt{r}}\frac{1}{m(q_*)}\bigr)
\quad\Longleftrightarrow\quad
m(q_*)^2 = 1
\text{ or } p_* = 0.
\]
As $m(q_*)$ is positive for $\omega \leq 0,$ we must therefore have $m(q_*)=1$ which occurs if and only if $-p_* = (1-\sqrt{r})^2$.  From \eqref{eq:criticalpoint} this does not occur for $p_* = -\varrho - i\sqrt{\omega},$ and in conclusion for this branch of $p_*$ we are always in the second case of \eqref{eq:thechoice}.

On the other hand for $p_*=-\varrho+i\sqrt{\omega},$ when $\gamma < \gamma_*$ we must still be in the second case of \eqref{eq:thechoice}.  By continuity, it suffices to check a single value of $\gamma > \gamma_*$ to determine in which case \eqref{eq:thechoice} we are in.  The most convenient value of $\gamma = \frac{2}{r},$ but at this point, both are $0$ as $p_*$ = 0.  If we parameterize the first equation in terms of $t=\frac{r\gamma}{2},$ then we can write the second case of \eqref{eq:thechoice} as
\[
1-t+\frac{t p_*(t)}{\sqrt{r}m(q_*(t))} = 0.
\]
Differentiating in $t,$ at $t=1,$ and observing $p_*(1)=0,$ we arrive at
\[
-1 + \frac{1}{\sqrt{r}m({q_*}(1))} \frac{dp_*(t)}{dt}\bigg\vert_{t=1}
=
-1 - \frac{\max\{r,1\}}{\sqrt{r}m({q_*}(1))} 
=
-1 - \frac{\max\{r,1\}}{\min\{r,1\}}
\neq 0,
\]
except when $r=1$.  Note that when $r=1,$ no $\omega < 0$ is possible.

We summarize the outcome of this argument as follows
\begin{equation}\label{eq:p_*} 
  \begin{aligned}
    \bigl(1- \frac{r\gamma}{2} + \frac{r\gamma p_*}{2\sqrt{r}}\frac{1}{m(q_*)}\bigr)&=0,
    &\text{ for } p_* = -\varrho-i\sqrt{\omega} \text{ and } \gamma < \frac{2}{r}, \\
    \bigl(1- \frac{r\gamma}{2} + \frac{r\gamma p_*}{2\sqrt{r}}\frac{1}{m(q_*)}\bigr)&=0,
    &\text{ for } p_* = -\varrho+i\sqrt{\omega} \text{ and } \gamma \le \gamma_*, \\
    \bigl(1- \frac{r\gamma }{2} + \frac{r\gamma p_*}{2\sqrt{r}}{m(q_*)}\bigr)&=0,
    &\text{ for } p_* = -\varrho+i\sqrt{\omega} \text{ and } \gamma_* < \gamma < \frac{2}{r}. \\
  \end{aligned}
\end{equation}
Suppose we are in the first two cases of \eqref{eq:p_*}, in particular, we have $1 - \frac{r\gamma}{2} + \frac{r \gamma p_*}{2  \sqrt{r}} \frac{1}{m(q_*)} = 0$, then the following holds
\[
m(q_*) = 
-\bigg(1- \frac{r \gamma }{2}\bigg)^{-1}\frac{r\gamma p_*}{2\sqrt{r}}
.
\]
This then implies that 
\begin{equation} \label{eq:cp_1}
\int_{0}^\infty \frac{x}{x+p_*} \dif \MP(x)
=
-\bigl(1- \frac{r\gamma}{2}\bigr)^{-1}\frac{\gamma p_*}{2}.
\end{equation}
On the other hand, when we are in the second case of \eqref{eq:p_*}, namely $p_* = -\varrho + i \sqrt{\omega}$ and $\gamma_* < \gamma < \frac{2}{r}$, then $m(q_*) = -\frac{2 \sqrt{r}}{r \gamma p_*} \left (1- \frac{r\gamma}{2} \right )$ and consequently,
\begin{equation} \label{eq:cp_2}
\int_{0}^\infty \frac{x}{x+p_*} \dif \MP(x)
=
-\bigg(1- \frac{r\gamma}{2}\bigg) \frac{2 }{r \gamma p^*}.
\end{equation}
The result follows.
\end{proof}

\subsubsection{Noiseless setting.}
With these items in place we can start deriving an expression \eqref{eq:SGD_cp_2} in the noiseless setting, namely when $\widetilde{R} = 0$. In order to do so, we need to compute the Laplace transform of $k(t)$:
\begin{align} 
    K(p) &= \int_0^\infty e^{-pt} \left ( \int_0^\infty \frac{r\gamma}{2} e^{-x t} x^2 \, \dif \MP(x) \right ) \, \dif t = \frac{r\gamma}{2} \int_0^\infty \left( \int_0^\infty  e^{-(x + p)t} \, \dif t \right ) x^2  \, \dif \MP(x) \nonumber \\
    &= \frac{r\gamma}{2} \int_0^\infty \frac{x^2}{x+p} \, \dif \MP(x). \label{eq:SGD_cp_blah_4}
\end{align} 
Using this definition for the Marchenko-Pastur measure \eqref{eq:MP}, we deduce from \eqref{eq:SGD_cp_blah_4} that
\begin{align*}
K(p) &= \frac{\gamma}{4\pi} \int_{\lambda^-}^{\lambda^+} \frac{x+p - p}{x+p} \sqrt{(x-\lambda^-)(\lambda^+-x)} \, \dif x\\
&= \frac{\gamma}{4\pi} \int_{\lambda^-}^{\lambda^+} \sqrt{(x-\lambda^-)(\lambda^+-x)} \, dx - \frac{r \cdot \gamma \cdot p}{2} \int_0^\infty \frac{x}{x+p} \, \dif \MP(x).
\end{align*}
By applying Lemma~\ref{lem:semicirclestieltjes}, we can connect the Stieltjes transform $m$ to $K$. Recalling $ q = \frac{-p-(1+r)}{\sqrt{r}}$, we deduce
\begin{equation} \begin{aligned} \label{eq:Kst}
    K(p) &=  \frac{r\gamma}{2} - \frac{\sqrt{r} \cdot \gamma \cdot p }{2} \cdot m \left ( q \right )
\end{aligned}
\end{equation}
Consequently a simple string computations give the following identity
\begin{equation} \begin{aligned} \label{eq:SGD_cp_5}
    \frac{R \left (1-  \frac{2}{r \gamma} K(p) \right )}{(1-K(p))p} &= R \cdot \frac{ \tfrac{p}{\sqrt{r}} m(q)}{p (1 - \frac{r\gamma }{2} + \frac{r \gamma p}{2 \sqrt{r}} m(q) )} \\
    &=  R \cdot \frac{ \tfrac{1}{\sqrt{r}} m(q)}{1 - \frac{r \gamma}{2} + \frac{r \gamma p}{2 \sqrt{r}} m(q) } \cdot \frac{1 - \tfrac{r \gamma}{2} + \frac{r \gamma p}{ 2 \sqrt{r}} \hat{m}(q)}{1 -\tfrac{r\gamma }{2} + \frac{r \gamma p}{2 \sqrt{r}} \hat{m}(q)} \\
    &=  R \cdot \frac{\tfrac{1}{\sqrt{r}} m(q)  \left (1 -\tfrac{r\gamma}{2} + \frac{r\gamma p}{2  \sqrt{r}} \hat{m}(q) \right )}{ \big ( 1-\tfrac{r\gamma }{2} \big )^2 + \frac{ r \gamma p}{2 \sqrt{r}} \big (1 - \tfrac{r\gamma}{2 } \big ) \left ( \tfrac{p+1+r}{ \sqrt{r}}\right ) + \tfrac{1}{r} \big ( \frac{ r \gamma p}{2} \big )^2 }\\
    &= R \cdot \frac{  \tfrac{1}{\sqrt{r}} m(q)  \left (1 - \tfrac{r\gamma }{2} + \frac{r \gamma p}{2 \sqrt{r}} \hat{m}(q) \right )}{ \big ( 1- \tfrac{r\gamma }{2} \big )^2 + \frac{pr\gamma(1+r)}{2 r} \big (1 - \tfrac{r\gamma}{2} \big )  + \frac{r \gamma p^2}{2 r} }\\
    &= R \cdot 
    \frac{  \tfrac{1}{\sqrt{r}} m(q)  \left (1 - \tfrac{r\gamma}{2}\right ) + \frac{ p\gamma}{2} }
    { \big ( 1- \tfrac{r\gamma}{2} \big )^2 + \frac{p\gamma(1+r)}{2} \big (1 - \tfrac{r\gamma}{2} \big )  + \frac{p^2 \gamma}{2} }.
\end{aligned} \end{equation}

\begin{lemma} \label{lem:magichat} Fix the stepsize $0 < \gamma < \frac{2}{r}$ and set the following constants
\[\gamma_* = \frac{2}{\sqrt{r}(r-\sqrt{r}+1)}, \quad \varrho = \frac{1+r}{2} \left ( 1- \frac{r\gamma}{2} \right ), \quad \text{and} \quad \omega = \frac{1}{4} \left (1-\frac{r\gamma}{2 } \right )^2 \left (\frac{8}{\gamma} - (1+r)^2 \right ).\]
If $\gamma \le \gamma_*$, then the following holds
\[\mathcal{L}^{-1} \left \{ \frac{\left (1-  \frac{2}{r\gamma} K(p) \right )}{(1-K(p))p}\right \} = \frac{2}{\gamma}\bigg (1 - \frac{r\gamma}{2} \bigg )\int_0^\infty \frac{x e^{-xt}}{(x-\varrho)^2 + \omega} \, \dif \MP(x),\]
and in the case that $\gamma_* < \gamma < \frac{2}{r}$, one gets that
\begin{align*}
\mathcal{L}^{-1} \left \{ \frac{\left (1-  \frac{2}{r \gamma} K(p) \right )}{(1-K(p))p}\right \} &=  \frac{2}{\gamma} \bigg (1 - \frac{r\gamma}{2} \bigg )\int_0^\infty \frac{x e^{-xt}}{(x-\varrho)^2 + \omega} \, \dif \MP(x)\\
& \qquad + \frac{2i \sqrt{\omega}}{4 \omega} \cdot \bigg [\varrho - i \sqrt{\omega} - \bigg(\frac{2}{\gamma}\bigg)^2 \left (1- \frac{r\gamma}{2} \right )^2 \frac{\varrho + i \sqrt{\omega}}{r (\varrho^2 + \omega)} \bigg ] e^{(-\varrho + i \sqrt{\omega}) t}.
\end{align*}
\end{lemma}

\begin{proof} We first suppose that $\gamma \le \gamma_*$ and define the function 
\[ y(t) \defas \frac{2}{\gamma} \bigg ( 1- \frac{r\gamma}{2} \bigg ) \int_0^\infty \frac{x e^{-x t} }{(x-\varrho)^2 + \omega} \, \dif \MP(x).\]
Then the Laplace transform of this function is given by the equation
\[ Y(p) \defas \mathcal{L}\{y(t)\}(p) = \frac{2}{\gamma} \bigg (1-\frac{r\gamma}{2} \bigg ) \int_0^\infty \frac{x}{(x+p)((x-\varrho)^2 + \omega)} \, \dif \MP(x). \]
The roots of $(x-\varrho)^2 + \omega$ are precisely $\varrho \pm i \sqrt{\omega}$ so by partial fractions, we have that 
\begin{equation} \begin{aligned} \label{eq:partial_fractions}
    \frac{x}{(x+p)((x-\varrho)^2 + \omega)} &= \frac{1}{(p+ \varrho)^2 + \omega} \cdot \frac{x}{x+p} - \frac{2 i \sqrt{\omega}}{4 \omega (p+\varrho + i \sqrt{\omega})} \cdot \frac{x}{x-\varrho - i \sqrt{\omega}}\\
    &+ \frac{2i \sqrt{\omega} }{4\omega ( p + \varrho- i \sqrt{\omega})} \cdot \frac{x}{x-\varrho + i \sqrt{\omega}}
\end{aligned} \end{equation}
provided that $\sqrt{\omega} \neq 0$. Using Lemmas~\ref{lem:semicirclestieltjes} and \ref{lem:magical}, we have that 
    \begin{gather} 
    \frac{2}{\gamma} \bigg ( 1-\frac{r\gamma}{2} \bigg ) \int_0^\infty \frac{x}{x-\varrho - i \sqrt{\omega}} \, \dif \MP(x) = \varrho + i \sqrt{\omega}, \label{eq:magic_formula_gamma_*} \\
    \frac{2}{\gamma} \bigg ( 1-\frac{r\gamma}{2} \bigg ) \int_0^\infty \frac{x}{x-\varrho + i \sqrt{\omega}} \, \dif \MP(x) = \varrho - i \sqrt{\omega}, \quad \text{and} \quad
    \int_0^\infty \frac{x}{x+p} \, \dif \MP(x) = \frac{m(q)}{\sqrt{r}}, \nonumber
    \end{gather}
where the function $m(q)$ and the point $q = -\frac{p+1+r}{\sqrt{r}}$ are defined in Lemma~\ref{lem:semicirclestieltjes}. A simple calculation using \eqref{eq:magic_formula_gamma_*} shows that 
\begin{align}
    \frac{2}{\gamma} \bigg (1 - \frac{r\gamma}{2} \bigg ) \int_0^{\infty} &\left ( \frac{2i \sqrt{\omega} }{4\omega ( p + \varrho- i \sqrt{\omega})} \cdot \frac{x}{x-\varrho + i \sqrt{\omega}} -\frac{2 i \sqrt{\omega}}{4 \omega (p+\varrho + i \sqrt{\omega})} \cdot \frac{x}{x-\varrho - i \sqrt{\omega}} \right ) \, \dif \MP(x) \nonumber \\
    &= \frac{2i \sqrt{\omega}}{4 \omega ( p + \varrho -i \sqrt{\omega})} \cdot (\varrho - i \sqrt{\omega} ) - \frac{2i \sqrt{\omega}}{4 \omega ( p + \varrho + i\sqrt{\omega})} \cdot (\varrho + i \sqrt{\omega}) \nonumber\\
    &= \frac{p}{(p + \varrho)^2 + \omega}.  \label{eq:magic_1}
\end{align}
Combining the partial fractions decomposition of $Y(s)$ in \eqref{eq:partial_fractions} and \eqref{eq:magic_formula_gamma_*}, we deduce that 
\[ Y(s) = \frac{\tfrac{2}{\gamma} \big (1-\frac{r\gamma}{2} \big ) \frac{m(q)}{\sqrt{r}} + p}{(p+ \varrho)^2 + \omega}.\]
Since both $Y(s)$ and the RHS make sense when $\omega = 0$ by continuity this result holds when $\omega = 0$. After noting that 
\[ \frac{\gamma}{2} \big [ (p + \varrho)^2 + \omega \big ] = \bigg ( 1- \frac{r\gamma}{2} \bigg )^2 + \frac{p\gamma(1+r)}{2} \bigg (1 - \frac{r\gamma}{2} \bigg )  + \frac{\gamma p^2}{2},\]
the result follows for $\gamma \le \gamma_*$ from comparing with \eqref{eq:SGD_cp_5}. 

Next, we consider the setting where $\gamma > \gamma_*$. In this case, we have that $\omega < 0$. Let $A_1$ and $A_2$ be indeterminates and define
  \begin{align*}
  w(t) &\defas \frac{2}{\gamma} \bigg (1 - \frac{r\gamma}{2} \bigg )\int_0^\infty \frac{x e^{-xt}}{(x-\varrho)^2 + \omega} \, \dif \MP(x)
  +A_1 \frac{2i \sqrt{\omega}}{4 \omega} e^{(-\varrho + i\sqrt{\omega})t}+ A_2 \frac{2i \sqrt{\omega}}{4 \omega} e^{(-\varrho - i\sqrt{\omega})t}\\
  & = y(t) + A_1 \frac{2i \sqrt{\omega}}{4 \omega} e^{(-\varrho + i\sqrt{\omega})t}+ A_2 \frac{2i \sqrt{\omega}}{4 \omega} e^{(-\varrho - i\sqrt{\omega})t}.
  \end{align*}
  Particularly the Laplace transform of $w(t)$ is
  \begin{equation} \label{eq:W}
  W(p) \defas \mathcal{L}\{w(t)\} = Y(p) + \frac{2i \sqrt{\omega}}{4 \omega} \frac{A_1}{p + \varrho - i \sqrt{\omega}} + \frac{2i \sqrt{\omega}}{4 \omega}  \frac{A_2}{p + \varrho + i \sqrt{\omega}}.
  \end{equation}
  As in the previous case, we have the partial fraction decomposition in \eqref{eq:partial_fractions} for $Y(p) = \mathcal{L}\{y(t)\}$. However in this case, using Lemmas~\ref{lem:semicirclestieltjes} and \ref{lem:magical}, we get different formulas for \eqref{eq:magic_formula_gamma_*}:
\begin{equation}
    \begin{gathered} \label{eq:magic_formula_gamma_*_1}
      \frac{2}{\gamma} \bigg ( 1-\frac{r\gamma}{2} \bigg ) \int_0^\infty \frac{x}{x-\varrho + i \sqrt{\omega}} \, \dif \MP(x) = \bigg (\frac{2}{\gamma} \bigg)^2 \left (1- \frac{r\gamma}{2} \right )^2 \frac{\varrho + i \sqrt{\omega}}{r (\varrho^2 + \omega)},\\
    \frac{2}{\gamma} \bigg ( 1-\frac{r\gamma}{2} \bigg ) \int_0^\infty \frac{x}{x-\varrho - i \sqrt{\omega}} \, \dif \MP(x) = \varrho + i \sqrt{\omega}, \, \, \text{and} \, \,
    \int_0^\infty \frac{x}{x+p} \, \dif \MP(x) = \frac{m(q)}{\sqrt{r}}.
    \end{gathered}
\end{equation}
With these equations, we have that 
\begin{equation}
\begin{aligned} \label{eq:magic_2}
    \frac{2}{\gamma} \bigg (1 - \frac{r\gamma}{2} \bigg ) \int_0^{\infty} &\bigg ( \frac{2i \sqrt{\omega} }{4\omega ( p + \varrho- i \sqrt{\omega})} \cdot \frac{x}{x-\varrho + i \sqrt{\omega}}  \\
    & \qquad \qquad \qquad - \frac{2 i \sqrt{\omega}}{4 \omega (p+\varrho + i \sqrt{\omega})} \cdot \frac{x}{x-\varrho - i \sqrt{\omega}} \bigg ) \, \dif \MP(x)\\
    &= \frac{2i \sqrt{\omega}}{4 \omega ( p + \varrho -i \sqrt{\omega})} \bigg(\frac{2}{\gamma}\bigg)^2 \left (1- \frac{r\gamma}{2} \right )^2 \frac{\varrho + i \sqrt{\omega}}{r (\varrho^2 + \omega)}\\
    & \qquad \qquad \qquad - \frac{2i \sqrt{\omega}}{4 \omega ( p + \varrho + i\sqrt{\omega})} \cdot (\varrho + i \sqrt{\omega}).
\end{aligned}
\end{equation}
We observe that if $\big(\tfrac{2}{\gamma}\big)^2 \left (1- \frac{r\gamma}{2} \right )^2 \frac{\varrho + i \sqrt{\omega}}{r (\varrho^2 + \omega)} = \varrho-i\sqrt{\omega}$, then by \eqref{eq:magic_1} we would be done. Therefore we can use $A_1$ to make this term equal to $\varrho-i\sqrt{\omega}$. Hence, we should choose $A_1$ such that
\begin{gather*}
    A_1 + \bigg(\frac{2}{\gamma}\bigg)^2 \left (1- \frac{r\gamma}{2} \right )^2 \frac{\varrho + i \sqrt{\omega}}{r (\varrho^2 + \omega)} = \varrho - i \sqrt{\omega}\\
    \Rightarrow \quad A_1 = \varrho - i \sqrt{\omega} - \bigg(\frac{2}{\gamma}\bigg)^2 \left (1- \frac{r\gamma}{2} \right )^2 \frac{\varrho + i \sqrt{\omega}}{r (\varrho^2 + \omega)} \quad \text{and} \quad A_2 = 0. 
\end{gather*}
If we define
\begin{align*} 
Z(p) \defas \frac{2}{\gamma} &\bigg (1 - \frac{r\gamma}{2} \bigg ) \int_0^{\infty} \bigg ( \frac{2i \sqrt{\omega} }{4\omega ( p + \varrho- i \sqrt{\omega})} \cdot \frac{x}{x-\varrho + i \sqrt{\omega}}\\
& \qquad \qquad \qquad \qquad \qquad -\frac{2 i \sqrt{\omega}}{4 \omega (p+\varrho + i \sqrt{\omega})} \cdot \frac{x}{x-\varrho - i \sqrt{\omega}} \bigg ) \, \dif \MP(x),
\end{align*}
then by the construction of $A_1$ and $A_2$, we deduce that
\begin{align*}
    Z(p) + \frac{2i \sqrt{\omega}}{4 \omega} \frac{A_1}{p + \varrho - i \sqrt{\omega}} + \frac{2i \sqrt{\omega}}{4 \omega} \frac{A_1}{p + \varrho + i \sqrt{\omega}} = \frac{p}{(p + \varrho)^2 + \omega}.
\end{align*}
from \eqref{eq:magic_1}. Putting together this with \eqref{eq:magic_formula_gamma_*_1} and the definition of $W(p)$ in \eqref{eq:W}, we get that
\[ W(p) =  \frac{\frac{2}{\gamma} \big (1-\frac{r\gamma}{2} \big ) \frac{m(q)}{\sqrt{r}} + p}{(p+ \varrho)^2 + \omega}.\]
The result then follows. 
\end{proof}

\begin{theorem}[Dynamics of SGD, noiseless setting] Suppose $\widetilde{R} = 0$ and the batchsize satisfies $\beta(n) \leq n^{1/5-\delta}$ for some $\delta >0$, and the stepsize is $0 < \gamma < \frac{2}{r}$. Define the critical stepsize $\gamma_* \in \mathbb{R}$ and constants $\varrho > 0$ and $\omega \in \mathbb{C}$, 
\begin{equation*}
\gamma_* = \frac{2}{\sqrt{r}(r-\sqrt{r}+1)}, \, \, \varrho = \frac{1+r}{2} \left ( 1- \frac{r\gamma}{2} \right ), \, \, \text{and} \, \, \omega = \frac{1}{4} \left (1-\frac{r\gamma}{2} \right )^2 \left (\frac{8}{\gamma} - (1+r)^2 \right ).
\end{equation*}
The iterates of SGD satsify if $\gamma \le \gamma_*$ 
\begin{align*}
f(\xx_{ \lfloor \tfrac{n}{\beta} t \rfloor }) \Prto[n] \quad  &R \cdot \frac{1}{\gamma} \bigg (1 - \frac{r\gamma}{2} \bigg )\int_0^\infty \frac{x e^{-2\gamma xt} }{(x-\varrho)^2 + \omega} \, \dif \MP(x)
\end{align*}
and if $\gamma > \gamma_*$, the iterates of SGD satisfy
\begin{align*} f(\xx_{ \lfloor \tfrac{n}{\beta} t \rfloor }) \Prto[n] \quad &R \cdot \frac{1}{\gamma} \bigg (1 - \frac{r \gamma}{2} \bigg )\int_0^\infty \frac{x e^{-2\gamma xt}}{(x-\varrho)^2 + \omega} \, \dif \MP(x)\\
& \qquad + R \cdot \frac{1}{4 \sqrt{|\omega|}} \cdot \bigg [\varrho +  \sqrt{|\omega|} - \bigg(\frac{2}{\gamma}\bigg)^2 \left (1- \frac{r\gamma}{2} \right )^2 \frac{\varrho - \sqrt{|\omega|}}{r (\varrho^2 - |\omega|)} \bigg ] e^{-2\gamma (\varrho + \sqrt{|\omega|}) t}.
 \end{align*}
 Here the convergence is locally uniformly. 
\end{theorem}

\begin{proof} The result follows immediately from Lemma~\ref{lem:magichat} after noting that when $\gamma_* < \gamma$ we have $\omega < 0$ and that $f(\xx_{ \lfloor \tfrac{n}{\beta} t \rfloor }) \Prto[n] \psi_0(t) = \frac{1}{2}\widehat{\psi}_0(2 \gamma t)$.
\end{proof}

\subsubsection{Noisy term} We now turn to solving the noisy term in \eqref{eq:SGD_cp_2} (\textit{i.e.} the term with $\widetilde{R}$), namely
\begin{align*}
    \widetilde{R} \cdot \frac{r \mathcal{L}\{T(t)\}(p) + \frac{1-r}{p}}{1-K(p)}.
\end{align*}
First we rewrite the Laplace transform of $T$ in terms of the point $q = \tfrac{-p - 1-r}{\sqrt{r}}$. Recall the function $m(z)$ in \eqref{eq:m} as in Lemma~\ref{lem:semicirclestieltjes} by $m(z) = \frac{-z + \sqrt{z^2 - 4}}{2}$ so that the Laplace transform of $T$ becomes
\[ r \mathcal{L}\{T(t)\}(p) = \frac{\sqrt{r}}{p } \cdot \frac{q - \sqrt{q^2 -4}}{2} + \frac{r}{p} = \frac{r}{p}-\frac{\sqrt{r}}{p} m(q). \]
We note again from Lemma~\ref{lem:semicirclestieltjes} that $m(q) \hat{m}(q) = 1$ and $\hat{m}(q) = - q - m(q)$. Using the definition of $K(p)$ from \eqref{eq:Kst}, we get the following equality for the noisy term
\begin{equation}
\begin{aligned} \label{eq:noisy_term_1}
    \widetilde{R} \cdot \frac{r \mathcal{L}\{T(t)\}(p) + \frac{1-r}{p}}{1-K(p)} &= \widetilde{R} \cdot \frac{-\sqrt{r} \cdot m(q) + 1}{p (1-K(p))} \cdot \frac{1 - \frac{r\gamma}{2} + \frac{r \gamma p}{2 \sqrt{r}} \hat{m}(q)}{1 - \frac{r\gamma}{2 } + \frac{r\gamma p}{2 \sqrt{r}} \hat{m}(q) }\\
    &= \widetilde{R} \cdot \frac{ -\sqrt{r} \big (1 - \frac{r\gamma}{2}  \big ) m(q) - \frac{r\gamma p}{2 } + 1- \frac{r \gamma}{2} + \frac{r \gamma p}{2  \sqrt{r}} \hat{m}(q) }{ p \big [ \big ( 1 - \frac{r\gamma}{2} \big )^2 + \frac{rp\gamma(1+r)}{2 r} \big (1- \frac{r\gamma}{2} \big ) + \frac{r \gamma p^2 }{2r} \big ] }\\
    &= \widetilde{R} \cdot \frac{- \big [\frac{2}{\gamma} \sqrt{r} \big ( 1-\frac{r\gamma}{2} \big ) + \sqrt{r}p \big ] m(q) + p^2 + p + \frac{2}{\gamma} \big (1- \frac{r\gamma}{2} \big )}{ p ( (p + \varrho)^2 + \omega )},
\end{aligned} \end{equation}
where $\varrho$ and $\omega$ are defined in \eqref{eq:rho_omega}. With this, we able to conclude derive an expression for the noisy term in $\widehat{\psi}_0(t)$ of \eqref{eq:Old_Volterra_Equation}.

\begin{lemma} \label{lem: noisy_inverse_Laplace} Fix the stepsize $0 < \gamma < \tfrac{2}{r}$ and set the following constants
\[ \gamma_* = \frac{2}{\sqrt{r} (r - \sqrt{r} +1)}, \, \, \varrho = \frac{1+r}{2} \left (1 - \frac{r\gamma}{2} \right ), \, \, \text{and} \quad \omega = \frac{1}{4} \left (1-\frac{r\gamma}{2} \right )^2 \bigg (\frac{8}{\gamma}  - (1+r)^2 \bigg ).\]
If $\gamma \le \gamma_*$, then the following holds
\begin{align*}
    \mathcal{L}^{-1} \left \{ \frac{r \mathcal{L}\{T(t)\} + \frac{1-r}{p}}{1-K(p)} \right \} = \frac{\frac{2}{\gamma}  (1-r) \big (1-\frac{r\gamma}{2} \big )}{\varrho^2 + \omega} + \int_0^\infty \frac{-rx + \frac{2}{\gamma} r \big (1-\frac{r\gamma}{2} \big )}{(x-\varrho)^2 + \omega} e^{-xt} \, \dif \MP(x)
\end{align*}
and in the case that $\gamma_* < \gamma < \frac{2}{r}$, one gets that
\begin{align*}
  \mathcal{L}^{-1} &\left \{ \frac{r \mathcal{L}\{T(t)\} + \frac{1-r}{p}}{1-K(p)} \right \} = \frac{\frac{2}{\gamma}  (1-r) \big (1-\frac{r\gamma}{2} \big )}{\varrho^2 + \omega} + \int_0^\infty \frac{-rx + \frac{2}{\gamma} r \big (1-\frac{r\gamma}{2} \big )}{(x-\varrho)^2 + \omega} e^{-xt} \, \dif \MP(x)\\
  & + \frac{(2i \sqrt{\omega}) r \big [ \frac{2}{\gamma}   \big (1-\frac{r\gamma}{2} \big ) - (\varrho - i \sqrt{\omega}) \big ] }{4 \omega (\varrho-i\sqrt{\omega})  } \cdot 
  \bigg [ \frac{\varrho-i\sqrt{\omega}}{\frac{2}{\gamma} \left ( 1- \frac{r\gamma}{2} \right )} - \frac{\frac{2}{\gamma}(\varrho + i\sqrt{\omega}) \big (1-\frac{r\gamma}{2} \big )}{r (\varrho^2 + \omega)} \bigg ] e^{(-\varrho + i \sqrt{\omega}) t}.
\end{align*}
\end{lemma}

\begin{proof} We first consider the setting where $\gamma \le \gamma_*$ and we define the functions
\[ y(t) \defas \int_0^\infty \frac{-rx + \frac{2r}{\gamma} \big (1-\frac{r\gamma}{2} \big )}{(x-\varrho)^2 + \omega} e^{-xt} \, \dif \MP(x) \quad \text{and} \quad  j(t) \defas \frac{\frac{2}{\gamma} (1-r) \big (1-\frac{r\gamma}{2} \big )}{\varrho^2 + \omega}. \]
Then the Laplace transform of this function is given by the equation
\begin{equation} \label{eq:noisy_stuff_1}
   Y(p) \defas \mathcal{L}\{y(t)\} = \int_0^\infty \frac{-rx + \frac{2r}{\gamma}\big ( 1-\frac{r\gamma}{2} \big )}{x (x+p) [(x-\varrho)^2 + \omega]} \, x \, \dif \MP(x).
\end{equation}
The roots of $(x-\varrho)^2 + \omega$ are precisely $\varrho \pm i \sqrt{\omega}$ so by partial fractions, we have that
\begin{equation} \begin{aligned} \label{eq:noisy_term_2}
   \frac{-rx + \frac{2 r}{\gamma}\big ( 1-\frac{r\gamma}{2} \big )}{x (x+p) [(x-\varrho)^2 + \omega]} \cdot x &= \frac{\frac{2 r}{\gamma} \big (1- \frac{r\gamma}{2} \big )}{p (\varrho^2 + \omega)}  - \frac{rp + \frac{2r}{\gamma} \big (1-\frac{r\gamma}{2} \big )}{\big [ (p + \varrho)^2 + \omega \big ] p} \cdot \frac{x}{x+p}\\
   &+ \frac{ \frac{2 r}{\gamma}\big (1-\frac{r\gamma}{2} \big ) - r ( \varrho + i \sqrt{\omega}) }{ (p + \varrho + i \sqrt{\omega}) (\varrho + i \sqrt{\omega}) (2i \sqrt{\omega}) } \cdot \frac{x}{x - \varrho - i \sqrt{\omega}}\\
   &+ \frac{ \frac{2r}{\gamma} \big (1-\frac{r\gamma}{2} \big ) - r(\varrho - i \sqrt{\omega})}{(p + \varrho -i\sqrt{\omega})(\varrho - i\sqrt{\omega})(-2i \sqrt{\omega}) } \cdot \frac{x}{x - \varrho + i \sqrt{\omega}}.
\end{aligned} \end{equation}
We will consider each term individual. By using Lemma~\ref{lem:semicirclestieltjes}, we deduce that 
\begin{equation} \begin{aligned} \label{eq:m_q_noisy}
\int_0^\infty \frac{rp + \frac{2r}{\gamma} \big (1-\frac{r\gamma}{2 } \big )}{\big [ (p + \varrho)^2 + \omega \big ] p} \cdot \frac{x}{x+p} \, \dif \MP(x) = \frac{m(q)}{\sqrt{r}} \cdot \frac{rp + \frac{2 r}{\gamma} \big (1 - \frac{r\gamma}{2} \big )}{p ( (p + \varrho)^2 + \omega )}.
\end{aligned} \end{equation}
This matches the term in front of the $m(q)$ in \eqref{eq:noisy_term_1}. For the last two terms, we apply Lemma~\ref{lem:magical} and some simple computations to conclude that 
\begin{equation} \begin{aligned} \label{eq:noisy_stuff_2}
    L(p) \defas \int_0^\infty &\frac{ \frac{2 r}{\gamma}\big (1-\frac{r\gamma}{2 } \big ) - r ( \varrho + i \sqrt{\omega}) }{ (p + \varrho + i \sqrt{\omega}) (\varrho + i \sqrt{\omega}) (2i \sqrt{\omega}) } \cdot \frac{x}{x - \varrho - i \sqrt{\omega}} \, \dif \MP(x)
    \\ 
    & \quad \qquad + \int_0^\infty \frac{ \frac{2 r}{\gamma} \big (1-\frac{r\gamma}{2} \big ) - r(\varrho - i \sqrt{\omega})}{(p + \varrho -i\sqrt{\omega})(\varrho - i\sqrt{\omega})(-2i \sqrt{\omega}) } \cdot \frac{x}{x - \varrho + i \sqrt{\omega}} \, \dif \MP(x)\\
    &= \frac{ \frac{2 r}{\gamma}\big (1-\frac{r\gamma}{2} \big ) - r ( \varrho + i \sqrt{\omega}) }{ (p + \varrho + i \sqrt{\omega}) (2i \sqrt{\omega}) \frac{2}{\gamma} \big (1-\frac{r\gamma}{2} \big ) } + \frac{ \frac{2 r}{\gamma} \big (1-\frac{r\gamma}{2} \big ) - r(\varrho - i \sqrt{\omega})}{(p + \varrho -i\sqrt{\omega})(-2i \sqrt{\omega}) \frac{2}{\gamma} \big (1-\frac{r\gamma}{2} \big ) }\\
    &= \frac{-r \gamma( \frac{2}{\gamma} + p -r )}{ 2( (p + \varrho)^2 + \omega)  \big ( 1-\frac{r\gamma}{2} \big ) }.
\end{aligned} \end{equation}
We next observe that $\mathcal{L}\{1\} = \tfrac{1}{p}$. This observation applied to the function $j(t)$ together with the terms not associated with $m(q)$ in \eqref{eq:noisy_term_2} gives the following result
\begin{align*}
    \frac{2 (1-r) \big (1-\frac{r\gamma}{2} \big )}{p \gamma(\varrho^2 + \omega) } + L(p) + \frac{2 r \big (1- \frac{r\gamma}{2} \big )}{p \gamma (\varrho^2 + \omega)} = \frac{p^2 + p + \frac{2}{\gamma} \big (1-\frac{r\gamma}{2} \big )}{p ((p + \varrho)^2 + \omega)}.
\end{align*}
This combined with \eqref{eq:m_q_noisy} and \eqref{eq:noisy_term_1} shows that result holds when $\gamma \le \gamma_*$. 

Next, we consider the setting where $\gamma > \gamma_*$. In this case, we always have that $\omega < 0$. Let $A_1$ be an indeterminate and define 
\[w(t) \defas y(t) + j(t) + \frac{\frac{2r}{\gamma} \left (1-\frac{r\gamma}{2} \right ) - r(\varrho - i \sqrt{\omega})}{(\varrho - i \sqrt{\omega})(-2i \sqrt{\omega}) } A_1 e^{(-\varrho + i\sqrt{\omega})t}. \]
The Laplace transform of $w(t)$ is 
\[W(p) \defas \mathcal{L}\{w(t)\} = Y(p) + \frac{2 (1-r) \big (1-\frac{r\gamma}{2} \big )}{\gamma(\varrho^2 + \omega)} \cdot \frac{1}{p} + \frac{\frac{2r}{\gamma} \left (1-\frac{r\gamma}{2} \right ) - r(\varrho - i \sqrt{\omega})}{(\varrho - i \sqrt{\omega})(-2i \sqrt{\omega}) } A_1 \cdot \frac{1}{p + \varrho - i \sqrt{\omega}}.\]
As in the previous case, we have that \eqref{eq:noisy_stuff_1}, \eqref{eq:noisy_term_2}, and \eqref{eq:m_q_noisy} all still hold. The only difference occurs in the second term in the function $L(p)$ in \eqref{eq:noisy_stuff_2}. In this case using Lemma~\ref{lem:magical}, we get that
\begin{equation} \begin{aligned}
    \widehat{L}(p) \defas \int_0^\infty &\frac{ \frac{2r}{\gamma} \big (1-\frac{r\gamma}{2} \big ) - r ( \varrho + i \sqrt{\omega}) }{ (p + \varrho + i \sqrt{\omega}) (\varrho + i \sqrt{\omega}) (2i \sqrt{\omega}) } \cdot \frac{x}{x - \varrho - i \sqrt{\omega}} \, \dif \MP(x)
    \\ 
    & \quad \qquad + \int_0^\infty \frac{ \frac{2 r}{\gamma} \big (1-\frac{r\gamma}{2} \big ) - r(\varrho - i \sqrt{\omega})}{(p + \varrho -i\sqrt{\omega})(\varrho - i\sqrt{\omega})(-2i \sqrt{\omega}) } \cdot \frac{x}{x - \varrho + i \sqrt{\omega}} \, \dif \MP(x)\\
    &= \frac{ \frac{2r}{\gamma} \big (1-\frac{r\gamma}{2} \big ) - r ( \varrho + i \sqrt{\omega}) }{ (p + \varrho + i \sqrt{\omega}) (2i \sqrt{\omega}) \frac{2}{\gamma} \big (1-\frac{r\gamma}{2} \big ) }\\
    &\qquad \qquad + \frac{ \frac{2r}{\gamma} \big (1-\frac{r\gamma}{2} \big ) - r(\varrho - i \sqrt{\omega})}{(p + \varrho -i\sqrt{\omega}) (\varrho-i \sqrt{\omega} )(-2i \sqrt{\omega}) } \frac{2 (\varrho + i \sqrt{\omega}) \big (1-\frac{r\gamma}{2})}{r \gamma (\varrho^2 + \omega)}.
\end{aligned} \end{equation}
Using the term $A_1$, we are able to make $\widehat{L}(p)$ exactly equal to the RHS of $L(p)$. In particular, we choose the constant $A_1$ as
\begin{equation}
    A_1 = \frac{\varrho-i\sqrt{\omega}}{\frac{2}{\gamma} \big (1-\frac{r\gamma}{2} \big )} - \frac{2(\varrho + i\sqrt{\omega}) \big ( 1-\frac{r\gamma}{2} \big )}{r\gamma(\varrho^2 + \omega)}.
\end{equation}
By this choice of $A_1$, we guarantee that the sum of $\widehat{L}(p)$ and the $A_1$ term equals the RHS of $L(p)$:
\begin{align*}
    \widehat{L}(p) + \frac{\frac{2r}{\gamma} \left (1-\frac{r\gamma}{2} \right ) - r(\varrho - i \sqrt{\omega})}{(\varrho - i \sqrt{\omega})(-2i \sqrt{\omega}) } A_1 \cdot \frac{1}{p + \varrho - i \sqrt{\omega}} =  \frac{-r \gamma ( \frac{2}{\gamma} + p -r )}{ 2( (p + \varrho)^2 + \omega)  \big ( 1-\frac{r\gamma}{2} \big ) }.
\end{align*}
The result then immediately follows from the previous case when $\gamma > \gamma_*$.
\end{proof}
From this lemma, we can now derive the main result which shows that the function values concentrate.

\begin{theorem}[Dynamics of SGD, noisy setting] \label{thm:isotropic_noisy} Suppose the batchsize satisfies $\beta(n) \leq n^{1/5-\delta}$ for some $\delta >0$ and the stepsize is $0< \gamma < \frac{2}{r}$. Define the critical stepsize $\gamma_*$ and constants $\varrho > 0$ and $\omega \in \mathbb{C}$ by
\begin{equation*} \label{eq:important_constants}
\gamma_* = \frac{2}{\sqrt{r}(r-\sqrt{r}+1)}, \, \, \varrho = \frac{1+r}{2} \left ( 1- \frac{r\gamma}{2} \right ), \, \, \text{and} \, \, \omega = \frac{1}{4} \left (1-\frac{r\gamma}{2} \right )^2 \left (\frac{8}{\gamma} - (1+r)^2 \right ).
\end{equation*}
The iterates of SGD satisfy if $\gamma \le \gamma_*$ 
\begin{align*}
f(\xx_{ \lfloor \tfrac{n}{\beta} t \rfloor }) \Prto[d] \quad  &R \cdot \frac{1}{\gamma} \bigg (1 - \frac{r\gamma}{2} \bigg )\int_0^\infty \frac{x e^{-2\gamma xt} }{(x-\varrho)^2 + \omega} \, \dif \MP(x)\\
& \qquad + \widetilde{R} \cdot \frac{1}{2} \cdot \bigg [ \frac{\frac{2}{\gamma} (1-r) \big (1-\frac{r\gamma}{2} \big )}{\varrho^2 + \omega} + \int_0^\infty \frac{-rx + \frac{2r}{\gamma} \big (1-\frac{r\gamma}{2} \big )}{(x-\varrho)^2 + \omega} e^{-2\gamma xt} \, \dif \MP(x) \bigg ].
\end{align*}
and if $\gamma > \gamma_*$, the iterates of SGD satisfy
\begin{align*} &f(\xx_{ \lfloor \tfrac{n}{\beta} t \rfloor }) \Prto[d] \quad R \cdot \frac{1}{\gamma} \bigg (1 - \frac{r\gamma}{2} \bigg )\int_0^\infty \frac{x e^{-2\gamma xt}}{(x-\varrho)^2 + \omega} \, \dif \MP(x)\\
& + R  \cdot \frac{1}{4 \sqrt{|\omega|}} \cdot \bigg [\varrho +  \sqrt{|\omega|} - \frac{4}{\gamma^2} \left (1- \frac{r\gamma}{2} \right )^2 \frac{\varrho - \sqrt{|\omega|}}{r (\varrho^2 - |\omega|)} \bigg ] e^{-2\gamma (\varrho + \sqrt{|\omega|}) t}\\
& + \widetilde{R} \cdot \frac{1}{2} \bigg [ \frac{\frac{2}{\gamma} (1-r) \big (1-\frac{r\gamma}{2} \big )}{\varrho^2 + \omega} + \int_0^\infty \frac{-rx + \frac{2 r}{\gamma} \big (1-\frac{r\gamma}{2} \big )}{(x-\varrho)^2 + \omega} e^{-2\gamma xt} \, \dif \MP(x)\\
  &+ \frac{(-2\sqrt{|\omega|}) r \big [ \frac{2}{\gamma}  \big (1-\frac{r\gamma}{2} \big ) - (\varrho + \sqrt{|\omega|}) \big ] }{4 \omega (\varrho+\sqrt{|\omega|})  } \cdot 
  \bigg ( \frac{\varrho+\sqrt{|\omega|}}{\frac{2}{\gamma} \left ( 1- \frac{r\gamma}{2} \right )} - \frac{\frac{2}{\gamma}  (\varrho - \sqrt{|\omega|})  (1-\frac{r\gamma}{2} )}{r (\varrho^2 + \omega)} \bigg ) e^{-2\gamma (\varrho + \sqrt{|\omega|}) t} \bigg ].
 \end{align*}
 Here the convergence is locally uniformly. 
\end{theorem}

\begin{proof} The result follows immediately from Lemma~\ref{lem: noisy_inverse_Laplace} after noting that when $\gamma_* < \gamma$ we have $\omega < 0$ and that $f(\xx_{ \lfloor \tfrac{n}{\beta} t \rfloor }) \Prto[d] \psi_0(t) = \frac{1}{2}\widehat{\psi}_0(2 \gamma t)$.
\end{proof}

\subsubsection{Computing average-complexity} With our explicit expressions for $\psi_0$, we can now derive the complexity results.

\begin{theorem}[Asymptotic convergence rates isotropic features] Suppose Assumptions~\ref{assumption: Vector} and \ref{assumption: spectral_density} hold with $d \mu(x) = d\MP(x)$. Fix the stepsize $0 < \gamma < \tfrac{2}{r}$ and let the batchsize satisfies $\beta(n) \leq n^{1/5-\delta}$ for some $\delta >0$. Define the constants $\varrho > 0, \omega \in \mathbb{C}$, and critical stepsize $\gamma_* \in \mathbb{R}$ as in \eqref{eq:important_constants}. If the ratio $r = 1$, the iterates of SGD satisfy
\begin{equation} 
\begin{aligned} \label{eq:convex_asymptotic_MP}
\lim_{n \to \infty} f(\xx_{ \lfloor \tfrac{n}{\beta} t \rfloor } ) \overset{\text{\rm Pr}}{\sim} \frac{1}{\gamma^{3/2}} \cdot \left (1-\frac{r\gamma}{2} \right ) \cdot \frac{\sqrt{\lambda^+}}{2 \sqrt{2 \pi}} \cdot \frac{1}{\varrho^2 + \omega} \bigg [ R \cdot \frac{1}{4 \gamma} \cdot \frac{1}{t^{3/2}} + \widetilde{R} \cdot r \cdot \frac{1}{t^{1/2}}  \bigg].
\end{aligned} \end{equation}
If the ratio $r \neq 1$ and $\gamma < \gamma_*$, the iterates of SGD satisfy
\begin{align*}
&\lim_{n \to \infty} f(\xx_{ \lfloor \tfrac{n}{\beta} t \rfloor } ) \, \,  \overset{\rm{\text{Pr}}}{\sim} \, \, 
\widetilde{R} \cdot  \frac{\tfrac{1}{\gamma} \big (1-\frac{r\gamma}{2} \big )}{(\varrho^2 + \omega)} \max \{0, 1-r\}\\
&+ \frac{1}{8 \sqrt{2\pi} r} \cdot \frac{(\lambda^+-\lambda^-)^{1/2}}{ \gamma^{3/2} \big [ (\lambda^--\varrho)^2 + \omega \big ] } \bigg [R \cdot \frac{1}{\gamma} \bigg (1 - \frac{r\gamma}{2} \bigg ) +  \frac{\widetilde{R}}{2} \left (  \frac{2r}{\gamma \lambda^-} \bigg (1-\frac{r\gamma}{2} \bigg ) - r \right ) \bigg ]\cdot e^{-2\gamma \lambda^- t} \cdot \frac{1}{t^{3/2}}.
\end{align*}
If the ratio $r \neq 1$ and $\gamma = \gamma_*$, the iterates of SGD satisfy
\begin{align*}
&\lim_{n \to \infty} f(\xx_{ \lfloor \tfrac{n}{\beta} t \rfloor } )  \, \, \overset{\rm \text{Pr}}{\sim} \, \, \widetilde{R} \cdot \frac{\tfrac{1}{\gamma} \big (1-\frac{r\gamma}{2} \big )}{(\varrho^2 + \omega)} \max \{0, 1-r\}\\
&+ \frac{1}{2 \sqrt{2\pi} r} \cdot \frac{1}{\gamma^{1/2} r_2 (\lambda^+-\lambda^-)^{1/2}} \bigg [ R \cdot \frac{1}{\gamma} \bigg (1 - \frac{r\gamma}{2} \bigg ) +  \widetilde{R} \cdot \bigg ( \frac{2r}{\gamma \lambda^-} \bigg (1-\frac{r\gamma}{2} \bigg ) - r \bigg ) \bigg ] \cdot e^{-2 \gamma \lambda^- t} \cdot \frac{1}{t^{1/2}}.
\end{align*}
If the ratio $r \neq 1$ and $\gamma > \gamma_*$, the iterates of SGD satisfy
\begin{align*} 
&\lim_{n \to \infty} f(\xx_{ \lfloor \tfrac{n}{\beta} t \rfloor } ) \, \, \overset{\rm \text{Pr}}{\sim} \, \, \widetilde{R} \cdot \frac{\tfrac{1}{\gamma} \big (1-\frac{r\gamma}{2} \big )}{(\varrho^2 + \omega)} \max \{0, 1-r\}\\
  & + R \cdot \frac{1}{4 \sqrt{|\omega|}} \cdot \bigg [\varrho +  \sqrt{|\omega|} - \frac{4}{\gamma^2} \left (1- \frac{r\gamma}{2} \right )^2 \frac{\varrho - \sqrt{|\omega|}}{r (\varrho^2 - |\omega|)} \bigg ] e^{-2\gamma (\varrho + \sqrt{|\omega|}) t}\\
  & + \widetilde{R} \cdot \frac{(-2\sqrt{|\omega|}) r \big [ \frac{2}{\gamma}  \big (1-\frac{r\gamma}{2} \big ) - (\varrho + \sqrt{|\omega|}) \big ] }{8 \omega (\varrho+\sqrt{|\omega|})  } \cdot 
  \bigg ( \frac{\varrho+\sqrt{|\omega|}}{\frac{2}{\gamma}  \left ( 1- \frac{r\gamma}{2} \right )} - \frac{\frac{2}{\gamma} (\varrho - \sqrt{|\omega|})  (1-\frac{r\gamma}{2} )}{r (\varrho^2 + \omega)} \bigg ) e^{-2 \gamma (\varrho + \sqrt{|\omega|}) t}.
  \end{align*}
  Here $t=1$ corresponds to computing $n$ stochastic gradients, the convergence is locally uniformly, and the notation $\overset{\text{Pr}}{\sim}$ means that you take the limit in probability and then compute the asymptotic.
\end{theorem}

\begin{proof} Throughout the proof we use $2\gamma t \mapsto t$ and we define $\varrho$ and $\omega$ as in Theorem~\ref{thm:isotropic_noisy}. Suppose we have that $r = 1$. Because $1-\frac{\gamma}{2} > 0$, we know that $\gamma < \gamma_*$. By construction of $\gamma_*$ in \eqref{eq:criticalpoint}, we have that 
\[ (1+r)^2 - 8\gamma^{-1}  < (1+r)^2 - 8\gamma_* = (1- \sqrt{r})^4 = 0.\]
In particular, this means that $\omega > 0$ and the roots of $(x-\varrho)^2 + \omega$ are precisely $x = \varrho \pm \sqrt{-\omega}$. As $\omega > 0$, these roots are complex with positive imaginary part and thus, $(x-\varrho)^2 + \omega \neq 0$ for any $x \in [0, \lambda^+]$. So there exists a constant $C > 0$ such that $(x-\varrho)^2 + \omega > C$ for all $x \in [0, \lambda^+]$ and consequently, 
\begin{equation} \label{eq:avg_value_1}
    \int_0^\infty \frac{x}{(x-\varrho)^2 + \omega} \, \dif \MP(x) \quad \text{is bounded.}
\end{equation}
We begin by computing $\int_0^\infty \frac{x e^{-xt}
}{(x-\varrho)^2 + \omega} \, \dif \MP$. If $x \ge \log^2(t)/t$, then it is clear that $e^{-tx}$ decays faster than any polynomial in $t$. Combining this with \eqref{eq:avg_value_1}, we get that 
\[ \int_{\log^2(t)/t}^\infty \frac{xe^{-xt}}{(x-\varrho)^2 + \omega} \, \dif \MP(x) \quad \text{decays faster than any polynomial in $t$.} \]
Now we suppose that $0 \le x < \log^2(t)/t$. Using a simple change of variables, we deduce that
\begin{align*}
    \int_0^{\log^2(t)/t} \frac{xe^{-tx}}{(x-\varrho)^2 + \omega} \, \dif \MP(x) = \frac{1}{2\pi} \cdot \frac{1}{t^{3/2}}  \int_0^{\min\{\log^2(t), t\lambda^+\}} \frac{\sqrt{u} \cdot e^{-u}}{ (\tfrac{1}{t} u - \varrho)^2 + \omega} \sqrt{\lambda^+- \tfrac{u}{t}} \, \dif u.
\end{align*}
We have that $\sqrt{\lambda^+- \tfrac{u}{t}} \le \sqrt{\lambda^+}$ and $0< \omega \le (x-\varrho)^2 + \omega$ so dominated convergence theorem holds
\begin{align*} \lim_{t \to \infty} \int_0^{\min\{ \log^2(t), t\lambda^+\}} \frac{\sqrt{u} \cdot e^{-u}}{ (\tfrac{1}{t} u - \varrho)^2 + \omega} \sqrt{\lambda^+- \tfrac{u}{t}} \, \dif u = \int_0^{\infty} \frac{\sqrt{u} \cdot e^{-u}}{ \varrho^2 + \omega} \sqrt{\lambda^+} \, \dif u = \frac{\sqrt{\lambda^+ \pi}}{2 (\varrho^2 + \omega)}.
\end{align*}
Consequently, we have that 
\begin{equation} \label{eq:r_1_first}
\int_0^{\infty} \frac{xe^{-tx}}{(x-\varrho)^2 + \omega} \, \dif \MP(x) \sim \frac{1}{4 \sqrt{\pi}} \cdot \frac{\sqrt{\lambda^+}}{\varrho^2 + \omega} \cdot \frac{1}{t^{3/2}}.
\end{equation}
Now we turn to $\int_0^\infty \frac{e^{-xt}}{(x-\varrho)^2 + \omega} \, \dif \MP(x)$. As before, we know that 
\[\int_{\log^2(t)/t}^\infty \frac{e^{-xt}}{(x-\varrho)^2 + \omega} \, \dif \MP(x) \quad \text{decays faster than any polynomial in $t$.}\] 
Again using a simple change of variables, we deduce that 
\begin{align*}
    \int_0^{\log^2(t)/t} \frac{e^{-tx}}{(x-\varrho)^2 + \omega} \, \dif \MP(x) &= \frac{1}{2\pi} \cdot \frac{1}{t^{1/2}} \int_0^{\min\{ \log^2(t), t \lambda^+\}} \frac{e^{-u} \sqrt{\lambda^+- \tfrac{u}{t}}}{\sqrt{u} \big ( (\tfrac{u}{t}-\varrho)^2 + \omega \big )} \, \dif u.
\end{align*}
By using dominated convergence theorem, we know that
\[ \lim_{t\to \infty} \int_0^{\min\{ \log^2(t), \lambda^+ t \}} \frac{e^{-u} \sqrt{\lambda^+-\tfrac{u}{t}}}{\sqrt{u} \big ( (\tfrac{u}{t}-\varrho)^2 + \omega \big )} \, \dif u = \int_0^\infty \frac{e^{-u} \sqrt{\lambda^+}}{\sqrt{u} (\varrho^2 + \omega)} \, \dif u = \frac{\sqrt{\pi \lambda^+}}{\varrho^2 + \omega}.\]
Consequently, we have that 
\begin{equation} \label{eq:r_1_second}
    \int_0^\infty \frac{e^{-xt}}{(x-\varrho)^2 + \omega} \, \dif \MP(x) \sim \frac{1}{2 \sqrt{\pi}} \cdot \frac{\sqrt{\lambda^+}}{\varrho^2 + \omega} \cdot \frac{1}{t^{1/2}}.
\end{equation}
After noting that $t = 2 \gamma t$ and Theorem~\ref{thm:isotropic_noisy}, the result immediately follows for the case when $r =1$ in \eqref{eq:convex_asymptotic_MP}.

Next we consider the setting where $r \neq 1$. Suppose $\ell \in \{0,1\}$. By a simple change of variables $x = \lambda^- + (\lambda^+-\lambda^-)u$, we have
\begin{align} 
    &
    \frac{1}{2\pi r} \int_{\lambda^-}^{\lambda^+} \frac{e^{-xt}}{x^{1-\ell} ( (x-\varrho)^2 + \omega )} \sqrt{(x-\lambda^-)(\lambda^+-x)} \, \dif x \label{eq:bound_integral_1} \\
    &= \frac{1}{2 \pi r} \left ( \int_0^{\log^2(t)/t} + \int_{\log^2(t)/t}^1 \right )  \frac{ (\lambda^+-\lambda^-)^2 e^{-t\lambda^-} e^{-(\lambda^+-\lambda^-)ut} \sqrt{u(1-u)} }{(\lambda^- + (\lambda^+-\lambda^-) u)^{1-\ell} \big [ (\lambda^- + (\lambda^+-\lambda^-) u - \varrho)^2 + \omega \big ]} \, \dif u. \nonumber
\end{align} 
Let's first consider where $u \ge \log^2(t)/t$. As $r \neq 1$, we have that $\lambda^- > 0$ and therefore, $(\lambda^- + (\lambda^+-\lambda^-) u)^{1-\ell} \ge (\lambda^-)^{1-\ell}$. If $\omega > 0$, then $0 < \omega < (\lambda^-+(\lambda^+-\lambda^-)u - \varrho)^2 + \omega$ so that 
\begin{equation}\label{eq:bound_integral} \int_{\log^2(t)/t}^1 \frac{1}{(\lambda^- + (\lambda^+-\lambda^-)u)^{1-\ell} \big [ (\lambda^- + (\lambda^+-\lambda^-) u - \varrho)^2 + \omega\big ] } \sqrt{u(1-u)} \, \dif u \le C, 
\end{equation}
or equivalently this integral is bounded by some $C > 0$. Now we suppose that $\omega \le 0$. The roots of $(\lambda^- + (\lambda^+-\lambda^-)u - \varrho)^2 + \omega$ are precisely given by $u = \frac{\varrho -\lambda^- \pm \sqrt{-\omega}}{\lambda^+-\lambda^-}$. By \eqref{eq:criticalpoint} if $\gamma \neq \gamma_*$, we have that $\varrho + \sqrt{-\omega} < (1-\sqrt{r})^2 = \lambda^-$. Hence there exists a constant $\widehat{C} > 0$ such that $\widehat{C} < (\lambda^- + (\lambda^+-\lambda^-)u - \varrho)^2 +\omega$ for all $u \in [0,1]$. It immediately follows that \eqref{eq:bound_integral} holds. Now suppose that $\gamma = \gamma_*$. Then $\varrho + \sqrt{|\omega|} = (1-\sqrt{r})^2 = \lambda^-$ so the polynomial $(\lambda^- + (\lambda^+-\lambda^-)u - \varrho)^2 + \omega$ is $0$ when $u = 0$. We observe that $u = 0$ is not a double root since there does not exist an $r$ with $\gamma = \gamma_*, \omega = 0,$ and $\gamma_* < \tfrac{2}{r}$ (here $\gamma = \gamma_*$ and $\omega =0$ imply that $r = 1$, but then $\gamma_* = 2$ which violates $\gamma_* < \tfrac{2}{r}$). Consequently, we can write 
\[ (\lambda^- + (\lambda^+-\lambda^-)u - \varrho)^2 + \omega = (\lambda^+-\lambda^-)^2u(u+r_2), \]
where $r_2 > 0$ as the second root is negative. Since the Marchenko-Pastur measure has $\sqrt{u}$-behavior near $0$, it immediately follows that
\begin{align*}
    \int_{\log^2(t)/t}^1 &\frac{1}{(\lambda^- + (\lambda^+-\lambda^-)u)^{1-\ell} \big [ (\lambda^- + (\lambda^+-\lambda^-) u - \varrho)^2 + \omega \big ] } \sqrt{u(1-u)} \, \dif u\\
    &\le \frac{1}{r_2(\lambda^-)^{\ell-1}(\lambda^+-\lambda^-)^2} \int_0^1 \frac{\sqrt{1-u}}{\sqrt{u}} \, \dif u < C.
\end{align*}
Hence in all cases we have that 
\begin{align*}
    \frac{(\lambda^+-\lambda^-)^2}{2 \pi r} \int_{\log^2(t)/t}^1& \frac{e^{-(\lambda^+-\lambda^-)ut}}{\big (\lambda^- + (\lambda^+-\lambda^-)u)^{1-\ell}  \big [(\lambda^- + (\lambda^+-\lambda^-)u - \varrho)^2 + \omega\big ]} \sqrt{u(1-u)} \, \dif u \\
    &\le e^{-(\lambda^+-\lambda^-)\log^2(t)} C
\end{align*}
where $C > 0$ is some constant. Since $e^{-(\lambda^+-\lambda^-) \log^2(t)}$ decays faster than polynomial, this part of the integral does not have the interesting asymptotic. Now returning to \eqref{eq:bound_integral_1} the interesting asymptotic occurs near $u = 0$. First we consider the setting where $\gamma \neq \gamma_*$. As we saw for $u \in [\log^2(t)/t, 1]$, we also know that there exists a constant $\widehat{C} > 0$ such that
\[ \widehat{C} < (\lambda^- + (\lambda^+-\lambda^-)u)^{1-\ell}  \big [ (\lambda^- + (\lambda^+-\lambda^-) u - \varrho)^2 + \omega \big ] \quad \text{for all $u \in [0, 1]$.}  \]
Using a change of variables $u = v/t$, we deduce that
\begin{align*}
    \int_0^{\log^2(t)/t} &\frac{e^{-(\lambda^+-\lambda^-) ut}}{ (\lambda^- + (\lambda^+-\lambda^-)u)^{1-\ell} \big [ (\lambda^- + (\lambda^+-\lambda^-) u - \varrho)^2 + \omega \big ] } \sqrt{u(1-u)} \, \dif u 
    \\ 
    &\le \frac{1}{\widehat{C}} \cdot \frac{1}{t^{3/2}} \int_0^{\log^2(t)} e^{-(\lambda^+-\lambda^-) v} \sqrt{v (1- \tfrac{v}{t})} \, \dif v\\
    &\le \frac{1}{\widehat{C}} \cdot \frac{1}{t^{3/2}} \int_0^\infty e^{-(\lambda^+-\lambda^-)v} \sqrt{v} \, \dif v
\end{align*}
Since the last integral is bounded, we can apply dominated convergence theorem. Using the change of variables, $u = \tfrac{v}{t}$, we have that
\begin{equation} \begin{aligned} \label{eq:asymptotic_1}
    &\int_0^{\log^2(t)/t} \frac{e^{-(\lambda^+-\lambda^-) ut}}{(\lambda^- + (\lambda^+-\lambda^-)u)^{1-\ell} \big [ (\lambda^- + (\lambda^+-\lambda^-) u - \varrho)^2 + \omega \big ]} \sqrt{u(1-u)} \, \dif u\\
    &= \frac{1}{t^{3/2}} \cdot \int_0^{\log^2(t)} \frac{e^{-(\lambda^+-\lambda^-) v}}{(\lambda^- + (\lambda^+-\lambda^-)\tfrac{v}{t})^{1-\ell}  \big [ (\lambda^- + (\lambda^+-\lambda^-) \tfrac{v}{t} - \varrho)^2 + \omega \big ] } \sqrt{v(1-\tfrac{v}{t})} \, \dif v\\
    &\sim \frac{1}{t^{3/2}} \int_0^\infty \frac{e^{-(\lambda^+-\lambda^-)v}}{(\lambda^{-})^{1-\ell} \big [(\lambda^- - \varrho)^2 + \omega\big ]} \sqrt{v} \, \dif v\\
    &\sim \frac{1}{t^{3/2}} \cdot \frac{\sqrt{\pi}}{2 (\lambda^+-\lambda^-)^{3/2} (\lambda^-)^{1-\ell} ((\lambda^- - \varrho)^2 + \omega)}.
\end{aligned} \end{equation}
Consequently, we get for any $\ell \in \{0,1\}$
\begin{align*}
\frac{1}{2\pi r} \int_{\lambda^-}^{\lambda^+} \frac{e^{-xt}}{{x^{1-\ell}\big ( (x-\varrho)^2 + \omega \big )}} &\sqrt{(x-\lambda^-)(\lambda^+-x)} \, \dif x\\
&\sim \frac{1}{4 \sqrt{\pi} r} \cdot \frac{(\lambda^+-\lambda^-)^{1/2}}{(\lambda^-)^{1-\ell} \big [ (\lambda^--\varrho)^2 + \omega \big ] } \cdot e^{-\lambda^- t} \cdot \frac{1}{t^{3/2}}. 
\end{align*}
Now consider the setting where $\gamma = \gamma_*$. As we saw for $u \in [\log^2(t)/t, 1]$, we know that the polynomial has a root at $u = 0$ (not a double root). In particular, we get that
\begin{align*} (\lambda^- + (\lambda^+-\lambda^-)u)^{1-\ell} & \big [ (\lambda^- + (\lambda^+-\lambda^-)u - \varrho)^2 + \omega \big ]\\
&= (\lambda^+-\lambda^-)^2 (\lambda^- + (\lambda^+-\lambda^-)u)^{1-\ell}  u (u + r_2), 
\end{align*}
where $-r_2$ is the second root of the quadratic $(\lambda^- + (\lambda^+-\lambda^-)u -\varrho)^2 + \omega$. We know this root is negative (i.e. $r_2 > 0$). Using a change of variables $u = v/t$ and a simple lower bound on $(\lambda^- + (\lambda^+-\lambda^-)u)^{1-\ell}$, we get that 
\begin{align*}
    \int_0^{\log^2(t)/t} &\frac{e^{-(\lambda^+-\lambda^-) ut}}{ (\lambda^- + (\lambda^+-\lambda^-)u)^{1-\ell} \big [ (\lambda^- + (\lambda^+-\lambda^-) u - \varrho)^2 + \omega \big ] } \sqrt{u(1-u)} \, \dif u 
    \\ 
    &\le \frac{1}{(\lambda^-)^{1-\ell} (\lambda^+-\lambda^-)^2 r_2} \cdot \frac{1}{t^{1/2}} \int_0^{\log^2(t)} \frac{e^{-(\lambda^+-\lambda^-) v}}{v} \sqrt{v (1- \tfrac{v}{t})} \, \dif v\\
    &\le \frac{1}{(\lambda^-)^{1-\ell} (\lambda^+-\lambda^-)^2 r_2} \cdot \frac{1}{t^{1/2}} \int_0^\infty \frac{e^{-(\lambda^+-\lambda^-)v}}{\sqrt{v}} \, \dif v
    \end{align*}
Since the last integral is bounded, we can apply dominated convergence theorem. Using the change of variables, $u = \tfrac{v}{t}$, we have that 
\begin{equation} \begin{aligned} \label{eq:asymptotic_2}
    &\int_0^{\log^2(t)/t} \frac{e^{-(\lambda^+-\lambda^-) ut}}{(\lambda^- + (\lambda^+-\lambda^-)u)^{1-\ell} \big [ (\lambda^+-\lambda^-)^2u(u+r_2) \big ]} \sqrt{u(1-u)} \, \dif u\\
    &= \frac{1}{t^{1/2}} \cdot \int_0^{\log^2(t)} \frac{e^{-(\lambda^+-\lambda^-) v}}{(\lambda^- + (\lambda^+-\lambda^-)\tfrac{v}{t})^{1-\ell}  \big [ (\lambda^+-\lambda^-)^2 v(\tfrac{v}{t}+r_2) \big ] } \sqrt{v(1-\tfrac{v}{t})} \,\dif v\\
    &\sim \frac{1}{t^{1/2}} \int_0^\infty \frac{e^{-(\lambda^+-\lambda^-)v}}{(\lambda^{-})^{1-\ell} \big [ (\lambda^+-\lambda^-)^2 r_2 \big ]} \cdot \frac{1}{\sqrt{v}} \, \dif v\\
    &\sim \frac{1}{t^{1/2}} \cdot \frac{\sqrt{\pi}}{ (\lambda^+-\lambda^-)^{5/2} (\lambda^-)^{1-\ell} r_2}.
\end{aligned} \end{equation}

Consequently, we have when $\gamma = \gamma_*$ that

\begin{align*}
    \frac{1}{2\pi r} \int_{\lambda^-}^{\lambda^+} \frac{e^{-xt}}{{x^{1-\ell}\big ( (x-\varrho)^2 + \omega \big )}} &\sqrt{(x-\lambda^-)(\lambda^+-x)} \, \dif x \\
    &\sim \frac{1}{2 \sqrt{\pi} r} \cdot \frac{1}{r_2 (\lambda^+-\lambda^-)^{1/2} (\lambda^-)^{1-\ell}} \cdot e^{-\lambda^- t} \cdot \frac{1}{t^{1/2}}.
\end{align*}
We note that the Dirac delta from the Marchenko-Pastur terms combined with the constant term yields that
\begin{align*}
    \frac{\tfrac{2}{\gamma} (1-r) \big ( 1- \frac{r \gamma}{2} \big )}{\varrho^2 + \omega} + \frac{\tfrac{2r}{\gamma} (1-r) \big ( 1- \frac{r \gamma}{2} \big )}{\varrho^2 + \omega} \max \big \{0, 1-\tfrac{1}{r} \big \} = \frac{\tfrac{2}{\gamma} \big (1-\frac{r\gamma}{2} \big )}{\varrho^2 + \omega} \max \{0, 1-r\}.
\end{align*}
The result immediately follows. 
\end{proof}

\section{Numerical simulation details and extra experiments} \label{apx:exp_details}

\paragraph{Problem setup.}
The vectors $\xx_0$, $\widetilde{\xx}$ and are sampled i.i.d. from the Gaussian $N({\boldsymbol{0}}, \tfrac{1}{d}\II)$.
The objective function in which we run SGD is in all cases the least squares objective function $f(\xx) = {\tfrac{1}{2n} \|\AA \xx -\bb \|^2}$, where $\bb$ is generated as in \eqref{eq:lsq}. The definition of $\AA$ is different depending on the data-generating process considered:
\begin{itemize}
    \item For the isotropic features model, the rows of $\AA$ are generated i.i.d. from a standard Gaussian distribution.
    \item In the one-hidden layer model these are generated following \eqref{eq:random_features_model}, with both $\WW$ and and $\YY$ are i.i.d. from a standard Gaussian and $n/m= X$, $m/d = Y$, and $g$ is a shifted hinge loss:
    \begin{equation}
        g(z) = \max(x, 0) - \frac{1}{\sqrt{2 \pi}}\,.
    \end{equation}
    The substraction of $- \frac{1}{\sqrt{2 \pi}}$ ensures that the zero Gaussian mean assumption is verified \eqref{eq: Gaussian_mean}.
\end{itemize}

Throughout the experiments $r$ is fixed to $1.5$ and $\tilde{R}=0$. We experimented with different values, and always obtained very similar qualitative results.

\paragraph{Algorithms.}
We simulate the \textbf{SME and SDE models} (see \eqref{eq:SDE} for description) using Euler-Murayama discretization with stepsize $10^{-3}$. The SDE model discretizes the same as equation as the SME model (Eq. \eqref{eq:SDE}) with a scaled identity covariance ($\boldsymbol{\Sigma} = \sigma^2 \boldsymbol{I}$). In this model $\sigma^2$ is a free parameter which for the experiments we set to 0.1, as this value was giving the closest fit to SGD across the log-space grid of parameters $10^{-i}, i=0, 1, ...$.

For the \textbf{streaming model}, we use SGD updates and regenerate $\aa_i$ (following the same model as SGD) at every step.

\paragraph{Extra experiments.} We provide extra experiments following the same setting as in Figure \ref{fig:comparision_SDE_volterra} but with different choices of the ratio $r = \frac{d}{n}$ parameter.

\begin{figure}[ht]
    \centering
    \includegraphics[width=\linewidth]{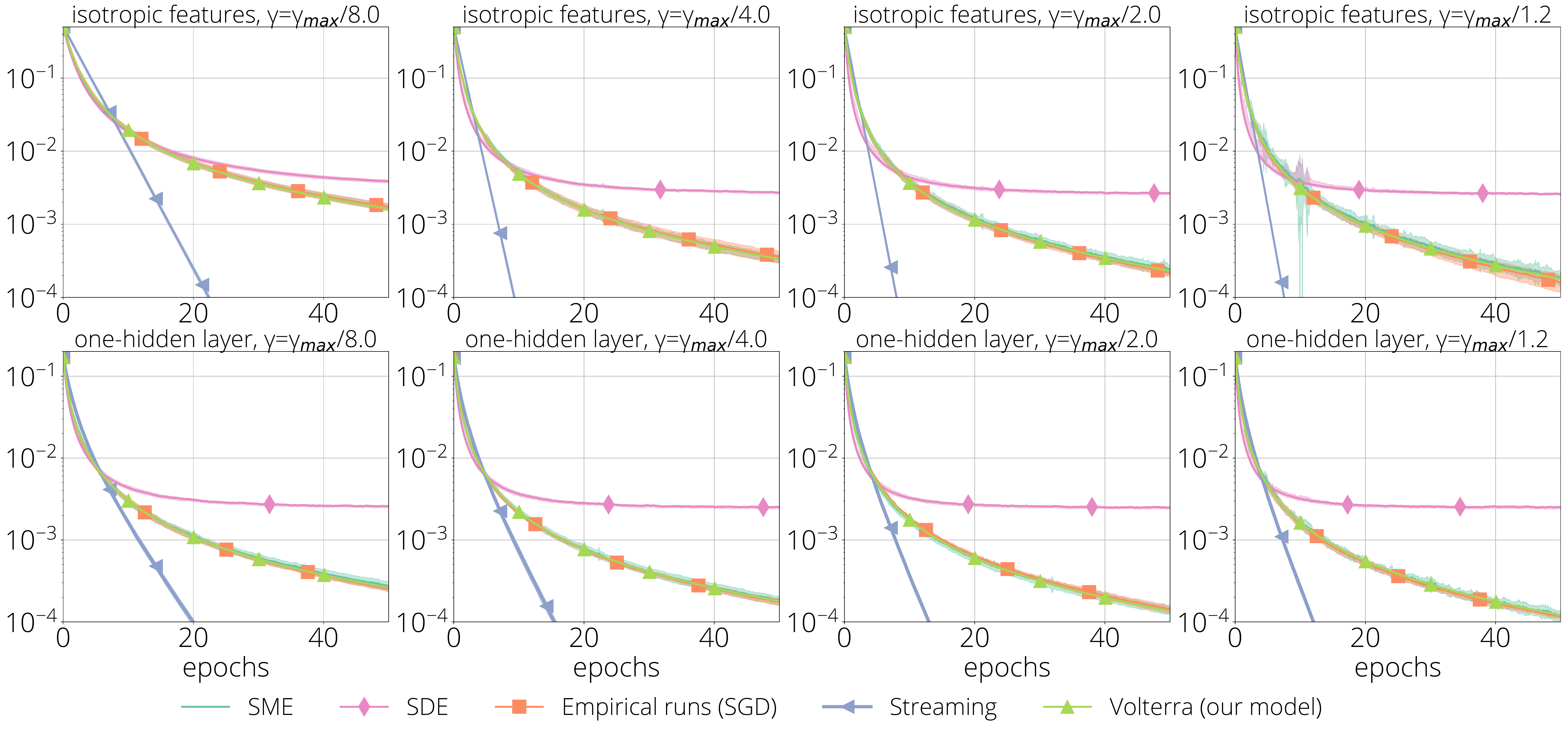}
    \includegraphics[width=\linewidth]{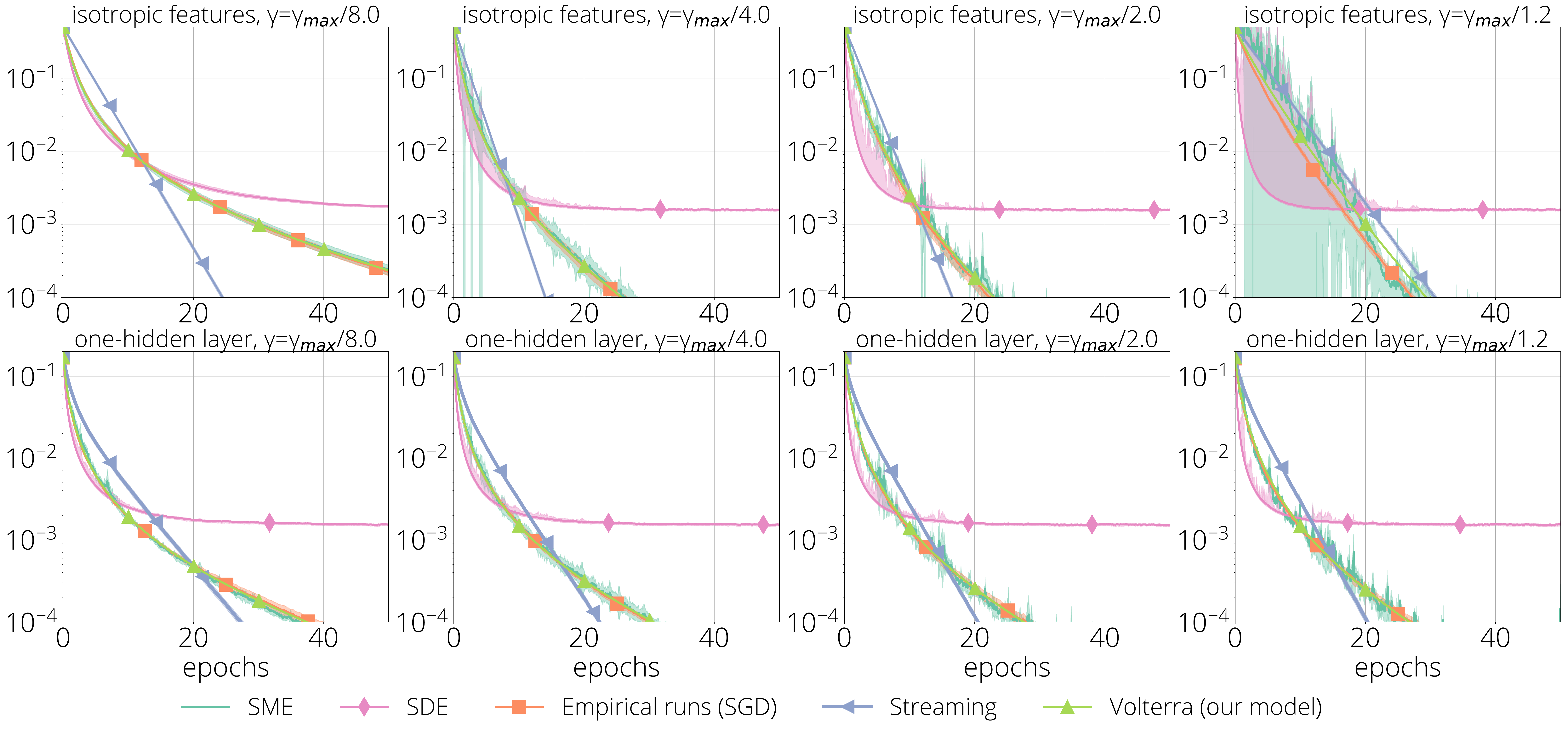}
    \vspace{-0.5cm}
    \caption{{\bfseries Comparison of different SGD models with $r=0.8$ and $r=1.6$}: isotropic features (top) and one-hidden layer network (bottom). 
}
    \label{fig:comparision_SDE_volterra2}
\end{figure}

\end{document}